\documentclass{amsart}
\usepackage{amssymb,amsmath,stmaryrd,mathrsfs}
\def\definetac{\newif\iftac}    
\ifx\tactrue\undefined
  \definetac
  \ifx\state\undefined\tacfalse\else\tactrue\fi
\fi
\iftac\else\usepackage{amsthm}\fi
\usepackage[all,2cell]{xy}
\UseAllTwocells
\usepackage{enumitem}
\usepackage{color}
\definecolor{darkgreen}{rgb}{0,0.45,0} 
\usepackage[pagebackref,colorlinks,citecolor=darkgreen,linkcolor=darkgreen]{hyperref}
\usepackage{mathtools}          
\usepackage{braket}             

\usepackage{url}                
\usepackage{xspace}             

\makeatletter
\let\ea\expandafter

\def\mdef#1#2{\ea\ea\ea\gdef\ea\ea\noexpand#1\ea{\ea\ensuremath\ea{#2}\xspace}}
\def\alwaysmath#1{\ea\ea\ea\global\ea\ea\ea\let\ea\ea\csname your@#1\endcsname\csname #1\endcsname
  \ea\def\csname #1\endcsname{\ensuremath{\csname your@#1\endcsname}\xspace}}

\DeclareRobustCommand\widecheck[1]{{\mathpalette\@widecheck{#1}}}
\def\@widecheck#1#2{%
    \setbox\z@\hbox{\m@th$#1#2$}%
    \setbox\tw@\hbox{\m@th$#1%
       \widehat{%
          \vrule\@width\z@\@height\ht\z@
          \vrule\@height\z@\@width\wd\z@}$}%
    \dp\tw@-\ht\z@
    \@tempdima\ht\z@ \advance\@tempdima2\ht\tw@ \divide\@tempdima\thr@@
    \setbox\tw@\hbox{%
       \raise\@tempdima\hbox{\scalebox{1}[-1]{\lower\@tempdima\box
\tw@}}}%
    {\ooalign{\box\tw@ \cr \box\z@}}}

\newcount\foreachcount

\def\foreachletter#1#2#3{\foreachcount=#1
  \ea\loop\ea\ea\ea#3\@alph\foreachcount
  \advance\foreachcount by 1
  \ifnum\foreachcount<#2\repeat}

\def\foreachLetter#1#2#3{\foreachcount=#1
  \ea\loop\ea\ea\ea#3\@Alph\foreachcount
  \advance\foreachcount by 1
  \ifnum\foreachcount<#2\repeat}

\def\definescr#1{\ea\gdef\csname s#1\endcsname{\ensuremath{\mathscr{#1}}\xspace}}
\foreachLetter{1}{27}{\definescr}
\def\definecal#1{\ea\gdef\csname c#1\endcsname{\ensuremath{\mathcal{#1}}\xspace}}
\foreachLetter{1}{27}{\definecal}
\def\definebold#1{\ea\gdef\csname b#1\endcsname{\ensuremath{\mathbf{#1}}\xspace}}
\foreachLetter{1}{27}{\definebold}
\def\definebb#1{\ea\gdef\csname l#1\endcsname{\ensuremath{\mathbb{#1}}\xspace}}
\foreachLetter{1}{27}{\definebb}
\def\definefrak#1{\ea\gdef\csname f#1\endcsname{\ensuremath{\mathfrak{#1}}\xspace}}
\foreachletter{1}{9}{\definefrak} 
\foreachletter{10}{27}{\definefrak}
\foreachLetter{1}{27}{\definefrak}
\def\definebar#1{\ea\gdef\csname #1bar\endcsname{\ensuremath{\overline{#1}}\xspace}}
\foreachLetter{1}{27}{\definebar}
\foreachletter{1}{8}{\definebar} 
\foreachletter{9}{15}{\definebar} 
\foreachletter{16}{27}{\definebar}
\def\definetil#1{\ea\gdef\csname #1til\endcsname{\ensuremath{\widetilde{#1}}\xspace}}
\foreachLetter{1}{27}{\definetil}
\foreachletter{1}{27}{\definetil}
\def\definehat#1{\ea\gdef\csname #1hat\endcsname{\ensuremath{\widehat{#1}}\xspace}}
\foreachLetter{1}{27}{\definehat}
\foreachletter{1}{27}{\definehat}
\def\definechk#1{\ea\gdef\csname #1chk\endcsname{\ensuremath{\widecheck{#1}}\xspace}}
\foreachLetter{1}{27}{\definechk}
\foreachletter{1}{27}{\definechk}
\def\defineul#1{\ea\gdef\csname u#1\endcsname{\ensuremath{\underline{#1}}\xspace}}
\foreachLetter{1}{27}{\defineul}
\foreachletter{1}{27}{\defineul}

\def\autofmt@n#1\autofmt@end{\mathrm{#1}}
\def\autofmt@b#1\autofmt@end{\mathbf{#1}}
\def\autofmt@l#1#2\autofmt@end{\mathbb{#1}\mathsf{#2}}
\def\autofmt@c#1#2\autofmt@end{\mathcal{#1}\mathit{#2}}
\def\autofmt@s#1#2\autofmt@end{\mathscr{#1}\mathit{#2}}
\def\autofmt@f#1\autofmt@end{\mathfrak{#1}}
\def\autofmt@u#1\autofmt@end{\underline{\smash{\mathsf{#1}}}}
\def\autofmt@U#1\autofmt@end{\underline{\underline{\smash{\mathsf{#1}}}}}
\def\autofmt@h#1\autofmt@end{\widehat{#1}}
\def\autofmt@r#1\autofmt@end{\overline{#1}}
\def\autofmt@t#1\autofmt@end{\widetilde{#1}}
\def\autofmt@k#1\autofmt@end{\check{#1}}

\def\auto@drop#1{}
\def\autodef#1{\ea\ea\ea\@autodef\ea\ea\ea#1\ea\auto@drop\string#1\autodef@end}
\def\@autodef#1#2#3\autodef@end{%
  \ea\def\ea#1\ea{\ea\ensuremath\ea{\csname autofmt@#2\endcsname#3\autofmt@end}\xspace}}
\def\autodefs@end{blarg!}
\def\autodefs#1{\@autodefs#1\autodefs@end}
\def\@autodefs#1{\ifx#1\autodefs@end%
  \def\autodefs@next{}%
  \else%
  \def\autodefs@next{\autodef#1\@autodefs}%
  \fi\autodefs@next}

\DeclareSymbolFont{bbold}{U}{bbold}{m}{n}
\DeclareSymbolFontAlphabet{\mathbbb}{bbold}
\newcommand{\bbDelta}{\ensuremath{\mathbbb{\Delta}}\xspace}
\newcommand{\bbone}{\ensuremath{\mathbbb{1}}\xspace}
\newcommand{\bbtwo}{\ensuremath{\mathbbb{2}}\xspace}

\newcommand{\phibar}{\ensuremath{\overline{\varphi}}\xspace}

\mdef\delbar{\overline{\partial}}

\mdef\hf{\textstyle\frac12 }
\mdef\thrd{\textstyle\frac13 }
\mdef\qtr{\textstyle\frac14 }

\newcommand{\op}{^{\mathrm{op}}}

\newcommand{\coop}{^{\mathrm{coop}}}

\SelectTips{cm}{}
\newdir{ >}{{}*!/-10pt/@{>}}    

\newcommand{\pullbackcorner}[1][dr]{\save*!/#1-1.2pc/#1:(-1,1)@^{|-}\restore}

\mdef\Id{\mathrm{Id}}
\mdef\id{\mathrm{id}}
\alwaysmath{ell}
\alwaysmath{infty}
\alwaysmath{odot}
\def\frc#1/#2.{\frac{#1}{#2}}   
\mdef\ten{\mathrel{\otimes}}

\mdef\sqten{\mathrel{\boxtimes}}

\DeclareRobustCommand\widecheck[1]{{\mathpalette\@widecheck{#1}}}
\def\@widecheck#1#2{%
    \setbox\z@\hbox{\m@th$#1#2$}%
    \setbox\tw@\hbox{\m@th$#1%
       \widehat{%
          \vrule\@width\z@\@height\ht\z@
          \vrule\@height\z@\@width\wd\z@}$}%
    \dp\tw@-\ht\z@
    \@tempdima\ht\z@ \advance\@tempdima2\ht\tw@ \divide\@tempdima\thr@@
    \setbox\tw@\hbox{%
       \raise\@tempdima\hbox{\scalebox{1}[-1]{\lower\@tempdima\box
\tw@}}}%
    {\ooalign{\box\tw@ \cr \box\z@}}}

\DeclareMathOperator\colim{colim}

\DeclareMathOperator\sk{sk}

\DeclareMathOperator\End{End}

\newcommand{\too}[1][]{\ensuremath{\overset{#1}{\longrightarrow}}}
\newcommand{\ot}{\ensuremath{\leftarrow}}

\let\toto\rightrightarrows
\let\into\hookrightarrow

\mdef\we{\overset{\sim}{\longrightarrow}}
\mdef\leftwe{\overset{\sim}{\longleftarrow}}

\let\maps\colon

\newcommand{\fib}{\mathsf{fib}}
\newcommand{\cof}{\mathsf{cof}}

\let\xto\xrightarrow

\def\rightarrowtailfill@{\arrowfill@{\Yright\joinrel\relbar}\relbar\rightarrow}
\newcommand\xrightarrowtail[2][]{\ext@arrow 0055{\rightarrowtailfill@}{#1}{#2}}

\def\twoheadrightarrowfill@{\arrowfill@{\relbar\joinrel\relbar}\relbar\twoheadrightarrow}
\newcommand\xtwoheadrightarrow[2][]{\ext@arrow 0055{\twoheadrightarrowfill@}{#1}{#2}}

\def\slashedarrowfill@#1#2#3#4#5{%
  $\m@th\thickmuskip0mu\medmuskip\thickmuskip\thinmuskip\thickmuskip
   \relax#5#1\mkern-7mu%
   \cleaders\hbox{$#5\mkern-2mu#2\mkern-2mu$}\hfill
   \mathclap{#3}\mathclap{#2}%
   \cleaders\hbox{$#5\mkern-2mu#2\mkern-2mu$}\hfill
   \mkern-7mu#4$%
}
\def\rightslashedarrowfill@{%
  \slashedarrowfill@\relbar\relbar\mapstochar\rightarrow}
\newcommand\xslashedrightarrow[2][]{%
  \ext@arrow 0055{\rightslashedarrowfill@}{#1}{#2}}
\mdef\hto{\xslashedrightarrow{}}
\mdef\htoo{\xslashedrightarrow{\quad}}

\def\toiso{\xto{\smash{\raisebox{-.5mm}{$\scriptstyle\sim$}}}}
\def\otiso{\xleftarrow{\smash{\raisebox{-.5mm}{$\scriptstyle\sim$}}}}

\def\shvar#1#2{{\ensuremath{%
  \hspace{1mm}\makebox[-1mm]{$#1\langle$}\makebox[0mm]{$#1\langle$}\hspace{1mm}%
  {#2}%
  \makebox[1mm]{$#1\rangle$}\makebox[0mm]{$#1\rangle$}%
}}}
\def\sh{\shvar{}}

\def\Bigsh{\shvar{\Big}}

\long\def\my@drawfill#1#2;{%
\@skipfalse
\fill[#1,draw=none] #2;
\@skiptrue
\draw[#1,fill=none] #2;
}
\newif\if@skip
\newcommand{\skipit}[1]{\if@skip\else#1\fi}
\newcommand{\drawfill}[1][]{\my@drawfill{#1}}

\newif\ifhyperref
\@ifpackageloaded{hyperref}{\hyperreftrue}{\hyperreffalse}
\iftac
  \let\your@state\state
  \def\state#1{\gdef\currthmtype{#1}\your@state{#1}}
  \let\your@staterm\staterm
  \def\staterm#1{\gdef\currthmtype{#1}\your@staterm{#1}}
  \let\defthm\newtheorem
  \def\currthmtype{}
  \ifhyperref
    \def\autoref#1{\ref*{label@name@#1}~\ref{#1}}
  \else
    \def\autoref#1{\ref{label@name@#1}~\ref{#1}}
  \fi
  \AtBeginDocument{%
    \let\old@label\label%
    \def\label#1{%
      {\let\your@currentlabel\@currentlabel%
        \edef\@currentlabel{\currthmtype}%
        \old@label{label@name@#1}}%
      \old@label{#1}}
  }
\else
  \ifhyperref
    \def\defthm#1#2{%
      \newtheorem{#1}{#2}[section]%
      \expandafter\def\csname #1autorefname\endcsname{#2}%
      \expandafter\let\csname c@#1\endcsname\c@thm}
  \else
    \def\defthm#1#2{\newtheorem{#1}[thm]{#2}}
    \ifx\SK@label\undefined\let\SK@label\label\fi
    \let\old@label\label
    \let\your@thm\@thm
    \def\@thm#1#2#3{\gdef\currthmtype{#3}\your@thm{#1}{#2}{#3}}
    \def\currthmtype{}
    \def\label#1{{\let\your@currentlabel\@currentlabel\def\@currentlabel%
        {\currthmtype~\your@currentlabel}%
        \SK@label{#1@}}\old@label{#1}}
    \def\autoref#1{\ref{#1@}}
  \fi
\fi

\newtheorem{thm}{Theorem}[section]

\defthm{cor}{Corollary}
\defthm{prop}{Proposition}
\defthm{lem}{Lemma}
\defthm{sch}{Scholium}
\defthm{assume}{Assumption}
\defthm{claim}{Claim}
\defthm{conj}{Conjecture}
\defthm{hyp}{Hypothesis}
\defthm{fact}{Fact}
\iftac\theoremstyle{plain}\else\theoremstyle{definition}\fi
\defthm{defn}{Definition}
\defthm{notn}{Notation}
\iftac\theoremstyle{plain}\else\theoremstyle{remark}\fi
\defthm{rmk}{Remark}
\defthm{eg}{Example}
\defthm{egs}{Examples}
\defthm{ex}{Exercise}
\defthm{ceg}{Counterexample}

\def\thmqedhere{\expandafter\csname\csname @currenvir\endcsname @qed\endcsname}

\setitemize[1]{leftmargin=2em}
\setenumerate[1]{leftmargin=*}

\iftac
  \let\c@equation\c@subsection
\else
  \let\c@equation\c@thm
\fi
\numberwithin{equation}{section}

\@ifpackageloaded{mathtools}{\mathtoolsset{showonlyrefs,showmanualtags}}{}

\alwaysmath{alpha}
\alwaysmath{beta}
\alwaysmath{gamma}
\alwaysmath{Gamma}
\alwaysmath{delta}
\alwaysmath{Delta}
\alwaysmath{epsilon}
\mdef\ep{\varepsilon}
\alwaysmath{zeta}
\alwaysmath{eta}
\alwaysmath{theta}
\alwaysmath{Theta}
\alwaysmath{iota}
\alwaysmath{kappa}
\alwaysmath{lambda}
\alwaysmath{Lambda}
\alwaysmath{mu}
\alwaysmath{nu}
\alwaysmath{xi}
\alwaysmath{pi}
\alwaysmath{rho}
\alwaysmath{sigma}
\alwaysmath{Sigma}
\alwaysmath{tau}
\alwaysmath{upsilon}
\alwaysmath{Upsilon}
\alwaysmath{phi}
\alwaysmath{Pi}
\alwaysmath{Phi}
\mdef\ph{\varphi}
\alwaysmath{chi}
\alwaysmath{psi}
\alwaysmath{Psi}
\alwaysmath{omega}
\alwaysmath{Omega}

\makeatother

\usepackage{tikz}
\usepackage{mathdots}
\usepackage{wasysym}

\newcommand{\tr}{\ensuremath{\operatorname{tr}}}
\newcommand{\tw}{\ensuremath{\operatorname{tw}}}
\newcommand{\dual}[1]{D{#1}}
\newcommand{\rdual}[1]{D_r{#1}}

\autodefs{\bProf\cProf\bSet\nAut\bCat\lMod\lMat\lProf}
\let\D\sD
\let\E\sE

\newcommand{\bg}{\ensuremath{{\bB G}}\xspace}
\newcommand{\bh}{\ensuremath{{\bB H}}\xspace}
\newcommand{\bth}{\ensuremath{{\bB \theta}}\xspace}
\newcommand{\ob}{\operatorname{ob}}

\newcommand{\tc}{\bbone}   
\newcommand{\pt}{\star} 

\newcommand{\rup}[1]{\overrightarrow{#1}}
\newcommand{\rdn}[1]{\overleftarrow{#1}}

\newcommand{\bsl}[2]{\partial(#1/#2)}

\newcommand{\card}[1]{\# #1}

\let\prof\bProf
\let\V\bV
\let\W\bW
\let\T\bT

\let\dV\sV
\let\dW\sW
\let\dT\sT
\let\dS\sS
\let\dprof\cProf

\let\tens\varoast

\newcommand{\symm}{\mathfrak{s}}
\renewcommand{\th}{^{\textrm{th}}}
\autodefs{\cDER\cCat\cCAT\cMONCAT\bSet\cProf\bCat\cPro\cMONDER\cBICAT\ncomp\nid\sProf\ncyl\fC\fZ\fT\nhocolim\nholim\cPsNat\ncoll\bTop\cRMODEL\cEx\cSp\cGpd}
\def\ho{\mathscr{H}\!\mathit{o}\xspace}

\def\shift#1#2{{#1}^{#2}}

\newcommand{\Ex}{\cEx} 

\newcommand{\loops}[1]{{#1}^{\circlearrowleft}}

\theoremstyle{plain}
\makeatletter
\newtheorem*{rep@theorem}{\rep@title}
\newcommand{\newreptheorem}[2]{%
\newenvironment{rep#1}[1]{%
 \def\rep@title{Restatement of #2 \ref{##1}}%
 \begin{rep@theorem}}%
 {\end{rep@theorem}}}
\makeatother

\newreptheorem{thm}{Theorem}
\newreptheorem{lem}{Lemma}

\title{The linearity of traces in monoidal categories and bicategories}
\author{Kate Ponto and Michael Shulman}
\date{\today}
\thanks{The first author  was partially 
supported by NSF grant DMS-1207670. The second author was partially supported by an NSF postdoctoral fellowship and 
NSF grant DMS-1128155, and appreciates the hospitality of the University of Kentucky.
Any opinions, findings, and conclusions or recommendations expressed in this material are those of the authors and 
do not necessarily reflect the views of the National Science Foundation.}
\begin{document}
\maketitle

\begin{abstract}
  We show that in any symmetric monoidal category, if a weight for colimits is absolute, then the resulting colimit of any diagram of dualizable objects is again dualizable.
  Moreover, in this case, if an endomorphism of the colimit is induced by an endomorphism of the diagram, then its trace can be calculated as a linear combination of traces on the objects in the diagram.
  The formal nature of this result makes it easy to generalize to traces in homotopical contexts (using derivators) and traces in bicategories.
  These generalizations include the familiar additivity of the Euler characteristic and Lefschetz number along cofiber sequences, as well as an analogous result for the Reidemeister trace, but also the orbit-counting theorem for sets with a group action, and a general formula for homotopy colimits over EI-categories.
\end{abstract}

\tableofcontents

\section{Introduction}
\label{sec:introduction}

In this paper, we study the following question: given a diagram in a category, when can the ``size'' of its colimit be calculated in terms of the ``size'' of the objects occurring in the diagram?
Such a question might pertain to various notions of ``size'', such as cardinality, dimension, or Euler characteristic.
Here are a few well-known facts that can be interpreted as answers to instances of this question.

\begin{enumerate}
\item If $X$ and $Y$ are finite sets, then we have an obvious formula for the cardinality of their disjoint union:\label{item:eg-coprod}
  \[ \card{(X\sqcup Y)} = \card X + \card Y. \]
\item More generally, for finite CW-complexes $X$ and $Y$, the Euler characteristic of their disjoint union is the sum of their Euler characteristics:\label{item:eg-coprod-sp}
  \[ \chi(X\sqcup Y) = \chi(X) + \chi(Y). \]
\item Similarly, if $X$ and $Y$ are finite-dimensional vector spaces, we have an analogous formula for the dimension of their sum:\label{item:eg-coprod-vect}
  \[ \mathrm{dim}(X\oplus Y) = \mathrm{dim}(X) + \mathrm{dim}(Y). \]
\item If $X\into Y$ and $X\into Z$ are injections of finite sets, then we have the ``inclusion-exclusion'' formula for the cardinality of their pushout:\label{item:eg-incl-excl}
  \[ \card{(Y+_{X} Z)} = \card Y + \card Z - \card X. \]
\item More generally, if $Y \ot X\to Z$ is an arbitrary span of finite CW complexes, then there is a similar formula for the Euler characteristic of their \emph{homotopy} pushout:\label{item:eg-hopo}
  \[ \chi(Y+_X^h Z) = \chi(Y) + \chi(Z) - \chi(X). \]
\item As a particular case of~\ref{item:eg-hopo}, if $Z=\pt$ is the one-point space and $X\to Y$ is the inclusion of a subcomplex, then the homotopy pushout is homotopy equivalent to the quotient $Y/X$, and we have\label{item:eg-cofiber}
  \[ \chi(Y/X) = \chi(Y) - \chi(X). \]
\item If $X$ is a finite-dimensional chain complex and $\mathrm{dim}(X) = \sum_n (-1)^n \mathrm{dim}(X_n)$ is its \emph{graded dimension}, then there is an obvious formula for the graded dimension of its suspension:\label{item:eg-susp-ch}
  \[ \mathrm{dim}(\Sigma X) = - \mathrm{dim}(X). \]
\item Similarly, if $X$ is a finite CW complex, then we have an analogous formula for the Euler characteristic of its suspension:\label{item:eg-susp-sp}
  \[ \chi(\Sigma X) = - \chi(X). \]
\item If $G$ is a finite group and $X$ a finite $G$-set, then we have the orbit-counting theorem (a.k.a.\ Burnside's lemma or the Cauchy-Frobenius lemma) for the cardinality of its quotient:\label{item:eg-quot}
  \[ \card {(X/G)} = \frac{1}{\card G} \sum_{g\in G} \card{(X^g)}. \]
  Here $X^g$ is the set of fixed points of $g\in G$ acting on $X$.
\item If $e\colon X\to X$ is an idempotent linear operator on a finite-dimensional vector space (i.e.\ a projection), then the dimension of its quotient is equal to its trace:\label{item:eg-idem}
  \[ \mathrm{dim}(X/e) = \tr(e). \]
\item The cardinality (or Euler characteristic) of the empty set is zero:\label{item:eg-empty}
  \[ \card\emptyset = \chi(\emptyset) = 0. \]
  as is the dimension of the zero vector space:
  \[ \mathrm{dim}(0) = 0. \]
\end{enumerate}
In all cases, the formulas have a common shape: the size of a colimit is expressed as a \emph{linear combination} of the sizes of its inputs (or other related trace-like invariants).
The first general theory of such formulas was described by Leinster~\cite{leinster:eccat}: he showed that if $A$ is a finite category that admits a \emph{weighting}, which is a function $k\colon \mathrm{ob}(A) \to \mathbb{Q}$ satisfying certain properties, then the formula
\[ \card{\colim(X)} = \sum_a k_a \cdot \card {X_a} \]
holds whenever $X\colon A\to\mathrm{Set}$ is a finite coproduct of representables.
This includes examples~\ref{item:eg-coprod}, \ref{item:eg-incl-excl}, the special case of~\ref{item:eg-quot} when the action is free, and a similar special case of~\ref{item:eg-idem}.
However, it applies only to finite sets, thus excluding the algebraic or homotopical examples; nor does it deal with the case of non-free actions.

Our original motivation to study this question came from a generalization of~\ref{item:eg-cofiber} to a statement about \emph{Lefschetz numbers}.
In fact, all of the above formulas can be similarly generalized to become statements about a trace-like invariant of an endomorphism, which reduce to the previous statements in the case of identity maps.
Specifically:
\begin{itemize}
\item For an endomorphism $f\colon X\to X$ of a finite set, we can consider the cardinality $\card{\mathrm{Fix}(f)}$ of the set of fixed points of $f$.
  When $f=\id_X$ this reduces to $\card X$.
\item For an endomorphism $f\colon X\to X$ of a finite-dimensional vector space, we can consider its \emph{trace} in the usual sense.
  When $f=\id_X$ this reduces to $\mathrm{dim}(X)$.
\item For an endomorphism $f\colon X\to X$ of a finite-dimensional manifold, we can consider its \emph{Lefschetz number}.
  When $f=\id_X$ this reduces to $\chi(X)$.
\end{itemize}
All of the above formulas remain true if cardinalities, dimensions, and Euler characteristics are replaced by fixed-point counts, traces, and Lefschetz numbers.
More specifically, given a \emph{natural} endomorphism of a diagram, there is an induced endomorphism of its colimit, and we have formulas calculating trace-like invariants of the latter in terms of the corresponding trace-like invariants of the objects in the diagram.
For example:
\begin{itemize}
\item If $X\into Y$ and $X\into Z$ are injections of finite sets and we have endofunctions $f\colon Y\to Y$ and $g\colon Z\to Z$ which agree when restricted to $X$, then there is an induced endofunction $h\colon Y+_X Z \to Y+_X Z$, and we have 
\[ \card{\mathrm{Fix}(h)} = \card{\mathrm{Fix}(f)} + \card{\mathrm{Fix}(g)} - \card{\mathrm{Fix}(f|_X)}. \]
\item If $e$ is an idempotent linear operator on a finite-dimensional vector space $X$, and $f\colon X\to X$ is any linear operator, then there is an induced operator $g\colon X/e \to X/e$ and we have
  \[ \tr(g) = \tr(e \circ f). \]
\item If $X\into Y$ is an inclusion of finite CW complexes, and $f\colon Y\to Y$ is an endomorphism such that $f(X) \subseteq X$, then there is an induced endomorphism $f/X$ of the quotient $Y/X$, and we have
  \begin{equation}
    L(f/X) = L(f) - L(f|_X). \label{eq:intro-add2}
  \end{equation}
  where $L$ denotes the Lefschetz number.
\end{itemize}
Eq.~\eqref{eq:intro-add2} is better known when written in the following way:
\begin{equation}
  L(f)=L(f|_X)+L(f/X)\label{eq:intro-add}
\end{equation}
In this form it is known as the \emph{additivity} of the Lefschetz number.

In~\cite{add}, May gave a very general proof of~\eqref{eq:intro-add}, using the fact that the Lefschetz number is an instance of an abstract notion of \emph{trace} that can be defined for an endomorphism of a \emph{dualizable object} in any symmetric monoidal category.
All of the above ``size-like'' and ``trace-like'' invariants can be put into this framework, sometimes by first mapping them into another category.
Namely:
\begin{itemize}
\item A vector space is dualizable just when it is finite-dimensional, and in that case the categorical trace of an endomorphism is precisely the classical trace.
\item If a space $X$ is a finite CW complex, then its \emph{suspension spectrum} is dualizable in the stable homotopy category, and in that case the categorical trace of an endomorphism is precisely the Lefschetz number.
\item A finite set can either be regarded as a finite CW complex and mapped into the stable homotopy category, or else regarded as the basis of a  vector space.
  In either case, the resulting categorical trace gives precisely the number of fixed points of an endofunction.
\end{itemize}
Certain properties of this abstract categorical trace are well-known and easy to prove.
For instance, if the monoidal category is semi-additive (i.e.\ finite products and coproducts coincide naturally), then the trace is additive on direct sums; this implies~\ref{item:eg-coprod}, \ref{item:eg-coprod-sp}, and~\ref{item:eg-coprod-vect}, and a nullary version of it implies~\ref{item:eg-empty}.
The trace is also \emph{cyclic}; this fairly easily implies~\ref{item:eg-idem}.

May showed an analogous, but more complicated, general result: if the symmetric monoidal category is \emph{triangulated} in a way compatible with its monoidal structure, then the categorical trace is additive along distinguished triangles, in the sense of~\eqref{eq:intro-add}.
This implies~\ref{item:eg-cofiber} and~\ref{item:eg-hopo}, and thereby~\ref{item:eg-incl-excl}.
(It is also fairly easy to see that May's axioms for compatibility between a triangulation and a monoidal structure imply~\ref{item:eg-susp-ch} and~\ref{item:eg-susp-sp}.)

Our original motivation was a desire to extend May's result to an additivity theorem for the \emph{Reidemeister trace}, a fixed-point invariant that refines the Lefschetz number.
Unlike the Lefschetz number, the Reidemeister trace is not a categorical trace in a symmetric monoidal category, but it is an instance of a more general kind of abstract trace that takes place in a \emph{bicategory}~\cite{kate:traces,PS2}.
We found that the most natural way to do this was to set up a general theory that applies to colimits of potentially arbitrary shapes, and indeed potentially arbitrary \emph{weights}, which turns out to include \emph{all} the above examples.

Recall that in enriched category theory, we consider not just ordinary colimits but \emph{weighted} colimits: if $A$ is a small category describing the shape of our diagram, then the {\bf weight} is a functor $\Phi\colon A\op \to \V$.
(Ordinary ``unweighted'' colimits are the special case when $\Phi$ is constant at the unit object.)
Such a weight $\Phi$ is said to be {\bf absolute} if $\Phi$-weighted colimits are preserved by every \V-functor; for instance, finite coproducts are absolute in vector spaces.
The simplest case of our general theorem is then:

\begin{thm}\label{thm:intro-smc}
  Let \V be a closed, cocomplete, semi-additive, symmetric monoidal category.
  If $A$ is a finite category, $\Phi\colon A\op\to \V$ is absolute, and $X\colon A\to \V$ is a diagram such that each $X_a$ is dualizable, then the weighted colimit $\colim^\Phi(X)$ is also dualizable, and we have a formula for its formal Euler characteristic (the trace of its identity map):
  \[ \chi(\colim^{\Phi}(X)) = \sum_{[\alpha]} \phi_{[\alpha]} \tr(X_\alpha) \]
  More generally, for any endo-natural-transformation $f\colon X\to X$ of such an $X$, we have a similar formula for its trace:
  \[ \tr(\colim^{\Phi}(f)) = \sum_{[\alpha]} \phi_{[\alpha]} \tr(X_\alpha\circ f_a). \]
\end{thm}

We call this theorem a \emph{linearity formula}, because it expresses the trace associated to the colimit as a \emph{linear combination} of traces associated to the input diagram.
The sum is indexed by ``conjugacy classes'' of endomorphisms $\alpha \colon  a\to a$ in the category $A$ (we will define these later).
In most of the above examples, the only endomorphisms are identities, so it reduces to a sum over objects of $A$.
In particular, \autoref{thm:intro-smc} has the following specializations.
\begin{itemize}
\item If \V is pointed with zero object $0$, then $0$ is dualizable and $\chi(0) = 0$, giving example~\ref{item:eg-empty}.
\item If \V is semi-additive and $X$ and $Y$ are dualizable, then $\chi(X\oplus Y) = \chi(X) + \chi(Y)$.
  This implies examples~\ref{item:eg-coprod}, \ref{item:eg-coprod-sp}, and~\ref{item:eg-coprod-vect}.
\item In any \V, if $X$ is dualizable and $e:X\to X$ is idempotent, then $\chi(X/e) = \tr(e)$, giving example~\ref{item:eg-idem}.
  Here $e$ itself serves as the only relevant ``conjugacy class''.
\item If \V is semi-additive, and $X$ is dualizable with an action of a finite group $G$ whose cardinality is invertible in \V (e.g.\ if $\V$ is rational vector spaces), then
  \begin{equation}
    \chi(X/G) = \frac{1}{\card G} \sum_{g\in G} \tr(X(g)).
  \end{equation}
  This implies example~\ref{item:eg-quot}.
  Here the ``conjugacy classes'' are ordinary conjugacy classes in $G$.
\end{itemize}
All of these apply also to traces of nonidentity morphisms.

However, \autoref{thm:intro-smc} does not apply as stated to the homotopical examples, including~\ref{item:eg-cofiber} and the motivating case~\eqref{eq:intro-add2}, since homotopy colimits are not particular weighted colimits.%
\footnote{They can be calculated in examples \emph{using} certain weighted colimits, but the relevant weights are not absolute, and the ``dualizability'' is also only up to homotopy.}
We need a version of it that applies to a ``natively homotopical'' context, and for this we find it most convenient to use \emph{derivators}.
A derivator is an enhancement of a homotopy category with just enough information to determine homotopy limits and colimits by universal properties, which is exactly what we need for this theorem.
Derivators are often also easier to work with for formal results of this sort than other models of homotopy theory, such as model categories or $(\infty,1)$-categories.

Thus, after proving \autoref{thm:intro-smc} as stated, we prove an analogous theorem for closed symmetric monoidal derivators.

\begin{thm}\label{thm:intro-der}
  Let \dV be a closed, semi-additive, symmetric monoidal derivator.
  If $A$ is a finite category, $\Phi\colon A\op\to \dV$ is absolute and has a coefficient decomposition, and $X\colon A\to \dV$ is a diagram such that each $X_a$ is dualizable, then the weighted colimit $\colim^\Phi(X)$ is also dualizable, and for any endomorphism $f\colon X\to X$ we have
  \[ \tr(\colim^{\Phi}(f)) = \sum_{[\alpha]} \phi_{[\alpha]} \tr(X_\alpha\circ f_a). \]
\end{thm}

As before, the sum is again over conjugacy classes in $A$; the condition that $\Phi$ ``has a coefficient decomposition'' is technical and practically always satisfied.
\autoref{thm:intro-der} has the following specializations:
\begin{itemize}
\item All the examples of \autoref{thm:intro-smc} mentioned above also apply to derivators.
\item If \dV is stable and $i:X\to Y$ is a map between dualizable objects, then its cofiber $\cof(i)$ is also dualizable, and $\chi(\cof(i)) = \chi(Y) - \chi(X)$.
  This gives example~\ref{item:eg-cofiber}, while the generalization to traces gives~\eqref{eq:intro-add2}.
\item If \dV is stable and $Y\ot X \to Z$ is a span of dualizable objects, then its homotopy pushout $Y +_X^h Z$ is dualizable, and $\chi(Y+_X^h Z) = \chi(Y) + \chi(Z) - \chi(X)$.
  This implies examples~\ref{item:eg-hopo} and~\ref{item:eg-incl-excl}.
\item More generally, if \dV is stable and $X:A\to \dV$ is \emph{any} diagram with $A$ ``homotopy finite'' (see \S\ref{sec:hofin2}) and each $X_a$ dualizable, then $\colim(X)$ is dualizable, and we have
  \begin{equation}\label{eq:intro-hofin-chi}
    \chi(\colim(X))
    = \sum_{a} \chi(X_a) \cdot \sum_{k\ge 0} (-1)^k \cdot \card{\left\{
      \parbox{5.2cm}{\centering composable strings of nonidentity\\arrows of length $k$ starting at $a$}
    \right\}}
  \end{equation}
\item If \dV is stable and rational, and $X:A\to \dV$ is a diagram with each $X_a$ dualizable and $A$ a finite EI-category (i.e.\ every endomorphism is an isomorphism), then $\colim(X)$ is dualizable, and we have
  \begin{equation}
    \chi(\colim(X)) = \sum_{[a]} \sum_{C} \chi(X_{C}) \cdot
    \sum_k (-1)^k \sum_{[\vec{\alpha}]} \sum_{\vec{C}} \frac{\card{\vec{C}}}{\card{\nAut(\vec{\alpha})}}\label{eq:intro-ei-chi}
  \end{equation}
  where $\vec{\alpha}$ ranges over composable strings of noninvertible arrows of length $k$ starting at $a$, and $\vec{C}$ ranges over conjugacy classes of ``automorphisms of $\vec{\alpha}$'' (see \S\ref{sec:ei-categories}) restricting at $a$ to $C$.
\end{itemize}
As before, all of these also apply to traces of nonidentity morphisms.

These formulas also appear in the literature in various forms.
As mentioned before,~\eqref{eq:intro-add2} was proven abstractly by~\cite{add} for monoidal homotopy categories arising from a model structure, and then again in~\cite{gps:additivity} for stable monoidal derivators, using essentially the same method as May.
Our proof uses the basic definitions relating to monoidal derivators from~\cite{gps:additivity}, but the underlying idea of the proof is quite different from that of~\cite{add} --- and much more general, since it applies to colimits other than just cofibers.

On the other hand, when applied to diagrams of sets (via their suspension spectra), the formula~\eqref{eq:intro-hofin-chi} reproduces a large subclass of the formulas for cardinalities of colimits from~\cite{leinster:eccat}.
(Curiously, however, there are some examples to which both our theory and Leinster's apply, but yield different formulas.)

Finally, while this paper was in preparation, de Souza~\cite{souza:traces} independently obtained a formula equivalent to~\eqref{eq:intro-ei-chi} by other methods.
His approach relies on many explicit computations in derivators, while we use categorical abstraction to package such computations into conceptual facts.
(Much of this packaging was already done in~\cite{groth:ptstab,gps:stable,gps:additivity}; what remains is mostly isolated in \S\ref{sec:bco-der} of this paper.)
We expect that it would be possible to reduce both approaches to similar ideas, but in practice our paper proposes a very different perspective.

Even with \autoref{thm:intro-der} under our belts, however, we still have not captured all of the examples of interest.
For example, we cannot yet describe the additivity of the Reidemeister trace, since that is a trace in a bicategory (in the sense of~\cite{kate:traces,PS2}) rather than in a symmetric monoidal category.
However, it is completely straightforward to generalize Theorems~\ref{thm:intro-smc} and~\ref{thm:intro-der} to bicategories and even to \emph{derivator bicategories} (bicategories whose hom-categories are derivators).
This is a significant advantage of our approach to additivity over others such as~\cite{add} and~\cite{souza:traces}.
In the end, our most general linearity formula is the following.

\begin{thm}\label{thm:intro-derbi}
  Let \dW be a closed, locally semi-additive, derivator bicategory.
  Let $R$ and $S$ be objects of \dW, let $A$ be a finite category, let $\Phi\colon A\op\to \dW(R,R)$ be absolute and have a coefficient decomposition, and let $X\colon A\to \dW(R,S)$ be a diagram such that each $X_a$ is a dualizable 1-cell.
  Then the weighted colimit $\colim^\Phi(X)$ is also a dualizable 1-cell, and for any endomorphism $f\colon X\to X$ we have
  \[ \tr(\colim^{\Phi}(f)) = \sum_{[\alpha]} \phi_{[\alpha]} \tr(X_\alpha\circ f_a). \]
\end{thm}

Note that in the bicategorical case, our colimits are ``local colimits'' in a hom-category (or hom-derivator) $\dW(R,S)$, while the ``weight'' $\Phi$ is a diagram of 1-cells.
We recover ``unweighted'' colimits by taking $\Phi$ to be constant at the unit 1-cell $\lI_R$.
All the examples mentioned above generalize directly to the bicategorical context; here are a couple examples to give the idea.
\begin{itemize}
\item If \dW is locally semi-additive and $X,Y\in\dW(R,S)$ are dualizable 1-cells, then $X\oplus Y$ is dualizable, and $\chi(X\oplus Y) = \chi(X) + \chi(Y)$.
\item If \dW is locally stable and $i:X\to Y$ is a morphism of dualizable 1-cells in $\dW(R,S)$, then its cofiber is dualizable, and $\chi(\cof(i)) = \chi(Y) - \chi(X)$.
\end{itemize}
In particular, from the second example we can obtain a formula for the Reidemeister trace analogous to~\eqref{eq:intro-add2}: for $i:X\into Y$ an inclusion of dualizable spaces and $f:Y\to Y$ such that $f(X)\subseteq X$, we have
\[R(f)-i(R(f|_X))=R_{Y|X}(f),\]
where $R(f)$ is the Reidemeister trace of $f$ and $R_{Y|X}(f)$ is the ``relative Reidemeister trace''~\cite{kate:relative}.
However, this application requires a bit of work to construct the relevant derivator bicategory (which is a generalization of the ordinary bicategory of parametrized spectra from~\cite{maysig:pht}).
Since the focus of this paper is categorical rather than topological, we postpone this work to the companion paper~\cite{PS6}, which is logically dependent on this one; and give only a brief sketch of the proof in \S\ref{sec:derbicat}.

The generalization to bicategorical traces has one further advantage: it yields a uniqueness statement for linearity formulas.
Namely, if a linearity formula for some type of colimit can be shown to exist (by any method) and is sufficiently general (in particular, it must apply to bicategories as well as monoidal categories), then it \emph{must} arise from \autoref{thm:intro-derbi}.
This is a satisfying general statement that our approach does not ``miss'' any linearity formulas.

We now summarize the organization of the paper.
We begin in \S\ref{sec:traces} with a review of traces in symmetric monoidal categories and bicategories, including the notion of \emph{shadow} from~\cite{kate:traces,PS2} that enables the definition of bicategorical trace.
Of particular note is the \emph{composition theorem} for bicategorical trace, \autoref{thm:compose-traces}, which is easy to prove formally but directly gives rise to our linearity formulas.

The next two sections \S\S\ref{sec:moncat}--\ref{sec:moncat-examples} treat \autoref{thm:intro-smc} in the symmetric monoidal case.
In \S\ref{sec:moncat} we describe the general theorem (with the proof of one technical lemma postponed), then in \S\ref{sec:moncat-examples} we show how it applies to a number of examples.
Also in \S\ref{sec:moncat-examples} we recall the technical tool of \emph{base change objects} (representable profunctors), and use it to prove the missing lemma and construct several more examples, including the orbit-counting theorem.

In \S\S\ref{sec:der}--\ref{sec:eccatei} we move on to monoidal derivators.
We begin in \S\ref{sec:der} with the general theory, then in \S\ref{sec:eccat} we apply it to the main new class of examples: \emph{stable} monoidal derivators (such as classical stable homotopy theory).
In addition to the simple linearity formula~\eqref{eq:intro-add2}, we obtain a general formula for all \emph{homotopy finite} colimits in \S\ref{sec:fincolim},
 which agrees with Leinster's formula when both apply.
We generalize the  orbit-counting theorem to derivators in \S\ref{sec:deriv-burnside}.  In \S\ref{sec:eccatei}
we combine these results to obtain formulas for colimits over \emph{EI-categories} in rational stable derivators.

Next, in \S\ref{sec:bicat} we describe the theory for ordinary (i.e.\ non-derivator) bicategorical traces.
There are no especially new examples of traces here.
Then in \S\ref{sec:derbicat} we introduce derivator bicategories and prove the corresponding linearity theorem.  
Since this version of the theorem includes all the previous versions as special cases, it is not technically necessary to build up to it in stages.
However, it is easier to understand the ideas in simple cases first and then to introduce generalizations one by one.  
In \S\ref{sec:uniqueness} we prove the uniqueness statement for linearity formulas, establishing that the approach in this paper captures all similar linearity expressions.  As remarked previously, the generalization to 
derivator bicategories is an essential part of this result. 

Finally, in \S\S\ref{sec:base_change}--\ref{sec:bco-der} we discuss base change objects for monoidal derivators and derivator bicategories.  This allows us
to complete the identifications of the traces described in Parts \ref{part:more} and \ref{part:bicat}.
In \S\ref{sec:base_change} we describe a general structure for base change objects based on~\cite{shulman:frbi}; then we apply it to bicategories in \S\ref{sec:basechangebicat} and derivator bicategories in \S\ref{sec:bco-der}.

\part{Linearity in monoidal categories}
\label{part:moncat}

In this first part of the paper, we describe linearity explicitly and concretely in the simplest case: symmetric monoidal categories.

\section{Traces in monoidal categories and bicategories}
\label{sec:traces}

Let \V be a closed symmetric monoidal category with unit object \lS, monoidal product $\otimes$, and internal hom $\rhd$.
The latter means we have natural isomorphisms
\[ \V(X\otimes Y, Z) \cong \V(X, Y\rhd Z). \]
We refer to the internal-hom $X \rhd \lS$ as the \textbf{canonical dual} of $X$ and write it as $\dual X$.
There is a canonical \textbf{evaluation} map $\epsilon\colon \dual X \otimes X \to \lS$, defined by adjunction from the identity of $\dual X$.
See~\cite{dp:duality, lms:equivariant}.

We say that an object $X$ is \textbf{dualizable} if the canonical map
\begin{equation}\label{eq:dualmap}
  \mu_{X,U}\colon U\otimes \dual X \too X\rhd U
\end{equation}
(whose adjunct is $U\otimes \dual X \otimes X \xto{\id_U \otimes \epsilon} U\otimes \lS \cong U$)
is an isomorphism for all objects $U$.
It is sufficient to require this for $U=X$.

This definition of duality is convenient in concrete examples, but there is an equivalent characterization that tends to be more convenient for studying traces.

\begin{thm}
An object $X$ is dualizable if and only if there is an object $Y$ of $\V$ and morphisms 
\[\lS\xto{\eta} X\otimes Y\text{ and }Y\otimes X\xto{\epsilon} \lS\]
so that the composites 
\[X\cong \lS\otimes X\xto{\eta\otimes \id}X\otimes Y\otimes X\xto{\id\otimes \epsilon}X\otimes \lS\cong X\]
\[Y\cong Y\otimes \lS\xto{\id\otimes \eta}Y\otimes X\otimes Y\xto{\epsilon\otimes \id}\lS\otimes Y\cong Y\]
are identity maps.
\end{thm}
\begin{proof}[Sketch of proof]
  If $X$ is dualizable, let $Y=\dual X$, $\epsilon$ the evaluation as above, and $\eta$ the composite $\lS \to X\rhd X \xleftarrow{\cong} X\otimes \dual X$.
  Conversely, by composing with $\eta$ and $\epsilon$ we have natural isomorphisms $\V(Z\otimes X,U) \cong \V(Z,U\otimes Y)$, whence the Yoneda lemma gives $Y\cong \dual X$ by an isomorphism inducing~\eqref{eq:dualmap}.
\end{proof}

By analogy with the evaluation $\epsilon$, we call $\eta$ the \textbf{coevaluation}.

Using this characterization, we define the \textbf{trace} of an endomorphism $f\colon X\xto{}X$ of a dualizable object to be the composite 
\[\xymatrix{\lS\ar[r]^-\eta&X\otimes \dual X\ar[r]^-{f\otimes \id}&X\otimes \dual X \ar[r]^\cong&\dual X\otimes X\ar[r]^-\epsilon&\lS.
}\] 
We denote the trace of the identity morphism of $X$ by $\chi(X) = \tr(\id_X)$ and call it the \textbf{Euler characteristic} of $X$.

More generally, we may consider a closed \emph{bicategory} \W, with unit objects $\lI_B\in \W(B,B)$, bicategorical composition product $\odot$ and internal homs $\rhd$ and $\lhd$.
We write $\odot$ in diagrammatic order, so that if $X\in\W(A,B)$ and $Y\in\W(B,C)$ we have $X\odot Y \in\W(A,C)$, and we orient $\rhd$ and $\lhd$ so that the adjunction isomorphisms preserve cyclic order:
\[ \W(A,C)(X\odot Y,Z) \cong \W(A,B)(X, Y\rhd Z) \cong \W(B,C)(Y, Z\lhd X). \]
Any monoidal category can be regarded as a bicategory with only one object.
Another important example to keep in mind is the bicategory whose objects are (noncommutative) rings, whose morphisms are bimodules, with $\lI_B = {}_B B _B$ and $\odot$ the usual tensor product of bimodules, and $\rhd$ and $\lhd$ the usual hom-modules.

In a closed bicategory \W, if $X\in \W(A,B)$ is a 1-cell, we refer to the internal-hom $X\rhd \lI_B$ as the \textbf{canonical right dual}, written $\rdual X$.
There is again an evaluation map $\epsilon\colon \rdual X \odot X \to \lI_B$.
We say that $X$ is \textbf{right dualizable} if the analogous map
\begin{equation}\label{eq:rdualmap}
  \mu_{X,U}\colon U\odot \rdual X \too X\rhd U
\end{equation}
is an isomorphism for all 1-cells $U$.
Again, it suffices to require this for $U=X$.
As in the symmetric monoidal case, there are numerous other characterizations of dualizability; one will be relevant here.

\begin{thm}
An object $X\in \W(A,B)$ is dualizable if and only if there is an object $Y\in \W(B,A)$  and morphisms 
\[\lI_A\xto{\eta} X\odot Y\text{ and } Y\odot X\xto{\epsilon} \lI_B\]
so that the composites 
\[X\cong \lI_A\odot X\xto{\eta\odot \id}X\odot Y\odot X\xto{\id\odot \epsilon}X\otimes \lI_B\cong X\]
\[Y\cong Y\odot \lI_A\xto{\id\odot \eta}Y\odot X\odot Y\xto{\epsilon\odot \id}\lI_B\odot Y\cong Y\]
are identity maps.
\end{thm}
The proof is analogous to the monoidal case; see for instance~\cite[\S16.4]{maysig:pht}. 

To define the trace in a symmetric monoidal category we used the symmetry isomorphism.  The bicategories we are interested in do not have the same kind of 
symmetry, but we can introduce similar structure that will allow us to define a trace.  A \textbf{shadow} for $\W$ is collection of functors
\[\sh{-}\colon \W(A,A)\xto{} \T\]
for all objects $A$ of \W, where $\T$ is some fixed category, together with natural isomorphisms $\sh{X\odot Y}\cong \sh{Y\odot X}$ that are compatible with the unit and associativity isomorphisms of $\W$; see \cite[Defn.~4.1]{PS2} for details.
For an object $A$, we write $\sh{A} = \sh{\lI_A}$.
In the example of rings and bimodules, the shadow of an $A$-$A$-bimodule is its quotient abelian group that equalizes the right and left actions of $A$.
If a monoidal category is \emph{symmetric}, then its identity functor is a shadow for the corresponding one-object bicategory.

If $X\in\W(A,B)$ is right dualizable and \W is equipped with a shadow, then the {\bf trace} of a 2-cell $f\colon X\to X$ is the composite
\[\xymatrix{
\sh{A}\ar[r]^-\eta & \sh{X\odot \rdual X} \ar[r]^-{f\odot \id}&
\sh{X\odot \rdual X}\ar[r]^\cong&\sh{\rdual X\odot X}\ar[r]^-{\epsilon}&\sh{B}.
}\]
This general definition is due to~\cite{kate:traces} and was studied abstractly in~\cite{PS2}.
In particular, if $A$ and $B$ are noncommutative rings and $X$ is an $A$-$B$-bimodule, then this yields the \emph{Hattori-Stallings trace} of $f$.

More generally, a \emph{twisted} endomorphism $f\colon Q\odot X\xto{} X\odot P$ also has a trace, defined to be the composite
\[\xymatrix{
  \sh{Q}\ar[r] \ar[r]^-{\eta}
  & \sh{Q\odot X\odot \rdual X} \ar[r]^-{f\odot \id}
  &\sh{X\odot P\odot \rdual X} \ar[r]^-\cong &\sh{\rdual X\odot X\odot P}\ar[r]^-{\epsilon}&\sh{P}.
}\]
Originally, traces were only defined for untwisted endomorphisms, but 
there are many examples where the source and target twisting is essential, such as the Reidemeister trace to be discussed in \S\ref{sec:derbicat} and~\cite{PS6}.

The advantage of formulating traces abstractly in this way is that general theorems become easy to prove in the abstract context, but can reduce to quite nontrivial results in examples.
This is the case for our linearity formulas, which follow more or less directly (once the framework is set up correctly) from abstract theorems about compositions of dualizable objects.

For instance,
the following theorem is easy to prove, but it can be a source of many dual pairs that would otherwise be nontrivial to construct, as observed
in~\cite{maysig:pht}.
We will also use it in this way, to conclude that colimits of certain shapes are dualizable.

\begin{thm}\label{thm:compose-duals}
  If $Y$ and $X$ are right dualizable, then $Y\odot X$ is right dualizable, and we have $\rdual{(Y\odot X)} \cong \rdual X \odot \rdual Y$.
\end{thm}

In this case, if $g\maps Q\odot Y\to Y\odot P$
and $f\maps P\odot X \to X\odot L$ are two 2-cells, we have the
composite
\[(\id_Y\odot f)(g\odot \id_X)\maps Q\odot Y\odot X \too Y\odot X\odot L.\]
and we can ask about its trace.
This can be identified by a straightforward diagram chase.

\begin{thm}[{\cite[Prop.~7.5]{PS2}}]\label{thm:compose-traces}
  In the above situation, we have
  \[\tr\big((\id_Y\odot f)(g\odot \id_X)\big) = \tr(f)\circ \tr(g).\]
\end{thm}

\autoref{thm:compose-traces} is the origin of all our linearity formulas.
The basic idea is that given $X$ and $f$, we choose $Y$ so that $Y\odot X$ is the colimit of $X$.
With $g=\id_Y$, the left-hand side of \autoref{thm:compose-traces} is then the trace of $\colim(f)$, while the right-hand side expresses it as a composite of a ``row vector'' with a ``column vector'', hence a linear combination of the components of the trace of $f$.

\begin{rmk}\label{rmk:twisting}
Reflecting our interests here, we will make limited explicit use of twisted traces.  There will be none in the first two parts and only target twisting in the later parts.
Despite this, many of the results in the paper stated for untwisted or partially twisted traces extend to the case of more general twisting.
\end{rmk}

\section{Linearity in monoidal categories}
\label{sec:moncat}

For this section, let \V be a complete and cocomplete closed symmetric monoidal category, with tensor product $\otimes$, unit object $\lS$, and internal-hom $\rhd$.
Then we can construct the following closed bicategory $\prof(\V)$:
\begin{itemize}
\item Its objects are small categories $A$, $B$, $C$, \dots.
\item Its 1-cells are \V-profunctors (a.k.a.\ distributors, bimodules, or just ``modules'').
  A \V-profunctor $H\colon A\hto B$ is defined to be a functor $B\op\times A \to \V$.
\item Its 2-cells are morphisms of profunctors, i.e.\ natural transformations.
\item The composite of profunctors $H\colon A\hto B$ and $K\colon B\hto C$ is the coend
  \begin{equation}
    ({H\odot K})(c,a) = \int^{b\in B} H(b,a) \otimes K(c,b).
  \end{equation}
\item The unit 1-cell $\lI_A\colon A\hto A$ consists of copowers of the unit object \lS by the homsets of $A$:
  \[\lI_A(a,a') = A(a,a') \cdot \lS.\]
\item The right hom of profunctors $H\colon B\hto C$ and $K\colon A\hto C$ is the end
  \begin{equation}
    ({H\rhd K})(b,a) = \int_{c\in C} H(c,b) \rhd K(c,a)
  \end{equation}
  and similarly for the left hom $\lhd$.
\item It has a shadow valued in \V, defined by
  \begin{equation}
    \sh{H} = \int^{a\in A} H(a,a).
  \end{equation}
\end{itemize}

Let $\tc$ denote the terminal category, with one object and one (identity) morphism.
Then \V-profunctors $A\hto \tc$ are equivalent to \V-functors $A\to \V$, while profunctors $\tc\hto A$ are equivalent to functors $A\op\to \V$.
The latter sort of functor is the one traditionally used as a \emph{weight} for colimits in enriched category theory.
In the special case of \V itself, the traditional definition of weighted colimits is equivalent to the following.

\begin{defn}\label{def:smcwgtcolim}
  For functors $X\colon A\to \V$ and $\Phi\colon A\op\to\V$, the \textbf{$\Phi$-weighted colimit of $X$} is the composite of profunctors
  \[ \colim^\Phi(X) = \Phi\odot X = \int^{a\in A} \big(\underline\Phi(a) \otimes \uX(a)\big) \]
  regarded as an object of \V.
\end{defn}

If $\Phi$ is constant at the unit object $\lS$, then it is easy to identify the $\Phi$-weighted colimit of $X$ with its ordinary colimit.
This is the case we generally care most about, but it is conceptually helpful to consider the general case.
In particular, as we will see in \autoref{thm:trivmult}, including weighted colimits is what unifies ``additivity formulas'' with ``multiplicativity formulas''.

\begin{rmk}\label{rmk:enriched-cats}
  In fact, in enriched category theory one additionally considers  colimits where the diagram shape $A$ is a \V-enriched category; see for instance~\cite{kelly:enriched}.
  The definition of $\prof(\V)$ and everything we do with it can also be generalized to this situation; this follows from the general theorems in \autoref{part:formal}.
  However, the case of unenriched $A$ suffices for the examples here.
\end{rmk}

Now by \autoref{thm:compose-traces}, if $X$ and $\Phi$ are right dualizable when regarded as profunctors, then $\colim^\Phi(X)$ is dualizable in \V.
To avoid confusion, we introduce new names for profunctor right dualizability of $X$ and $\Phi$, which are of very different sorts.

\begin{defn}\ 
  \begin{itemize}
  \item A functor $X\colon A\to \V$ is \textbf{pointwise dualizable} if it is right dualizable in $\prof (\V)$ when regarded as a profunctor $A\hto \tc$.
  \item A functor $\Phi\colon A\op\to \V$ is \textbf{absolute} if it is right dualizable in $\prof (\V)$ when regarded as a profunctor $\tc\hto A$.
  \end{itemize}
\end{defn}

\noindent
Thus we have:

\begin{thm}
  If $X\colon A\to \V$ is pointwise dualizable and $\Phi\colon A\op\to \V$ is absolute, then $\colim^\Phi(X)$ is dualizable.
\end{thm}

By analogy with~\cite{PS3,PS4}, these notions might also be called \emph{fiberwise dualizable} and \emph{totally dualizable}.
However, when thinking of $\Phi$ as a weight for colimits, the term ``absolute'' is common;
it refers to the fact that in this case $\Phi$-weighted colimits in any \V-category are preserved by any \V-functor (see~\cite{street:absolute}).
(The word \emph{Cauchy} is also in use, since when metric spaces are regarded as enriched categories as in~\cite{lawvere:metric-spaces}, convergence of Cauchy sequences becomes an example.)
Similarly, when talking about diagrams (rather than fibrations, as in~\cite{PS4}), the adjective ``pointwise'' seems more intuitive than ``fiberwise''.
It is further justified by the following result that is closely related  to~\cite[Cor.~4.4]{PS4} and~\cite[Lem.~11.5]{gps:additivity}.

\begin{lem}\label{thm:smcpwdual}
  A functor $X\colon A\to \V$ is pointwise dualizable if and only if each object $\uX(a)$ is dualizable in \V.
\end{lem}
\begin{proof}
  For any $U\in\prof(V)(B,\tc)$, $a\in A$, and $b\in B$, the $(b,a)$ component of the morphism $\mu_{X,U}$ from~\eqref{eq:rdualmap} is $\mu_{\uX(a),\uU(b)}$.
  But $\mu_{X,U}$ is an isomorphism as soon as all its components are.
\end{proof}

Now \autoref{thm:compose-traces} gives us a formula for traces.

\begin{thm}\label{thm:compose-traces-smc}
  If $X\colon A\to\V$ is pointwise dualizable and $\Phi\colon A\op\to\V$ is absolute, then for any $f\colon X\to X$, we have
\begin{equation}
  \tr(\colim^\Phi(f)) = \tr(f) \circ \tr(\id_\Phi).\label{eq:smlin}
\end{equation}
\end{thm}
\begin{proof}
  Identify $\colim^\Phi(f)$ with
  $(\id_\Phi\odot f) \colon \Phi\odot X\to \Phi\odot X$
  and then apply \autoref{thm:compose-traces}.
\end{proof}

In order for this to be useful we need to be able to calculate $\tr(f)$ and $\tr(\id_\Phi)$ more concretely.
First note that since colimits commute with colimits (including copowers), we have
\[ \sh{A}\; = \;\int^{a\in A} \Big(A(a,a) \cdot \lS\Big) \;\cong\; \left(\int^{a\in A} A(a,a) \right) \cdot \lS. \]
The set $\int^{a\in A} A(a,a)$ is the disjoint union of all the endomorphism sets $A(a,a)$, quotiented by the relation $\alpha \circ \beta \sim \beta \circ \alpha$ for any $\alpha,\beta$ (which need not be endomorphisms themselves).
We call this relation \emph{conjugacy}, since when $A$ is a group regarded as a 1-object category, it becomes precisely the relation of conjugacy, and $\int^{a\in A} A(a,a)$ is the set of conjugacy classes.

Thus, $\sh{A}$ is the copower of $\lS$ by the set of conjugacy classes of $A$, and we have coprojections
\[ A(a,a) \cdot \lS \too \sh{A} \]
that send the copy of \lS in the domain corresponding to each $\alpha\in A(a,a)$ to the copy in the codomain corresponding to its conjugacy class.
Since these coprojections are jointly epimorphic,
$\tr(f)$ is determined by one morphism $\tr(f)_{[\alpha]}\colon \lS\to\lS$ for each conjugacy class $[\alpha]$ of $A$.
The following lemma identifies these morphisms.

\begin{lem}[The component lemma for symmetric monoidal
categories]\label{thm:smcomega2a}
   For any morphism $\alpha\in A(a,a)$, $\tr(f)_{[\alpha]}$ is the trace of the composite 
\begin{equation}
  \xymatrix{X(a) \ar[r]^-{X_{\alpha}} & X(a) \ar[r]^-{f_a} & X(a)}
\label{eq:trfal}
\end{equation}
\end{lem}

\begin{proof}
We will prove this in \S\ref{sec:moncat-examples} on page \pageref{thm:smcomega2p}. 
\end{proof}
Thus, we have a complete computation of $\tr(f)$ for any endomorphism $f$ of a pointwise dualizable $X\colon A\to\V$.

The other ingredient in~\eqref{eq:smlin} is $\tr(\id_\Phi)$.
This is a morphism $\lS \to \sh{A}$ which depends on $\Phi$; we call it the \textbf{coefficient vector} of $\Phi$.
This name is inspired by the case of most interest: when $A$ is finite and \V is {\bf semi-additive} (i.e.\ finite products and coproducts coincide naturally,  and are called \emph{biproducts} or \emph{direct sums} and written with $\oplus$).
In this case, $\sh{A}$ is a direct sum of copies of $\lS$ indexed by the conjugacy classes of $A$, and so $\tr(\id_\Phi)$ really is a ``column vector'' whose entries are morphisms $\lS\to\lS$.
(Note that when \V is semi-additive, $\V(\lS,\lS)$ is a commutative semiring which acts on every homset of \V.)
We denote the entries of this column vector by $\phi_{[\alpha]}$, and call them the \textbf{coefficients} of $\Phi$.

Similarly, in this case $\tr(f)$ is a ``row vector'' whose entries are the traces $\tr(f_a \circ X_\alpha)$, giving \autoref{thm:compose-traces-smc} a more familiar form.

\begin{cor}\label{thm:smclin2}
  If $\V$ is semi-additive, $A$ is finite, and $\Phi\colon A\op\to\V$ is absolute, then we have
  \begin{equation}\label{eq:smclin2}
    \tr(\colim^\Phi(f)) = \sum_{[\alpha]} \phi_{[\alpha]} \cdot \tr\big(f_a \circ X_\alpha\big).
  \end{equation}
  for any pointwise dualizable $X\colon A\to \V$ and $f\colon X\to X$.
\end{cor}

This is the origin of our term \textbf{linearity formula}.
In particular, when $f$ is the identity morphism, we have
\begin{equation}
  \chi(\colim^\Phi(X)) = \sum_{[\alpha]} \phi_{[\alpha]} \cdot \tr\big(X_{\alpha}\big).
\end{equation}

\section{Examples}
\label{sec:moncat-examples}

In order to obtain concrete examples, we need to identify some absolute weights and calculate their coefficient vectors.
There is no general way to do this: the question of which weights are absolute depends heavily on what \V we choose, even for simple categories $A$.

Most of the examples we can describe at this point are fairly trivial, in the sense that their linearity formula can be proven easily in more direct ways.
Thus, while it is satisfying to have a general theory, it may seem at this point that it doesn't buy us very much.
This is largely true in the non-homotopical case, although in Examples~\ref{eg:dirsum} and~\ref{thm:burnside} we get a little simplification, amounting to the fact that it suffices to prove the linearity formula in a few particularly simple cases.
This will also be true in the homotopical examples to be considered in \S\ref{sec:der} and beyond, but in that case it is a much bigger win.

\begin{eg}\label{eg:zero}
  Let $A$ be the empty category, and $\Phi\colon A\op \to\V$ the unique functor; then $\Phi$-weighted colimits are initial objects.
  For any category $B$, there is a unique profunctor $U\colon B\hto A$, and the map $\mu_{\Phi,U}\colon  U\odot \rdual\Phi \to \Phi\rhd U$ is the unique map from the initial to the terminal object of $\prof(\V)(B,\tc)$.
  To say that this is an isomorphism when $B=\tc$ is by definition to say that \V is \emph{pointed}, and this in turn implies the corresponding statement for general $B$.
  When \V is pointed, its joint initial and terminal object is called the \emph{zero object} and denoted $0$.

  Thus, this $\Phi$ is absolute just when \V is pointed.
  There is a unique functor $X\colon A\to\V$, which is trivially pointwise dualizable; hence its colimit, which is the zero object of \V, is dualizable.
  Finally, the shadow of $A$ is the zero object $0$, so the trace of the unique endomorphism of $0$ is the composite $\lS \to 0 \to \lS$, i.e.\ the zero endomorphism of \lS.
  It is quite trivial to prove all this directly, but it serves as a good beginning example to see the general theory working.
\end{eg}

\begin{eg}\label{eg:dirsum}
  Let $A$ be the discrete category with two objects $a$ and $b$.
  Then a diagram $X\colon A\to\V$ consists of a pair of objects $X_a$ and $X_b$, and is pointwise dualizable just when $X_a$ and $X_b$ are dualizable.

  Let $\Phi\colon A\to \V$ be constant at $\lS$.
  Then $\colim^\Phi(X) = X_a + X_b$, i.e.\ $\Phi$-weighted colimits are binary coproducts.
  Now the right dual $\rdual{\Phi}$ is given by
  \begin{align}
    (\rdual\Phi)_a
    &= \int^x \Phi(x) \rhd \lI_A(x,a)\\
    &\cong \Big(\Phi(a)\rhd \lI_A(a,a)\Big) \times \Big(\Phi(b) \rhd \lI_A(b,a)\Big)\\
    &\cong \Big(\lS\rhd \lS\Big) \times \Big(\lS \rhd \emptyset\Big)\\
    &\cong \lS \times \big(\lS \rhd \emptyset\big)
  \end{align}
  and similarly $(\rdual\Phi)_b \cong (\lS\rhd \emptyset) \times \lS$, where $\emptyset$ is the initial object of \V.

  If \V is pointed, then $\emptyset = 0$ and $\lS\rhd 0 \cong 0$ and $\lS\times0 \cong \lS$, so that $\rdual\Phi$ is also constant at \lS.
  Thus, for $U\in\prof(\V)(C,A)$ we have
  \begin{align}
    (U\odot \rdual\Phi)_{c} &= U_{c,a} + U_{c,b} \qquad\text{and}\\
    (\Phi \rhd U)_c &= U_{c,a} \times U_{c,b}
  \end{align}
  while $\mu_{\Phi,U}$ is the canonical morphism
  \[  U_{c,a} + U_{c,b}  \too U_{c,a} \times U_{c,b} \]
  whose components $U_{c,a}\to U_{c,a}$ and $U_{c,b}\to U_{c,b}$ are the identity and whose components $U_{c,a}\to U_{c,b}$ and $U_{c,b}\to U_{c,a}$ are zero morphisms.
  To say that this is an isomorphism when $C=\lS$ is by definition to say that \V is semi-additive, and this in turn implies the corresponding statement for general $C$.

  Thus, when \V is semi-additive, this $\Phi$ is absolute, and so binary coproducts of dualizable objects are dualizable.
  We have $\sh{A} \cong \lS \oplus \lS$, and for a pointwise dualizable $X$ and $f\colon X\to X$, \autoref{thm:smcomega2a} implies that $\tr(f)\colon \lS\oplus \lS \too \lS$ is the row vector composed of $\tr(f_a)$ and $\tr(f_b)$.
  Thus, we have
  \[\tr(f_a \oplus f_b) = \phi_a \cdot \tr(f_a) + \phi_b\cdot \tr(f_b)\]
  for some $\phi_a, \phi_b \in \V(\lS,\lS)$.
  Knowing that such $\phi_a$ and $\phi_b$ exist, and are the same for all $X$ and $f$, enables us to calculate them easily.
  Namely, let $X_a = 0$ and $X_b=\lS$ and let $f$ be the identity.
  Then $\tr(f_a)=0$ by the previous example, and $\tr(f_b)=1$ since it is the identity; while $X_a \oplus X_b \cong \lS$ and $f_a \oplus f_b = 1$, so that $\tr(\colim^\Phi f)=1$ as well.
  Thus, $1 = \phi_a \cdot 0 + \phi_b \cdot 1$, so $\phi_b = 1$.
  Similarly, $\phi_a=1$, so our linearity formula is
  \[\tr(f_a \oplus f_b) = \tr(f_a) + \tr(f_b).\]
  As before, of course, it is fairly easy to prove this directly.
\end{eg}

\begin{eg}\label{eg:sup}
  The formal analysis of \autoref{eg:dirsum} applies equally well when $A$ is \emph{any} discrete category.
  Semi-additivity of \V again implies that all \emph{finite} coproducts are absolute, with an analogous linearity formula.
  Examples of \V for which \emph{infinite} coproducts are absolute arise somewhat more rarely, but they do exist.
  For instance, if \V is the category of suplattices (i.e.\ posets with all suprema, and supremum-preserving functions), then coproducts of arbitrary cardinality are absolute, and we have an analogous linearity formula:
  \[ \tr\left(\bigoplus_a f_a\right) = \sum_{a} \tr(f_a). \]
  In this case, \lS is the two-element lattice, and $\V(\lS,\lS)$ is a two-element set, while sums of morphisms are pointwise suprema.
  Thus, traces carry very little information.
  Informally, while traces in the additive case ``count'' fixed points, traces in the suplattice case merely record whether any fixed point exists.
\end{eg}

\begin{eg}\label{thm:trivmult}
  Let $A=\tc$ be the terminal category.
  Then $X\colon A\to\V$ is just an object of \V, and is pointwise dualizable just when that object is dualizable.
  Similarly, $\Phi\colon A\op\to\V$ is also just an object of \V, and is absolute just when that object is dualizable, while $\colim^\Phi(X)$ is just the tensor product $\Phi\otimes X$.

  The shadow of $\tc$ is just the unit object \lS, and the trace of $f\colon X\to X$ is just its ordinary trace in \V.
  The trace of $\id_\Phi$ is also obviously just its trace in \V, giving the unique coefficient $\phi$.
  Thus the linearity formula reduces to
  \[ \tr(\Phi \otimes f) = \tr(f)\circ \tr(\id_\Phi), \]
  which is (a special case of) the usual \emph{multiplicativity} formula for traces, \cite{PS4}.
  Thus we see that \emph{linearity} includes both \emph{additivity} (as the special case when all coefficients are $1$) and \emph{multiplicativity} (as the special case when there is only one term).
  The ``twisted multiplicativity'' of~\cite{PS4} is also a sort of linearity, but using a more complicated bicategory.
\end{eg}

\begin{eg}\label{thm:suspension}
  Let \V be the category of $\mathbb{Z}$-graded objects in an additive symmetric monoidal category $\mathbf{U}$, with the usual tensor product:
  \[ (X\otimes Y)_n = \bigoplus_{k+m=n} X_k \otimes Y_m \]
  and the symmetry isomorphism that maps $X_k \otimes Y_m$ to $Y_m \otimes X_k$ by $(-1)^{k m}$.
  Let $S^n$ denote the graded object that is the unit object \lS of $\mathbf{U}$ in degree $n$ and $0$ in all other degrees.
  Then $S^n$ is dualizable (indeed, invertible) with dual $S^{-n}$, and inspecting the definition of trace yields $\tr(\id_{S^n})=(-1)^n$.

  Hence, \autoref{thm:trivmult} implies that if $X$ is dualizable, so is $S^1 \otimes X$, and the trace of $S^1\otimes f$ is the negative of the trace of $f$.
  Of course, $S^1 \otimes X$ is just the ``suspension'' of $X$, with $(S^1\otimes X)_n = X_{n-1}$.
  A similar argument applies to chain complexes in an abelian symmetric monoidal category.
\end{eg}

Before we give more examples of \autoref{thm:compose-traces-smc} we introduce some important general examples of dualizable profunctors.
For any profunctor $H\colon B\hto D$ and any functors $f\colon A\to B$ and $g\colon C\to D$, we have an induced profunctor $H(g,f)\colon A\hto C$ defined by
\[ ({H(g,f)})(c,a) = H(gc,fa)\cdot \lS. \]
In particular, taking $H$ to be the identity profunctor $\lI_B$ and $g$ to be the identity functor, we have a profunctor $\lI_B(\id_A,f)\colon A\hto B$, which we generally denote by $B(\id,f)$.
Similarly, we have $B(f,\id)\colon B\hto A$; these two are defined by
\begin{align}
  (B(\id,f))(b,a) = B(b,fa) \cdot \lS \qquad\text{and}\qquad
  (B(f,\id))(a,b) = B(fa,b) \cdot \lS.
\end{align}
These are called \textbf{representable profunctors} or \textbf{base change objects}.
The following facts about them are well-known and easy to prove.
Abstractly, they say that $\prof(\V)$ is a \emph{proarrow equipment}~\cite{wood:proarrows-i} or a \emph{framed bicategory}~\cite{shulman:frbi}; we will return to this point of view in \S\ref{sec:framed}.

\begin{prop}\label{thm:bcodual}
  $B(\id,f)$ is right dualizable, with right dual $B(f,\id)$.
  The evaluation has components
  \[ \int^{a\in A} (B(fa,b) \cdot \lS) \otimes (B(b',fa)\cdot \lS) \too  B(b',b) \cdot \lS \]
  given by composition in $B$, while the coevaluation has components
  \[ A(a,a')\cdot \lS \too B(fa,fa')\cdot \lS \;\cong\; \int^{b\in B} (B(b,fa')\cdot \lS) \otimes (B(fa,b)\cdot \lS) \]
  given by the action of $f$ on arrows.
\end{prop}

\begin{prop}\label{thm:bcorestr}
  For any \V-profunctor $H\colon B\hto D$ and \V-functors $f\colon A\to B$ and $g\colon C\to D$, we have
  \begin{align}
    H(\id,f) &\cong B(\id,f) \odot H\\
    H(g,\id) &\cong H \odot D(g,\id).
  \end{align}
\end{prop}

We can also explicitly calculate traces with respect to this dual pair.

\begin{prop}\label{thm:bcodual2}
  If $\mu\colon f\to f$ is a natural transformation, then the trace of the induced endomorphism $\mu\colon B(\id,f) \to B(\id,f)$ is the map
  \[ \sh{A} = \int^{a\in A} A(a,a) \cdot \lS \too \int^{b\in B} B(b,b)\cdot \lS = \sh{B} \]
  induced by the diagonals of the following commutative squares:
  \begin{equation}
    \vcenter{\xymatrix{
      A(a,a)\ar[r]^-{\mu_a\otimes f}\ar[d]_{f\otimes \mu_a} &
      B(fa,fa) \times B(fa,fa)\ar[d]^{\mathrm{comp}}\\
      B(fa,fa) \times B(fa,fa)\ar[r]_-{\mathrm{comp}} &
      B(fa,fa)
      }}
  \end{equation}
  In particular, the trace of the identity map $\id_{B(\id,f)}$ is induced by the maps $f\colon A(a,a)\xto{} B(fa,fa)$.
\end{prop}
\begin{proof}
  By inspection of the definition of traces and the description of the evaluation and coevaluation in \autoref{thm:bcodual}.
\end{proof}

\begin{eg}
  For any $a\in A$, the profunctor $A(\id,a)$ is an absolute weight.
  By \autoref{thm:bcorestr}, the colimit of $X\colon A\to\V$ weighted by $A(\id,a)$ is just $X(a)$, which is dualizable whenever $X$ is pointwise dualizable.
  By \autoref{thm:bcodual2}, the coefficient vector of $A(\id,a)$ is the map $\lS \to \sh{A}$ induced by the identity morphism of $a$, so that the linearity formula becomes the obvious fact that $\tr(f_a) = \tr(f_a)$.
\end{eg}

We can obtain less trivial examples by invoking the following easy fact.

\begin{prop}\label{thm:retracts}
  Any retract of a right dualizable 1-cell in a bicategory is again right dualizable.
  That is, if $X,Y\in \W(A,B)$ and we have $r\colon X\to Y$ and $s\colon Y\to X$ with $r s = \id_Y$, and $X$ is right dualizable, then so is $Y$.
  Moreover, if \W has a shadow, then the trace of any $f\colon Y\to Y$ is equal to the trace of $sfr\colon X\to X$.
\end{prop}
\begin{proof}
  If $Y$ is a retract of $X$, then $\mu_{Y,U}$ is a retract of $\mu_{X,U}$ for any $U$, and a retract of an isomorphism is an isomorphism.
  The statement about traces follows from cyclicity (\cite[Corollary~7.3]{PS2}), since $\tr(f) = \tr(f r s) = \tr(s f r)$.
\end{proof}

This immediately gives rise to the following somewhat tautological example.

\begin{eg}\label{thm:idempotents}
  Let $A$ be the category generated by a single object $a$ and a single idempotent $\alpha\colon a\to a$.
  Then a functor $A\to \V$ consists of an object $X$ together with an idempotent $e\colon X\to X$, and is pointwise dualizable just when $X$ is dualizable.
  Similarly, a functor $\Phi\colon A\op\to\V$ is also just an object with an idempotent.
  If we take $\Phi$ to be the unit object \lS with the identity idempotent, then a $\Phi$-weighted colimit of $(X,e)$ is a splitting of the idempotent $e$.

  To see $\Phi$ is absolute, first observe that 
  the unique representable profunctor $A(\id,a)\colon\tc \hto A$ is absolute.
  Concretely, the value of $A(\id,a)$ on the single object is the coproduct $\lS+ \lS$, equipped with the idempotent induced by $\alpha$, which projects it onto the first summand.
  The splitting of this idempotent is precisely our weight $\Phi$, so by \autoref{thm:retracts}, $\Phi$ is also absolute.
  Moreover, its coefficient vector $\tr(\id_\Phi)$ is the trace of the idempotent $\alpha$.
  Since $A$ has two conjugacy classes, the identity and the idempotent, this coefficient vector is a morphism
  \[ \phi \colon  \lS \to \{\id_a,\alpha\} \cdot \lS \cong \lS + \lS. \]
  By \autoref{thm:bcodual2}, this map is induced by the action of the functor $a\colon \lS \to A$ (which yields the first coprojection) followed by composing with $\alpha$.
  Thus, it is just the second coprojection.

  Now suppose $e\colon X\to X$ is an idempotent and we have $f\colon X\to X$ such that $f e = e f$, so that $f$ is an endomorphism of $(X,e)\colon A\to \V$.
  Then if $Y$ is a splitting of $(X,e)$, with section $s\colon Y\to X$ and retraction $r\colon X\to Y$, the induced endomorphism of $Y$ is the composite $r f s$, and the general linearity formula says that $\tr(r f s)$ is the composite
  \begin{equation}\label{eq:tridem3}
    \xymatrix@C=2pc{\lS \ar[r]^-{\phi} & \lS + \lS \ar[rr]^-{[\tr(f), \tr(f e)]} && \lS. }
  \end{equation}
  Since $\phi$ is the second coprojection, this yields $\tr(f e)$.
  This also follows directly from the cyclicity of ordinary traces.
  In an additive context, we may say that the coefficients of $\Phi$ are $\phi_{[\id_a]} = 0$ and $\phi_{[\alpha]} = 1$, but in this case this formula holds whether or not \V is additive.
\end{eg}

For a less trivial example, let $G$ be a finite group and $A=\bg$ the corresponding one-object groupoid.
Then a functor $\bg \to \V$ consists of an object $X$ with a left action by $G$, and is pointwise dualizable just when $X$ is dualizable.
If we take $\Phi\colon \bg\op \to \V$ to be \lS with the trivial right $G$-action, then $\colim^\Phi(X)$ is the quotient $X/G$.

For absoluteness of this weight, we need an additional condition on $\V$.
Recall that if $\V$ is semi-additive, then the monoid $\V(\lS,\lS)$ of endomorphisms of the unit object is a semiring.
In particular, any positive integer $n$ can be regarded as an element of this semiring, namely $\overbrace{\id_\lS+\id_\lS + \cdots+\id_\lS}^n$.
We say that $\V$ is \textbf{$n$-divisible} if this semiring element is invertible, and write $\frac{1}{n}$ for its inverse.
This implies that for any morphism $h\colon X\to Y$ in $\V$, there is a morphism $k\colon X\to Y$ such that $h = \overbrace{k+k+\cdots+k}^{n}$.
Specifically, $k$ is the composite
\[ X \cong X\otimes \lS \xrightarrow{h \otimes \frac{1}{n}}  Y\otimes \lS \cong Y. \]
We denote this morphism $k$ by $\frac{1}{n}\cdot h$.

\begin{thm}\label{thm:burnside}
  Suppose $G$ is a finite group, $X\colon \bg\to \V$, and $f\colon X\to X$ is a natural transformation.  If \V is semi-additive and $\card G$-divisible, then   
  \begin{equation}
    \tr(f/G) = \frac{1}{\card G} \sum_{g\in G} \tr(f\circ X(g))
  \end{equation}
  where $X(g)$ is the action of $g\in G$ on $X$.
\end{thm}

This a fundamental example of our approach and we will extend this result to derivators in \S\ref{sec:deriv-burnside}.

\begin{proof}
  The unique representable $\bg(\id,a)$ is the copower $G\cdot \lS \cong \bigoplus_{g\in G} \lS$, with the right $G$-action that permutes the summands, i.e.\ $g\in G$ sends the $h\th$ summand to the $(hg)\th$.
  If \lS has the trivial $G$-action, the fold map
  \[ r = [\id]_{g\in G} \colon \bigoplus_{g\in G} \lS \too \lS \]
  is $G$-equivariant, i.e.\ is a morphism in $\V^\bg$.
  The diagonal $(\id)_{g\in G}\colon  \lS \to \bigoplus_{g\in G} \lS$ is also $G$-equivariant for these actions; let $s$ be the morphism
  \[ s = {\textstyle\frac1{\card G}} \cdot (\id)_{g\in G} \colon \lS \too \bigoplus_{g\in G} \lS. \]
  Then the composite $r s \colon \lS \to \lS$ is
  \[ \displaystyle\sum_{g\in G} \frac{1}{\card G} \;=\; \card G\cdot \frac{1}{\card G} \;=\; \id_{\lS}.\]
  Hence, our weight $\Phi$ is a retract of $\bg(\id,a)$ in $\V^\bg= \prof(\V)(\tc,\bg)$, and thus is absolute.
  Therefore, if $X$ is a dualizable object with any left $G$-action, its quotient $X/G$ is also dualizable, and we have a linearity formula for traces.

  Since $\id_{\Phi} = rs$ as above, by the cyclicity of traces, the coefficient vector $\tr(\id_\Phi)$ is equal to $\tr(sr)$.
  Inspecting the definitions of $s$ and $r$, we see that $sr$ is the $G$-equivariant endomorphism of $\bigoplus_{g\in G} \lS$ determined by the $(\card G\times\card G)$-matrix where all entries are $\frac{1}{\card G}$.
  To calculate the trace of this endomorphism, note that the dual representable $\bg(a,\id)$ is $\bigoplus_{g\in G} \lS$ with the \emph{left} $G$-action where $g\in G$ sends the $h\th$ summand to the $(gh)\th$, while the unit profunctor $\lI_\bg$ is $\bigoplus_{g\in G} \lS$ with both right and left $G$-actions.
  By \autoref{thm:bcodual}, the coevaluation
  \[ \lS \;\too\; \int^{\bg} \left(\bigoplus_{g\in G} \lS \;\otimes\; \bigoplus_{h\in G} \lS\right)
  \;\toiso\; \int^{\bg} \bigoplus_{g,h\in G} \lS
  \]
  picks out the image of the $(e,e)\th$ summand, whereas the evaluation
  \[ \bigoplus_{g\in G} \lS \;\otimes\;  \bigoplus_{h\in G} \lS \;\;\toiso\;\;
  \bigoplus_{g,h\in G} \lS \;\too\; \bigoplus_{g\in G} \lS
  \]
  maps the $(g,h)\th$ summand to the $(gh)\th$ summand.
  Thus, the trace of $sr$:
  \[ \lS \too \int^{\bg} \bigoplus_{g,h\in G} \lS \xto{sr\otimes \id} \int^{\bg} \bigoplus_{k,\ell\in G} \lS \too \int^\bg \bigoplus_{m\in G} \lS \]
  is induced (after passage to $\int^\bg$) by the composite
  \begin{equation}
    \lS \xto{(\delta_{g,e}\cdot \delta_{h,e})_{g,h}} \bigoplus_{g,h\in G} \lS
    \xto{\left(\frac{1}{\card G} \cdot \delta_{h,\ell}\right)_{g,h,k,\ell}} \bigoplus_{k,\ell\in G} \lS
    \xto{(\delta_{k \ell,m})_{k,\ell,m}} \bigoplus_{m\in G} \lS\label{eq:pretrsr}
  \end{equation}
  in which the $\delta$'s are Kronecker's.
  Multiplying these matrices, we see that the $m\th$ component of~\eqref{eq:pretrsr} is
  \[ \sum_{g,h,k,\ell} \left(\delta_{g,e} \cdot \delta_{h,e} \cdot \frac{1}{\card G} \cdot \delta_{h,\ell} \cdot \delta_{k\ell,m}\right) = \frac{1}{\card G}\]
  Since passage to $\int^\bg$ simply identifies the $k\th$ and $(k')\th$ summands when $k$ and $k'$ are in the same conjugacy class of $G$, the trace of $sr$ is the vector
  \[ \lS \xto{\left(\frac{\card C}{\card G}\right)_{C}} \bigoplus_C \lS \]
  where $C$ ranges over conjugacy classes in $G$.

  In other words, the coefficient vector of $\Phi$ assigns to a conjugacy class $C$ the number $\frac{\card C}{\card G}$.
  Thus, our linearity formula is
  \begin{align}
    \tr(f/G) &= \sum_{C} \frac{\card C}{\card G} \cdot  \tr(f \circ X(C)).
  \end{align}
  where $X(C)$ denotes the action on $X$ of some element of $C$ --- by cyclicity of traces, it doesn't matter which.
  We can  split this up as a sum over elements of $G$ rather than conjugacy classes to recover the description in 
  the statement of the theorem.
\end{proof}

  In particular, if $f$ is the identity morphism of $X$, then we have
  \begin{equation}\label{eq:burnside}
    \chi(X/G) = \frac{1}{\card G} \sum_{g\in G} \tr(X(g)).
  \end{equation}
  This is a generalization of the orbit-counting theorem (a.k.a.\ Burnside's lemma or the Cauchy-Frobenius lemma).
  Namely, suppose $\V = R\text-\mathbf{Mod}$ for a commutative ring $R$ in which $\card G$ is invertible, $Z$ is a finite $G$-set, and $X = R[Z]$ with the induced $G$-action.
  Then $X/G \cong R[Z/G]$, so that $\chi(X/G) = \card (Z/G)$; while $\tr(X(g))$ is the number of fixed points of $g$ acting on $Z$.
  Thus~\eqref{eq:burnside} reduces exactly to the orbit-counting theorem:
  \[ \card{(X/G)} = \frac{1}{\card G} \sum_{g\in G} \card{(X^g)}. \]

\begin{rmk}
  Note that in the previous example, the pointwise trace of $\id_X$ is the ``row vector'' with entries indexed by conjugacy classes of $G$, which assigns to each conjugacy class the trace of its action on $X$.
  In other words, it is the \emph{character} of the group representation $X$.
\end{rmk}

\begin{rmk}
  In addition to colimits weighted by profunctors $\Phi \colon \tc\hto A$, we may consider those weighted by arbitrary profunctors $\Phi \colon  B\hto A$.
  In this case the ``colimit'' of $X\colon A \to\V$, defined as before using the tensor product of functors, is a diagram $ B\to \V$ rather than a single object.
  For instance, if $\theta \colon G\to H$ is a group homomorphism, with corresponding functor $\bth\colon \bg \to\bh$, then the colimit of $X\colon\bg\to\V$ weighted by the representable $\bh(\mathbf{B}\theta,\id)$ is the \emph{induced representation} of $X$ along $\theta$.
  In this case, a computation similar to \autoref{thm:burnside} produces the formula for the character of an induced representation:
  \[ \tr(\colim^{\bh(\theta,\id)}(X)(h)) =
  \frac{1}{\card G}
  \sum_{k\in H \atop k^{-1} h k = \theta(g)} \tr(X(g))
  \]
\end{rmk}

Finally, we can use the base change profunctors to prove \autoref{thm:smcomega2a}.
Recall the statement, which applies in the situation of a pointwise dualizable $X\colon A\to \V$ and an endomorphism $f\colon X\to X$.
	
\begin{replem}{thm:smcomega2a}[The  component lemma for symmetric monoidal
categories]\label{thm:smcomega2p}
   For any morphism $\alpha\in A(a,a)$, $\tr(f)_{[\alpha]}$ is the trace of the composite 
   \begin{equation}
     \xymatrix{X(a) \ar[r]^-{X_{\alpha}} & X(a) \ar[r]^-{f_a} & X(a)}
     \label{eq:trfal2}
   \end{equation}
\end{replem}

Here $\tr(f)_{[\alpha]}$ denotes the composite of $\tr(f)\colon \sh{A} \to \lS$ with the coprojection $\lS \to \sh{A}$ induced by the conjugacy class $[\alpha]$.
	
\begin{proof}
  Let $a$ also denote the functor $\tc \to A$ that picks out the object $a\in A$.
  Then by \autoref{thm:bcodual}, the profunctor $A(\id,a)\colon \tc \hto A$ is right dualizable, and we have an endomorphism $A(\id,\alpha) \colon A(\id,a) \to A(\id,a)$ induced by composition with $\alpha$.
  Therefore, by \autoref{thm:compose-traces}, the trace of the composite
  \begin{equation}
    A(\id,a) \odot X \xto{A(\id,\alpha)\odot \id} A(\id,a) \odot X \xto{\id \odot f} A(\id,a) \odot X \label{eq:trfal3}
  \end{equation}
  is equal to the composite
  \[ \lS \xto{\tr(A(\id,\alpha))} \sh{A} \xto{\tr(f)} \lS. \]
  However, under the isomorphism $A(\id,a) \odot X \cong X(a)$ of \autoref{thm:bcorestr},~\eqref{eq:trfal3} is identified with~\eqref{eq:trfal2}.
  Finally, \autoref{thm:bcodual2} tells us that $\tr(A(\id,\alpha))$ is the coprojection $\lS \to \sh{A}$ induced by $[\alpha]$.
\end{proof}

\begin{rmk}\label{rmk:non-cocomplete}
  For simplicity, we have assumed that our monoidal category \V is complete, cocomplete, and closed.
  However, it follows for formal reasons that the same linearity formulas hold even if \V admits only the particular colimits in question.
  For instance, using the methods of~\cite[\S3.11]{kelly:enriched}, we can embed any \V in a complete and cocomplete closed monoidal category $\V'$, by a functor that preserves limits, tensor products, and any relevant colimits, hence also dualizability and traces.
\end{rmk}

\part{Linearity in derivators}
\label{part:more}

In this second part of the paper, we extend the approach to linearity from \autoref{part:moncat} to homotopical situations, in which we must replace colimits by \emph{homotopy} colimits.
(Recall that our motivation for this generalization is a desire to capture the familiar additivity of the Lefschetz number,~\eqref{eq:intro-add2}.)
There are many axiomatic frameworks for homotopy theory, such as model categories and $(\infty,1)$-categories, but the one which we find most convenient is \emph{derivators}, which were inverted by Grothendieck~\cite{grothendieck:derivateurs}, Franke~\cite{franke:triang}, and 
Heller~\cite{heller:htpythys} and studied further by~\cite{m:introder, cisinski:idcm, groth:ptstab}.

\S\ref{sec:der} and \S\ref{sec:eccat} contain the general results and basic examples.
The remaining sections in this part, \S\S\ref{sec:hofin2}-\ref{sec:eccatei}, contain more examples, including the linearity formulas for homotopy finite colimits and EI-categories.
(A reader who is mainly interested in additivity on cofiber sequences, such as for the applications to the Lefschetz number and Reidemeister trace, should feel free to skip these sections.)
To minimize the background required for this part of the paper, we postpone some details and proofs until \autoref{part:formal}.

\section{Linearity in monoidal derivators}
\label{sec:der}

We begin by recalling some of the basic notions of derivator theory; see~\cite{groth:ptstab,gps:stable,gps:additivity} for details.
A \textbf{derivator} is a 2-functor $\sD\colon \cCat\op\to \cCAT$, where \cCat and \cCAT denote the 2-categories of small and large categories, respectively, that satisfies four axioms.
One of these is that for any functor $u\colon A\to B$, the functor
\[u^*\coloneqq \sD(u)\colon \sD(B)\to \sD(A)\]
has a left adjoint $u_!$ and a right adjoint $u_*$.  

We think of the category $\sD(A)$ as the homotopy category of $A$-shaped homotopy coherent diagrams in some homotopy theory that \sD represents, and refer to its objects as \textbf{coherent diagrams}.
With this in mind, we refer to $u^*$ as 
\emph{restriction along} $u$ and think of $u_!$ and $u_*$ as corresponding \emph{(homotopy) Kan extensions}.
Each object $a$ of a category $A$ defines a functor $a\colon \tc\to A$, where $\tc$ denotes the terminal category.
Thus, we have restriction functors $a^*\colon \sD(A)\to\sD(\tc)$ which take $X\in\sD(A)$ to $X_a \coloneqq a^* X$; together these define a functor $\sD(A)\to \sD(\tc)^A$.
We call $\sD(\tc)$ the \textbf{underlying category} of \sD, and the image of a coherent diagram $X\in \sD(A)$ under this functor its \textbf{underlying diagram}.    (This motivated our abuse of notation in \autoref{thm:intro-der} where 
we wrote $X\colon A\to \dV$ rather than the more correct $X\in \dV(A)$.)
Another axiom of a derivator is that the underlying diagram functor is {\bf conservative}, i.e.\ a morphism $X\to Y$ in $\sD(A)$ is an isomorphism just when its image in $\sD(\tc)^A$ is an isomorphism. 

The third axiom of a derivator is that it takes coproducts in \cCat to products in \cCAT, i.e.\ $\sD(\coprod_i A_i) \simeq \prod_i\sD(A_i)$.
The final axiom is more technical and allows us to compute with the functors $u_!$ and $u_*$.
Essentially it says that they are ``pointwise'' Kan extensions: each object $(u_! X)_a$ is a colimit (i.e.\ a left Kan extension to $\tc$) over the comma category $(u/a)$.
We will not make much explicit use of this axiom; instead we will rely on results from \cite{gps:stable, gps:additivity} that build on it.

Now if $\cMONCAT$ denotes the 2-category of monoidal categories and strong monoidal functors, a \textbf{monoidal derivator} is a 2-functor $\sD\colon \cCat\op\to \cMONCAT$ satisfying two further conditions.
The first is that the composite of $\sD$ with the forgetful functor $\cMONCAT\to \cCAT$ is a derivator.
The second is that the monoidal product is \emph{cocontinuous} in each variable, a compatibility condition between the monoidal product and left adjoints $u_!$ \cite[Definition~3.19]{gps:additivity}.
A monoidal derivator is {\bf symmetric} if it lifts to the 2-category of symmetric monoidal categories and symmetric monoidal functors.
Finally, there is also a notion of \emph{closed} monoidal derivator \cite[Definition~8.5]{gps:additivity}, but we will not need to use the details.

If \sD is a monoidal derivator, then we write the tensor product of $\sD(A)$ as $\otimes_A$ and its unit object as $\lS_A$.
When $A=\tc$ we sometimes omit the subscripts.
In particular, $\lS \coloneqq \lS_\tc$ is ``the unit object of \sD''.
Note that since $\pi_A^*\colon \sD(\tc) \to \sD(A)$ is strong monoidal, where $\pi_A\colon A\to \tc$ is the unique functor, we have in particular $\lS_A \cong \pi_A^*(\lS)$ for any $A$; a coherent diagram in the image of $\pi_A^*$ is said to be {\bf constant}.

The simplest examples of derivators are \textbf{represented}, where $y(\sC)(A) \coloneqq \sC^A$ for some complete and cocomplete ordinary category \sC.
If \sC is also a closed monoidal category, then $y(\sC)$ is a closed monoidal derivator.
Most other interesting examples arise from homotopy categories of model categories or $(\infty,1)$-categories, for which we have $\ho(\sC)(A) \coloneqq \sC^A[(\sW^A)^{-1}]$ for some class \sW of weak equivalences in \sC.
If \sC is a monoidal model category, then $\ho(\sC)$ is a closed monoidal derivator.

The \emph{classical homotopy derivator} $\ho(\bTop)$ of spaces under weak homotopy equivalence is \emph{universal}, in the sense that it acts on every other derivator.
We denote this action by $\tens \colon  \ho(\bTop)\times \sD\to\sD$.\label{page:tensor}
It is easiest to define by
\[ |NA| \tens X \coloneqq (\pi_A)_! (\pi_A)^* X \]
for $X\in\sD(\tc)$ and $A\in\cCat$, where $|NA|$ denotes the geometric realization of the nerve of $A$.
Since every space is weak homotopy equivalent to $|NA|$ for some $A$, this suffices to define the action $\tens$.
(It is not obvious that $(\pi_A)_! (\pi_A)^* X$ depends only on the weak homotopy type of $|NA|$; this is a theorem of Heller~\cite{heller:htpythys} and Cisinski~\cite{cisinski:presheaves}.)

\bigskip

For the rest of this section (and until the end of \S\ref{sec:eccat}), let \dV be a closed symmetric monoidal derivator.
We first intend to mimic the construction of the bicategory $\prof(\V)$ from \S\ref{sec:moncat}.
This requires the definition of \emph{coend} in a derivator, which was introduced in~\cite{gps:additivity}.
Let $\tw(A)$ be the {\bf twisted arrow category} of $A$.  The objects of $\tw(A)$ are the morphisms of $A$ 
and the morphisms from $a_1\xto{f_1} b_1$ to $a_2\xto{f_2} b_2$ in $\tw(A)$
are pairs of morphisms $b_1 \xto{h} b_2$ and $a_2 \xto{g} a_1$ such that $f_2 = h f_1 g$.
This category has source and target projections $(s,t)\colon \tw(A) \to A\op\times A$.
Therefore, its opposite has projections $(t\op,s\op)\colon \tw(A)\op \to A\op\times A$.

Now if \dV is a derivator, the \textbf{coend} of $X\in \dV(A\op\times A)$ is defined by
\[ \int^A X \coloneqq (\pi_{\tw(A)\op})_! (t\op,s\op)^* X, \]
where $\pi_{\tw(A)\op}\colon \tw(A)\op\to \tc$ is the unique functor.
(We will use $\pi$ throughout this paper to denote projection maps.)
It was shown in~\cite[Theorem~5.9]{gps:additivity} that we can construct a closed bicategory $\dprof(\dV)$  from a closed symmetric monoidal derivator \dV  as follows:
\begin{itemize}
\item Its objects are small categories.
\item Its hom-category from $A$ to $B$ is $\dV(A\times B\op)$.
\item Its composition functors are the composites
  \begin{align*}\dV(A\times B\op) &\times \dV(B\times C\op)\\
   & \xto{\pi_{B\times C\op}^*\times \pi_{A\times B\op}^*} \dV(A\times B\op\times B\times C\op)\times \dV(A\times B\op\times B\times C\op)\\
   & \xto{\otimes} \dV(A\times B\op\times B\times C\op)\\
   &\xto{\int^B} \dV(A\times C\op). 
   \end{align*}
\item The identity 1-cell of a small category $B$ is
  \begin{equation}
    \lI_B\;=\;(t,s)_! \lS_{\tw(B)} \;\cong\; (t,s)_! (\pi_{\tw(B)})^* \lS_{\tc} \; \in \dV(B\times B\op)
  \end{equation}
     where $\lS_\tc$ is the monoidal unit of $\sD(\tc)$.
\end{itemize}
(In the terminology of~\cite{gps:additivity}, the bicategorical composition is the external-canceling tensor product
  $ \otimes_{[B]} \colon \dV(A\times B\op) \times \dV(B\times C\op) \too \dV(A\times C\op). $)
When $\dV$ is the represented derivator on a complete and cocomplete closed symmetric monoidal category \V, this bicategory can be identified with 
the bicategory $\prof(\V)$ from \S\ref{sec:moncat}. 

As before, we define the \textbf{$\Phi$-weighted colimit} of $X\in\dV(A)$ by $\Phi\in\dV(A\op)$ to be the composite $\colim^\Phi(X) = \Phi\odot X$ in $\dprof(\dV)$, where $\Phi$ and $X$ are regarded as 1-cells $\tc\hto A$ and $A\hto \tc$ in $\dprof(\dV)$, respectively.
The following observation is somewhat less obvious in the derivator case than in the classical one.

\begin{prop}\label{thm:colim-is-colim}
  For any $X\in\dV(A)$, if $\Phi = (\pi_{A\op})^*\lS$ is constant at the unit object, then we have $\colim^{\Phi}(X) \cong \colim(X)$, where $\colim(X)$ denotes the usual derivator colimit $(\pi_A)_!(X)$.
\end{prop}
\begin{proof}
  By~\cite[Corollary~5.8]{gps:additivity}, we have
  \begin{align*}
    \colim^{(\pi_{A\op})^*\lS}(X)
    &= \int^{A} (\pi_{A\op})^*\lS \otimes X\\
    &\cong \int^{\tc} \lS \otimes (\pi_A)_! X\\
    &\cong (\pi_A)_! X.\qedhere
  \end{align*}
\end{proof}

We can now generalize the definitions and results of \S\ref{sec:moncat}.

\begin{defn}\ 
  \begin{itemize}
  \item A coherent diagram $X\in\dV(A)$ is \textbf{pointwise dualizable} if it is right dualizable when regarded as a 1-cell $A\hto \tc$ in $\dprof(\dV)$.
  \item A coherent diagram $\Phi\in\dV(A\op)$ is \textbf{absolute} if it is right dualizable when regarded as a 1-cell $\tc\hto A$ in $\dprof(\dV)$.
  \end{itemize}
\end{defn}

Thus, by \autoref{thm:compose-duals}, we have:

\begin{thm}
  If $X\in\dV(A)$ is pointwise dualizable and $\Phi\in\dV(A\op)$ is absolute, then $\colim^\Phi(X)$ is dualizable.
\end{thm}

We also have a version of \autoref{thm:smcpwdual}.

\begin{lem}[{\cite[Lemma~11.5]{gps:additivity}}]\label{thm:derptdual}
  $X\in\dV(A)$ is pointwise dualizable if and only if each object $X_a\in\dV(\tc)$ is dualizable.
\end{lem}

Our next order of business is to construct a shadow on $\dprof(\dV)$.

\begin{thm}\label{thm:derivshadow}
  The bicategory $\dprof(\dV)$ has a shadow valued in $\dV(\tc)$, defined for $H\in \dprof(\dV)(A,A) = \dV(A\times A\op)$ by
  \begin{equation}
    \sh{H} = \int^A \symm^* H
  \end{equation}
  where $\symm\colon A\times A\op \toiso A\op\times A$ is the symmetry.
\end{thm}
\begin{proof}
  The shadow isomorphism $\sh{H\odot K} \cong \sh{K\odot H}$ is essentially just the Fubini theorem,~\cite[Lemma~5.3]{gps:additivity}.
  The shadow axiom follows from essentially the same argument that proves the associativity of composition in $\dprof(\dV)$,~\cite[Lemma~5.12]{gps:additivity}.
\end{proof}

Thus, by \autoref{thm:compose-traces}, we have a linearity formula.

\begin{thm}\label{thm:dercomposite}
  If $X\in\dV(A)$ is pointwise dualizable and $\Phi\in\dV(A\op)$ is absolute, then for any $f\colon X\to X$ we have
  \[ \tr(\colim^\Phi (f)) = \tr(f) \circ \tr(\id_\Phi). \]
\end{thm}

As before, to make this useful, we need to analyze $\sh{A} = \sh{\lI_A}$ further.
By definition of $\lI_A$ and $\sh{-}$, this is obtained by transporting $\lS_\tc$ diagonally along the following diagram starting from the top right and ending at the 
bottom left.  We restrict horizontally and left Kan extend vertically (ignore the top left corner for the moment):
\begin{equation}\label{eq:dershunit}
  \vcenter{\xymatrix@C=3pc{
    \Lambda A \ar[r] \ar[d] \pullbackcorner & \tw(A) \ar[r] \ar[d]^{(t,s)} & \tc\\
    \tw(A)\op \ar[r]_{(s\op,t\op)} \ar[d] & A\times A\op \\
    \tc
  }}
\end{equation}

Let $\Lambda A$ be defined by the pullback in~\eqref{eq:dershunit}.
Explicitly, the objects of $\Lambda A$ are pairs of morphisms $(a\xto{\alpha} b,b\xto{\beta} a)$ in $A$ which are composable in both orders.
A morphism from $(\alpha,\beta)$ to $(\alpha',\beta')$ is a pair of morphisms $(a\xto{\xi} a', b'\xto{\zeta} b)$ such that $\alpha = \zeta\alpha'\xi$ and $\beta' = \xi\beta\zeta$.
Since $(t,s)$ is a discrete opfibration, the above pullback square is homotopy exact~\cite[Prop.~1.24]{groth:ptstab}, meaning that the Beck-Chevalley condition holds for the vertical left adjoints; thus we have
\begin{align}
  \sh{A}\, &\cong (\pi_{\Lambda A})_! (\pi_{\Lambda A})^* \lS_\tc \notag\\
  &= |N(\Lambda A)| \tens \lS.\label{eq:LA}
\end{align}

\begin{rmk}
In fact, the nerve $N(\Lambda A)$ is equivalent to the \emph{cyclic nerve} $ZA$ of $A$; thus $\sh{A}$ may be regarded as the ``Hochschild homology of $A$ with respect to \dV''.
Recall that the $n$-simplices of $ZA$ are composable loops of $(n+1)$ morphisms in $A$:
\[ a_0 \xto{\alpha_1} a_1 \xto{\alpha_2} \cdots \xto{\alpha_n} a_n \xto{\alpha_{n+1}} a_0, \]
with face and degeneracy maps defined by composition and inserting identities.
There is a map of simplicial sets $N(\Lambda A) \to Z A$ which sends an object $(\alpha,\beta)$ to the composite $\beta\alpha$, and an $n$-simplex
\[ (\alpha_0,\beta_0) \xto{(\xi_1,\zeta_1)} (\alpha_1,\beta_1) \xto{(\xi_2,\zeta_2)}\cdots \xto{(\xi_n,\zeta_n)} (\alpha_n,\beta_n) \]
of $N(\Lambda A)$ to the $n$-simplex
\[ a_0 \xto{\xi_1} a_1 \xto{\xi_2} \cdots \xto{\xi_n} a_n \xto{\beta_0 \zeta_0 \zeta_1 \cdots \zeta_n \alpha_n} a_0 \]
of $Z A$.
This map can be shown to be a weak homotopy equivalence.
However, we will need only the following weaker assertion.
\end{rmk}

\begin{lem}\label{thm:pi0LA}There is a bijection between 
  $\pi_0(\Lambda A)$ and the set of conjugacy classes of $A$.
\end{lem}
\begin{proof}
  Each object $(\alpha,\beta)\in\Lambda A$ induces a conjugacy class $[\alpha\beta] = [\beta\alpha]$, and if $(\xi,\zeta)\colon (\alpha,\beta)\to(\alpha',\beta')$ is a morphism, then
  \[ [\alpha'\beta'] = [\alpha'\xi\beta\zeta] = [\zeta\alpha'\xi\beta] = [\alpha\beta]. \]
  Thus, $\pi_0(\Lambda A)$ maps to the set of conjugacy classes, and the map is clearly surjective.
  For injectivity, first note that for any $(\alpha,\beta)$ we have morphisms $(\alpha,\id)\colon (\alpha,\beta) \to (\id,\alpha\beta)$ and $(\beta,\id)\colon (\alpha\beta,\id) \to (\alpha,\beta)$.
  Thus, it suffices to show that if $[\alpha] = [\beta]$ then $(\id,\alpha)$ and $(\beta,\id)$ are in the same connected component of $\Lambda A$.
  But if $[\alpha] = [\beta]$, then we have $\alpha = \zeta\xi$ and $\beta=\xi\zeta$ for some $\xi$ and $\zeta$, and hence there is a morphism $(\xi,\zeta)\colon (\alpha,\id) \to (\id,\beta)$ in $\Lambda A$.
\end{proof}

In particular, every conjugacy class $[\alpha]$ in $A$ yields a uniquely determined morphism $[\alpha]\colon \lS \to \sh{A}$ in $\dV(\tc)$.
More generally, for any functor $F\colon B\to A$ and any natural transformation $\mu\colon F\to F$, we have a functor $\Lambda\mu \colon  \Lambda B\to \Lambda A$ which sends $(b_1\xto{\beta_1} b_2, b_2\xto{\beta_2} b_1)$ to $(\mu_{b_2}\circ F(\beta_1), F(\beta_2))$ (naturality of $\mu$ makes this functorial).
Thus, we have an induced map $[\mu]\colon\sh{B} \to \sh{A}$.
In the case $B=\tc$, where $\mu$ is just an endomorphism $a\xto{\alpha} a$ in $A$, this yields the above map $[\alpha]\colon\lS \to \sh{A}$.

Now since natural transformations between functors induce homotopies between maps of nerves, if we have a transformation $\Lambda \mu \to \Lambda \nu$, then $[\mu]=[\nu]\colon\sh{B} \to\sh{A}$.
In the case $B=\tc$, this implies that $[\alpha]$ is determined by the connected component of $(\alpha,\id)$ in $\Lambda A$, and hence (by \autoref{thm:pi0LA}) by the conjugacy class of $\alpha$.

There is also a component lemma. 

\begin{lem}[The  component lemma for monoidal derivators]\label{thm:smderomega}
  If $X\in\dV(A)$ is pointwise dualizable and $f\colon X\to X$, then for any conjugacy class $[a\xto{\alpha} a]$ in $A$, the composite
  \begin{equation}
    \xymatrix{ \lS_\tc \ar[r]^-{[\alpha]} & \sh{A} \ar[r]^-{\tr(f)} & \lS_\tc }\label{eq:smderomtr}
  \end{equation}
  is equal to the trace in $\dV(\tc)$ of the composite
  \begin{equation}
    \xymatrix{ X_a \ar[r]^-{X_\alpha} & X_a \ar[r]^-{f_a} & X_a .}\label{eq:smderomres}
  \end{equation}
\end{lem}

\begin{proof}We will prove a generalization of this result in \S\ref{sec:basechangebicat} on page \pageref{thm:compderivbicat}.
\end{proof}

In contrast to the situation in \S\ref{sec:moncat}, $\tr(f)$ may not be determined by its composites with the morphisms $[\alpha]$.
However, for most applications, it is sufficient to know these composites, because the coefficient vector is built out of them.
This is the situation of the following definition.

We say that a derivator \dV is \textbf{semi-additive} if the category $\dV(\tc)$ is so.
This implies that all categories $\dV(A)$ are also semi-additive.
Moreover, if \dV is symmetric monoidal, then $\dV(\tc)(\lS,\lS)$ is a commutative semiring which acts on the homsets of all the $\dV(A)$.

\begin{defn}
  If $\Phi\in\dV(A\op)$ is absolute, then $\tr(\id_\Phi)\colon\lS \to \sh{A}$ is called its \textbf{coefficient vector}.
  If \dV is semi-additive and we can express the coefficient vector of $\Phi$ as a linear combination
  \begin{equation}
    \tr(\id_\Phi) = \sum_{[\alpha]} \phi_{[\alpha]} \cdot [\alpha],\label{eq:philin}
  \end{equation}
  for $\phi_{[\alpha]}\in\dV(\tc)(\lS,\lS)$, then we refer to the $\phi_{[\alpha]}$ as the \textbf{coefficients} of $\Phi$ and say that $\Phi$ has a \textbf{coefficient decomposition}.
\end{defn}

Often $\Lambda A$ is essentially discrete (and finite), so that $\sh{A} \cong \bigoplus_{[\alpha]} \lS$ and thus $\Phi$ automatically has a coefficient decomposition.
Whenever $\Phi$ has a coefficient decomposition, we can give a more familiar description of the formula in \autoref{thm:dercomposite}.

\begin{cor}\label{thm:smderlin2}
If \dV is semi-additive and $\Phi\in\dV(A\op)$ is absolute and has a coefficient decomposition, then 
\[ \tr(\colim^\Phi(f)) = \sum_{[\alpha]} \phi_{[\alpha]} \cdot \tr(f_a \circ X_{\alpha}) \]
for any endomorphism $f$ of a pointwise dualizable $X\in\dV(A)$.
\end{cor}

As before, the problem is now to produce examples of absolute weights and calculate their coefficient vectors.
We easily obtain derivator versions of the simple examples from \S\ref{sec:moncat-examples}: zero objects, direct sums, and tensor products.
For these, there is not much gained by passing to derivators, since they are actual colimits in the underlying category $\dV(\tc)$.
(This follows from the fact that $\dV\colon \cCat\op\to\cCAT$ takes coproducts to products.)

The splitting of idempotents is slightly less trivial, since idempotents in a derivator (that is, objects of $\dV(E)$ where $E$ is the free-living idempotent) are \emph{coherent} idempotents, and not every idempotent in $\dV(\tc)$ may admit a coherentification.
However, coherent idempotents are, in particular, idempotents in $V(\tc)$, and their coherent splitting is, in particular, a splitting in $V(\tc)$.
Thus, again the linearity formula follows from ordinary category theory, not requiring the derivator structure.

In the next four sections (and continuing in \S\ref{sec:derbicat})
we consider some situations where honestly new phenomena occur in the derivator context, primarily arising from \emph{stability}.

\begin{rmk}
  The component lemma for derivators is closely related to~\cite[Theorem 6.3]{PS3}.
  Indeed, a monoidal derivator is a particular kind of indexed monoidal category, for which the base category $\mathbf{S}$ is \cCat.
  It doesn't have \emph{indexed homotopy coproducts} in the sense of~\cite{PS3}, but it does if we restrict its domain to the category $\cGpd$ of groupoids (since comma categories for groupoids coincide with homotopy pullbacks).
  The bicategory constructed as in~\cite[Theorem 5.2]{PS3} from this indexed monoidal category is equivalent to the sub-bicategory of $\prof(\dV)$ whose objects are groupoids (using the facts that when $A$ is a groupoid, we have $A\cong A\op$ and $\tw(A)\simeq A$).

  Thus,~\cite[Theorem 6.1]{PS3} implies that when $A$ is a groupoid, $X\in\dV(A)$ is pointwise dualizable just when it is dualizable in the symmetric monoidal category $\dV(A)$, and in that case~\cite[Theorem 6.3]{PS3} says that from $\tr(f)\colon\sh{A}\to \lS$ we can extract the trace of $f$ in $\dV(A)$.
  Moreover, since for any $a\in A$ the functor $a^*\colon\dV(A)\to \dV(\tc)$ is strong monoidal and thus preserves traces, from the trace of $f$ in $\dV(A)$ we can extract the traces of each $f_a\colon X_a\to X_a$.
  Thus, with a little care, we can deduce the special case of the component lemma when $\alpha=\id_a$ from~\cite[Theorem 6.3]{PS3}.

  The general component lemma, in the case when $\dV$ is the homotopy derivator of chain complexes, is remarked on in~\cite[Example 6.6]{PS3}.
  In fact,~\cite[Theorem 6.3]{PS3} can also be generalized to a statement that includes the full component lemma, but we will not do that here.
\end{rmk}

\section{Stable derivators and additivity}
\label{sec:eccat}

Let \dV be a symmetric monoidal derivator which is \emph{pointed}, i.e. its initial and terminal objects coincide and are denoted by $0$.
Let $\bbtwo$ denote the arrow category $(a\xto{\alpha} b)$, and let $\ulcorner$ be the span category $(c \ot a \to b)$.
Then for any $X\in \dV(\bbtwo)$ with underlying diagram $(X_a \to X_b)$, we can right Kan extend it to an object of $\dV(\ulcorner)$ with underlying diagram $(0\ot X_a \to X_b)$.
The colimit of this coherent span is an object called the \textbf{cofiber} of $X$.
(Sometimes this word refers to the induced map from $X_b$ to this object, perhaps itself lifted to $\dV(\bbtwo)$; but we will care primarily about the object.)

In particular, from any object $X\in \dV(\tc)$ we can obtain (again by right Kan extension) an object of $\dV(\bbtwo)$ with underlying diagram $(X\to 0)$.
Its cofiber is called the \textbf{suspension} of $X$ and denoted $\Sigma X$.
In a represented derivator, the suspension is always $0$, but in general it is nontrivial and deserves its name.

Dually, we have the \textbf{fiber} of any object of $\dV(\bbtwo)$, and the \textbf{loop space} $\Omega X$ of any object $X$.
There is an adjunction $\Sigma \dashv \Omega$, and \dV is \textbf{stable} if and only if this adjunction is an equivalence (see~\cite{gps:stable}).
A stable derivator is automatically additive.

Let us consider the possible absoluteness of these constructions.
We begin by observing that cofibers in any pointed symmetric monoidal derivator are a weighted colimit.
Let $\Phi\in\dV(\bbtwo\op)$ be the essentially unique diagram of the form $(0\ot \lS)$, which we can obtain as before by right Kan extension.
Note that $\tw(\bbtwo)\cong \ulcorner\op$; and for any $X\in \dV(\bbtwo)$, the weighted colimit $\Phi \odot X$ is the pushout of the $\tw(\bbtwo)\op$-diagram
\[ \vcenter{\xymatrix{ \lS\otimes X_a \ar[r] \ar[d] & \lS \otimes X_b \\ 0 \otimes X_a }}
\qquad\text{or equivalently}\qquad
\vcenter{\xymatrix{ X_a \ar[r] \ar[d] & X_b \\ 0. }}
\]
Thus, $\Phi \odot X$ is exactly the cofiber of $X$, as defined above.

Since cofibers are a weighted colimit, we can now ask under what circumstances they are absolute.
To calculate the canonical right dual $\rdual{\Phi}$ of our weight $\Phi$, we start with the external hom $\Phi \rhd \lI_{\bbtwo} \in \dV(\bbtwo\times \bbtwo\times\bbtwo\op)$, restrict to $\tw(\bbtwo) \times \bbtwo$, then right Kan extend to $\bbtwo$.
The resulting $(\tw(\bbtwo) \times \bbtwo)$-diagram looks like
\begin{equation}
  \vcenter{\xymatrix{
      \Phi_a \rhd (\lI_\bbtwo)_{a,a} \ar[r]\ar[d] &
      \Phi_b \rhd (\lI_\bbtwo)_{a,a}\ar@{<-}[r]\ar[d] &
      \Phi_b \rhd (\lI_\bbtwo)_{b,a}\ar[d]\\
      \Phi_a \rhd (\lI_\bbtwo)_{a,b} \ar[r] &
      \Phi_b \rhd (\lI_\bbtwo)_{a,b}\ar@{<-}[r] &
      \Phi_b \rhd (\lI_\bbtwo)_{b,b}
    }}
\end{equation}
or equivalently
\begin{equation}
  \vcenter{\xymatrix{
      0 \ar[r]\ar[d] &
      \lS\ar@{<-}[r]\ar[d] &
      0\ar[d]\\
      0 \ar[r] &
      \lS\ar@{<-}[r] &
      \lS.
    }}
\end{equation}
Therefore, its right Kan extension to $\bbtwo$ looks like $(\Omega \lS \to 0)$; this is $\rdual{\Phi}$.

Now recall from \S\ref{sec:traces} that $\Phi$ is absolute iff the canonical map $U\odot \rdual {\Phi}\to \Phi \rhd U$ is an isomorphism for any $U$.
In our case, $U\odot \rdual{\Phi}$ is the pushout of a $\tw(\bbtwo)\op$-diagram that looks like $U_a \otimes \Omega\lS \ot U_b \otimes \Omega\lS \to 0$, i.e.\ the cofiber of $U_b \otimes \Omega\lS \to U_a \otimes \Omega\lS$.
On the other hand, $\Phi \rhd U$ is the limit of a $\tw(\bbtwo)$-diagram that looks like $(U_b \to U_a \ot 0)$, i.e.\ the \emph{fiber} of $U_b \to U_a$.
So absoluteness of $\Phi$ requires that the canonical map
\begin{equation}
  \cof(U_b \otimes \Omega\lS \to U_a \otimes \Omega\lS) \too \fib(U_b \to U_a)\label{eq:stable-cofib-iso}
\end{equation}
be an equivalence for all $U$; in other words, that we can calculate a fiber as a ``shifted'' cofiber.
This is similar to \autoref{eg:zero} and \autoref{eg:dirsum}: we are asking a colimit construction to coincide with a limit construction.

We also remarked in \S\ref{sec:traces} that it suffices to require this when $U=\Phi$.
The fiber of $\Phi$ is just $\lS$, while the cofiber of $\Phi\otimes \Omega \lS$ is the suspension $\Sigma\Omega\lS$.
Thus, absoluteness of $\Phi$ equivalently  requires that the canonical map
\[ \Sigma \Omega \lS \too \lS \]
be an equivalence.
This map is the counit of the adjunction $\Sigma\dashv\Omega$ at $\lS$; thus we have proven one direction of the following result.

\begin{thm}\label{eg:cofiber}
  A pointed symmetric monoidal derivator \dV is stable if and only if the weight $\Phi$ for cofibers is absolute.
  In this case, if $X\in\dV(\bbtwo)$ is pointwise dualizable, then so is its cofiber $\colim^\Phi(X)$, and for any $f\colon X\to X$ we have
  \begin{equation}
    \tr(\colim^\Phi(f)) = \tr(f_b) - \tr(f_a).\label{eq:cofiber-add}
  \end{equation}
\end{thm}
\begin{proof}
  We have shown ``only if''.
  For ``if'', suppose that $\Phi$ is absolute.
  Since $\otimes$ is cocontinuous, it preserves suspensions in both arguments; thus for any $Y\in\dV(\tc)$ we have
  \[ Y\otimes \Sigma\lS \cong \Sigma(Y\otimes \lS) \cong \Sigma Y. \]
  In particular, we have
  \[ \Omega\lS \otimes \Sigma\lS \cong \Sigma\Omega\lS \cong \lS.\]
  Thus, $\Sigma\lS$ and $\Omega\lS$ are invertible objects of $\dV(\tc)$.
  It follows that the functor $(-\otimes \Sigma \lS)$ is an equivalence with inverse $(-\otimes \Omega\lS)$, and hence so is the functor $\Sigma$.
  Thus, $\dV$ is stable.

  By \autoref{thm:smderlin2}, when this holds we have a linearity formula.
  Note that $\Lambda\bbtwo$ is the discrete category on two objects, so that $\sh{\bbtwo} = \lS \oplus \lS$.
  Thus, $\tr(\id_\Phi)$ is determined by two morphisms $\phi_a,\phi_b\colon \lS \to\lS$, and we have
  \[ \tr(\colim^\Phi (f)) = \phi_a \cdot \tr(f_a) + \phi_b \cdot \tr(f_b).\]
  As in \autoref{eg:dirsum}, now that we know that $\phi_a$ and $\phi_b$ exist, we can calculate them by considering some very special cases.
  On one hand, if $X$ is $\lS_\bbtwo$ (which looks like $(\lS\to \lS)$), then its cofiber is $0$, so if $f$ is the identity, we have
  \[ 0 = \phi_a \cdot 1 + \phi_b \cdot 1.\]
  On the other hand, if $X$ is $(0\to\lS)$, then its cofiber is $\lS$, so if $f$ is again the identity, we have
  \[ 1 = \phi_a \cdot 0 + \phi_b \cdot 1.\]
  Solving these equations, we obtain $\phi_b = 1$ and $\phi_a = -1$.
\end{proof}

If we identify cofiber sequences with distinguished triangles, then~\eqref{eq:cofiber-add} reproduces the additivity formula of~\cite{add,gps:additivity} for traces in symmetric monoidal stable derivators, as a special case of the general linearity formula.
In particular, it yields our motivating example of Lefschetz numbers as follows.
Given $X\into Y$ and $f\colon Y\to Y$ preserving $X$, we apply the suspension spectrum functor $\Sigma^\infty_+$ to get into the stable homotopy category, which is a stable symmetric monoidal derivator.
The trace of $\Sigma^\infty_+ f$ is then the Lefschetz number of $f$, and similarly for $f_X$ and $f/X$; thus~\eqref{eq:cofiber-add} becomes~\eqref{eq:intro-add2}.

\begin{rmk}\label{rmk:additive}
  Amusingly, we can extract from this a quick proof of the fact that closed symmetric monoidal stable derivators are not just semi-additive but additive, i.e.\ their homsets are not just abelian monoids but abelian groups.
  The proof of \autoref{eg:cofiber} requires only semi-additivity, but concludes that there is a morphism $\phi_a\colon \lS\to\lS$ such that $\phi_a + \id_\lS = 0$.
  Thus, tensoring any morphism $X\to Y$ with $\phi_a$ yields an additive inverse.
  This is somewhat shorter than the proof in~\cite[Proposition 4.12 and Corollary 4.14]{groth:ptstab}, but unlike that proof it only applies when \dV is closed symmetric monoidal.
  (In \autoref{thm:semiadditive} we will see that semi-additivity in the monoidal case can be similarly extracted.)
\end{rmk}

If we combine this with the Mayer--Vietoris sequence of~\cite{gps:stable}, we have a corresponding ``inclusion-exclusion'' result for pushouts.

\begin{cor}\label{thm:pushouts}
  Suppose \dV is stable and we have a coherent span $X\in \dV(\mathord\ulcorner)$:
  \[ \xymatrix{ X_{c} & X_{a} \ar[l] \ar[r] & X_{b}} \]
  which is pointwise dualizable.
  Then its pushout $\colim(X)$ is also dualizable.
  Moreover, if $f\colon X\to X$ is an endomorphism, then
  \begin{equation}
    \tr(\colim(f)) = \tr(f_{b}) + \tr(f_{c}) - \tr(f_{a}).\label{eq:pushout-add}
  \end{equation}
\end{cor}

\begin{proof}
  By~\cite[Theorem~6.1]{gps:stable}, from $X$ we can construct $Y\in \dV(\bbtwo)$ with $Y_a \cong X_{a}$ and $Y_b \cong X_{b} \oplus X_{c}$, such that $\colim(Y) \cong \colim(X)$.
  Moreover, the construction in the proof of this is functorial, so that $f\colon X\to X$ induces $g\colon Y\to Y$ with $\colim(g) =\colim(f)$.
  Thus, it suffices to invoke \autoref{eg:cofiber} and \autoref{eg:dirsum}.
\end{proof}

\section{Homotopy finite categories}
\label{sec:hofin2}
\label{sec:fincolim}

The additivity results of the previous section can be significantly generalized, using the fact that all ``finite'' colimits can be constructed from pushouts (and initial objects).
The relevant notion of ``finite'' in the homotopical case is slightly subtle, however.
We say that $A$ is \textbf{strictly homotopy finite} if its nerve $N A$ contains only finitely many nondegenerate simplices, which is equivalent to asking that $A$ is finite, skeletal, and with no nonidentity endomorphisms.
We call $A$ \textbf{homotopy finite} if is equivalent to a strictly homotopy finite category;  for calculating colimits this is just as good.

We begin by proving a general theorem about constructing homotopy finite colimits out of pushouts.
If \D is a derivator and $A$ a small category, we say that a full subcategory $\E\subseteq \D(\tc)$ is \textbf{closed under $A$-colimits} if whenever $X\in\D(A)$ is such that $X_a\in\E$ for all $a\in A$, then also $\colim X \in\E$.

\begin{thm}\label{thm:hofin-constr}
  Let $A$ be a homotopy finite category.
  \begin{enumerate}
  \item If \D is a derivator and $\E\subseteq \D(\tc)$ contains the initial object and is closed under pushouts, then it is closed under $A$-colimits.\label{item:hofin1}
  \item If $F\colon \D\to\E$ is a morphism of derivators that preserves the initial object and pushouts, then it preserves $A$-colimits.
  \end{enumerate}
\end{thm}

In the terminology of~\cite{ak:closure}, this theorem says that in the world of derivators, homotopy finite colimits lie in the \emph{closure} (later authors have preferred \emph{saturation}) of pushouts and initial objects.

The idea of our proof of \autoref{thm:hofin-constr} is to construct colimits over $A$ out of \emph{geometric realizations} of (semi)simplicial \emph{bar constructions}, which in turn can be computed using coproducts and pushouts.
The latter fact is an instance of a more general theorem about colimits over Reedy categories, which we will prove first as a lemma.

Recall that a {\bf Reedy category} is a small category $C$ together with a function $\deg \colon \ob(C) \to \mathbb{N}$ and two subcategories $\rup C$ and $\rdn C$ containing all the objects, such that every nonidentity map in $\rup C$ strictly raises degree, every nonidentity map in $\rdn C$ strictly lowers degree, and every map in $C$ factors uniquely as a map in $\rdn C$ followed by one in $\rup C$.
A good modern reference for Reedy categories is~\cite{rv:reedy}.

The best example to think of is the opposite $\bbDelta\op$ of the simplex category \bbDelta, for which there is only one object $[n]$ of each degree $n$, while $\rup {\bbDelta\op}$ consists of the degeneracy maps and $\rdn{\bbDelta\op}$ of the face maps.
A $\bbDelta\op$-diagram in a category \V is known as a {\bf simplicial object} in \V, and when $\V=\bTop$ is the category of topological spaces it is called a {\bf simplicial space}.
When a simplicial space $X$ is ``good'' (or \emph{proper} or \emph{Reedy cofibrant}), its homotopy colimit is equivalent to its \emph{geometric realization} $|X|$.
The latter can be defined directly as the tensor product of functors $X\otimes_{[\bbDelta\op]} \Delta$, where $\Delta\colon\bbDelta \to \bTop$ takes $[n]$ to the standard $n$-simplex $\Delta^n$.
However, it is well-known in classical algebraic topology (see e.g.~\cite[X.2.4]{ekmm} or~\cite[23.10]{shulman:htpylim}) that $|X|$ can also be constructed as a sequential colimit \emph{filtered by degree}:
\[ |X| = \colim\Big( |X|_0 \into |X|_1 \into |X|_2 \into \dots \Big) \]
in which the spaces $|X|_n$ can be constructed inductively by a pair of pushout squares:
\begin{equation}
  \vcenter{\xymatrix{
      \partial\Delta^n \times s_n X \ar[r]\ar[d] &
      \partial \Delta^n \times X_n\ar[d]\\
      \Delta^n \times s_n X\ar[r] &
      P_n X \ar[r] \ar[d] & X_n \ar[d] \\
      & |X|_{n-1} \ar[r] & |X|_n. }}
\end{equation}
Here $\partial \Delta^n$ is the boundary of $\Delta^n$ and is topologically an $(n-1)$-sphere, while $s_n X$ is the subspace of $X_n$ consisting of the degenerate simplices (those in the image of a degeneracy map from $X_k$ for some $k<n$).
When $X$ is ``good'', this sequential colimit and these pushouts are also homotopy colimits; thus we can view this as a construction of the homotopy colimit of $X$ out of homotopy pushouts and a homotopy sequential colimit, which can thus be expressed entirely in the classical homotopy derivator $\ho(\bTop)$.

The lemma we prove next generalizes this construction to arbitrary derivators \D and arbitrary Reedy categories $C$.
For this, we need analogues of the spaces $\Delta^n$, $\partial \Delta^n$, and $s_n X$.
The space $\Delta^n$ is contractible, so when working up to homotopy we can omit it.
For the other two, we introduce the following notation: if $c\in C$ with $\deg(c)=n$, $\bsl{C}{c}$ is the full subcategory of the slice category $C/c$ whose objects are morphisms that factor through some object of degree strictly less than $n$, or equivalently which are not in $\rdn C$.
Likewise, we have the category $\bsl{c}{C}$ of morphisms with domain $c$ that are not in $\rup C$.

Now if $X\in \D(C)$ and $c\in C$, let $L_c X$ denote the colimit of the restriction of $X$ to $\bsl{C}{c}$.
This is called the {\bf latching object} of $X$ at $c$; it generalizes the space $s_n X$ of degenerate simplices.
Similarly, we let $\partial D^c$ denote the geometric realization of the nerve of $\bsl{c}{C}$;
it generalizes the simplicial $(n-1)$-sphere $\partial \Delta^n$.

\begin{lem}\label{thm:reedy}
  Let $C$ be a Reedy category, \D a derivator, and $X\in\D(C)$.
  Then there is a coherent diagram whose shape is the poset $\mathbb{N}$ and whose colimit is $\colim(X)$:
  \begin{equation}
    \emptyset = \colim_{-1}(X) \to \colim_0(X) \to \colim_1(X) \to \colim_2(X) \to \cdots.\label{eq:colimX}
  \end{equation}
  Moreover, the morphisms $\colim_{n-1}(X) \to \colim_{n}(X)$ appear in cocartesian squares
  \begin{equation}\label{eq:colimnX}
    \vcenter{\xymatrix{
        \displaystyle\coprod_{\deg(c)=n}^{\vphantom{\deg}} P_c X \ar[r]\ar[d] &
        \displaystyle\coprod_{\deg(c)=n}^{\vphantom{\deg}} X_c \ar[d]\\
        \colim_{n-1}(X)\ar[r] &
        \colim_{n}(X)
      }}
  \end{equation}
  while the objects $P_c X$, for $c\in C$, appear in cocartesian squares
  \begin{equation}\label{eq:Pc}
    \vcenter{\xymatrix{
        \partial D^c\tens L_c X\ar[r]\ar[d] &
        \partial D^c\tens X_c\ar[d]\\
        L_c X\ar[r] &
        P_c X.
      }}
  \end{equation}
\end{lem}

In the latter pushout, $\tens$ denotes the tensoring of an arbitrary derivator \D over $\ho(\bTop)$ which we introduced on page~\pageref{page:tensor} in \S\ref{sec:der}.

\begin{proof}
  In this proof we will denote $\ho(\bTop)$ by $\dS$.
  Following the philosophy expressed by~\cite{rv:reedy} as ``it's all in the weights'', we construct the desired data in $\dS(C\times C\op)$ first.
  By~\cite[Observation 6.2]{rv:reedy}, we have a sequence of monomorphisms in $\bSet^{C\times C\op}$ whose colimit is the hom-functor $C(-,-)$:
  \begin{equation}
    \emptyset \into \sk_0 C \into \sk_1 C \into \sk_2 C \into \dots.\label{eq:skC}
  \end{equation}
  Moreover, the inclusions $\sk_{n-1} C\into \sk_n C$ are defined inductively (starting from $\sk_{-1} C = \emptyset$) by pushout diagrams
  \begin{equation}\label{eq:skCn}
  \vcenter{\xymatrix{
      \displaystyle\coprod_{\deg(c)=n}^{\vphantom\deg} Q_c \ar@{^(->}[r]\ar[d] &
      \displaystyle\coprod_{\deg(c)=n}^{\vphantom\deg} C(c,-) \times C(-,c)\ar[d]\\
      \sk_{n-1}C\ar@{^(->}[r] &
      \sk_nC
      }}
  \end{equation}
  while the objects $Q_c\in\bSet^{C\times C\op}$, for $c\in C$, are defined by pushout diagrams
  \begin{equation}\label{eq:Qc}
  \vcenter{\xymatrix{
      \partial C(c,-) \times \partial C(-,c)\ar@{^(->}[r]\ar@{^(->}[d] &
      \partial C(c,-) \times C(-,c)\ar@{^(->}[d]\\
      C(c,-) \times \partial C(-,c) \ar@{^(->}[r] &
      Q_c.
      }}
  \end{equation}
  Here $\partial C(-,c)$ denotes the subfunctor of the representable $C(-,c)$ consisting of morphisms factoring through some object of degree strictly less than $n=\deg(c)$, i.e.\ those morphisms that are not in $\rdn C$.
  Similarly, $\partial C(c,-)$ is the subfunctor of $C(c,-)$ consisting of morphisms that are not in $\rup C$.

  Now since~\eqref{eq:skCn} and~\eqref{eq:Qc} are pushouts along monomorphisms, they are also homotopy pushouts in $\bTop$ (regarding sets as discrete spaces), and hence yield cocartesian squares in \dS, or more precisely in the shifted derivator $\shift\dS{C\times C\op}$.
  Similarly, since~\eqref{eq:skC} consists of monomorphisms, its colimit is a homotopy colimit, hence a colimit in $\shift\dS{C\times C\op}$.

  We now apply the left Kan extension morphism $(\pi_C)_!\colon \shift\dS{C\times C\op} \to \shift\dS{C\op}$ to these diagrams.
  Since $(\pi_C)_!$ is cocontinuous, the images are again colimit diagrams in $\shift\dS{C\op}$.
  To compute them, notice first that $C(c,-) \in \dS(C)$ is the left Kan extension of the unit $\lS_{c/C}$ (which is also the terminal object) along the discrete opfibration $c/C \to C$; thus its colimit is $(\pi_{c/C})_!(\lS_{c/C}) \cong |N(c/C)|$, which is contractible since $c/C$ has an initial object.
  Similarly, $\partial C(c,-)$ is the left Kan extension of the unit along the discrete opfibration $\bsl{c}{C} \to C$, so its colimit is $|N \bsl{c}{C}|$, which we have christened $\partial D^c$.
  Thus,~\eqref{eq:Qc} becomes
  \begin{equation}\label{eq:piQc}
  \vcenter{\xymatrix{
      \partial D^c \times \partial C(-,c) \ar@{^(->}[r]\ar@{^(->}[d] &
      \partial D^c \times C(-,c) \ar@{^(->}[d]\\
      \partial C(-,c)\ar@{^(->}[r] &
      (\pi_C)_!Q_c.
      }}
  \end{equation}
  and~\eqref{eq:skCn} becomes
  \begin{equation}\label{eq:piskCn}
  \vcenter{\xymatrix{
      \displaystyle\coprod_{\deg(c)=n}^{\vphantom\deg} (\pi_C)_! Q_c \ar@{^(->}[r]\ar[d] &
      \displaystyle\coprod_{\deg(c)=n}^{\vphantom\deg} C(-,c)\ar[d]\\
      (\pi_C)_!\sk_{n-1}C\ar@{^(->}[r] &
      (\pi_C)_!\sk_nC.
      }}
  \end{equation}
  while~\eqref{eq:skC} becomes
  \begin{equation}\label{eq:piskC}
    \emptyset \into (\pi_C)_! \sk_0 C \into (\pi_C)_! \sk_1 C \into (\pi_C)_! \sk_2 C \into \dots
  \end{equation}
  whose colimit is the terminal object $\lS_{C\op} \in \shift\dS{C\op}$.

  Now we apply the functor $-\tens_C X$.
  We define $\colim_n(X) = (\pi_C)_! \sk_n C \tens_C X$.
  By essentially the same proof as \autoref{thm:colim-is-colim}, we have $\lS_{C\op}\tens_C X \cong \colim(X)$; thus the sequence~\eqref{eq:piskC} yields~\eqref{eq:colimX} with colimit $\colim(X)$ as desired.
  It remains to extract~\eqref{eq:colimnX} and~\eqref{eq:Pc} from~\eqref{eq:piskCn} and~\eqref{eq:piQc}.

  As before, since $C(-,c) \in \dS(C\op)$ is the left Kan extension of the unit along the discrete opfibration $(C/c)\op \to C\op$, by~\cite[Corollary 5.8]{gps:additivity} the weighted colimit $C(-,c) \tens_C X$ is the ordinary colimit of the restriction of $X$ to $C/c$.
  But $C/c$ has a terminal object $\id_c$, so this colimit is isomorphic to $X_c$.
  Thus, if we define $P_c X = (\pi_C)_! Q_c \tens_C X$, the cocartesian square~\eqref{eq:piskCn} yields~\eqref{eq:colimnX}.

  Finally, $\partial C(-,c)$ is the left Kan extension of the unit along the discrete opfibration $(\bsl{C}{c})\op \to C\op$, so $\partial C(-,c)\tens_C X$ is the ordinary colimit of the restriction of $X$ to $\bsl{C}{c}$.
  In other words, it is the latching object $L_c X$.
  Since $\tens$ associates with the product $\times$ in $\dS$, from~\eqref{eq:piQc} we obtain~\eqref{eq:Pc}.
\end{proof}

\begin{proof}[Proof of \autoref{thm:hofin-constr}]
  Since \D preserves equivalences of categories, we may assume $A$ to be strictly homotopy finite.
  The same construction shows both parts, but we will speak only about~\ref{item:hofin1}; we leave it to the reader to deduce the other in a similar way.
  Thus, let $X\in\D(A)$ be such that $X_a\in\E$ for all $a\in A$; we want to show $\colim X \in\E$.

  Let $B$ be the opposite of the category of nondegenerate simplices in $A$.
  Its objects are strings of composable nonidentity arrows $a_0 \to a_1 \to \cdots \to a_n$ in $A$, and its morphisms are \emph{face maps} generated by composing some of the arrows in a string, plus discarding some from the beginning and the end.
  (Composing nonidentity arrows can never produce an identity arrow, since $A$ is strictly homotopy finite.)

  There is a functor $q\colon B \to A$ sending each string to its first object.
  We claim it is homotopy final; recall that this means $A$-colimits in any derivator can equivalently be calculated as $B$-colimits after restriction along $q$.
  By~\cite[Corollary 3.13]{gps:stable}, it suffices to show that for any $a\in A$, the comma category $(a/q)$ is homotopy contractible, i.e.\ has a contractible nerve.
  The objects of this category are strings of composable arrows $a\to a_0 \to a_1 \to \cdots \to a_n$ in which the first arrow $a\to a_0$ might be an identity.
  The subcategory of $(a/q)$ consisting of those strings in which $a\to a_0$ \emph{is} an identity is coreflective.
  But this subcategory has a terminal object, namely the string $a \xto{\id_a} a$.
  Thus, $(a/q)$ is homotopy contractible, so $q$ is homotopy final.
  It follows that $\colim(X) \cong \colim(q^*X)$.

  Now since $A$ is strictly homotopy finite, there is a maximum length of any string of composable nonidentity arrows in $A$; call that maximum $n$.
  Let $\bbDelta'_n$ be the subcategory of the simplex category $\bbDelta$ containing only the objects $[0]$ through $[n]$ and only the coface maps.
  Then there is a functor $p\colon B\to (\bbDelta'_n)\op$ sending each string to its length.
  Thus, $\colim(q^*X) \cong \colim(p_!q^*X)$.
  Moreover, since $p$ is a discrete opfibration with finite fibers, each object occurring in $p_! q^* X$ is a finite coproduct of objects occurring in $X$.
  Since \E contains the initial object and is closed under pushouts, it is also closed under finite coproducts by~\cite[Cor.~4.11]{gps:stable}. 

  The diagram $p_! q^* X$ looks like a (finite) classical semisimplicial bar construction:
  \begin{equation}
    \xymatrix{
      \displaystyle\bigoplus_{a_0\to\cdots\to a_n}^{\vphantom a} X_{a_0}
      \ar@<1.6mm>[r] \ar@{}[r]|-{\vdots} \ar@<-3mm>[r] &
      \raisebox{-3mm}{$\cdots$}
      \ar@<1.6mm>[r] \ar@{}[r]|-{\vdots} \ar@<-3mm>[r] &
      \displaystyle\bigoplus_{a_0 \to a_1 \to a_2}^{\vphantom a} X_{a_0}
      \ar[r] \ar@<2mm>[r] \ar@<-2mm>[r] &
      \displaystyle\bigoplus_{a_0 \to a_1}^{\vphantom a} X_{a_0} \ar@<1mm>[r] \ar@<-1mm>[r] &
      \displaystyle\bigoplus_a^{\vphantom a} X_a
    }
  \end{equation}
  Therefore, we have reduced the problem to showing that \E is closed under $(\bbDelta'_n)\op$-colimits.
  However, $(\bbDelta'_n)\op$ inherits a Reedy structure from $\bbDelta\op$, with all morphisms being in $\rdn{(\bbDelta'_n)\op}$ and only identities in $\rup{(\bbDelta'_n)\op}$.
  Thus, by \autoref{thm:reedy}, for any $Y\in \dV((\bbDelta'_n)\op)$ we have a sequence
  \begin{equation}
    \emptyset \to \colim_0(Y) \to \colim_1(Y) \to \colim_2(Y) \to \cdots.\label{eq:colimY}
  \end{equation}
  with colimit $\colim(Y)$.
  The cocartesian squares~\eqref{eq:colimnX} and~\eqref{eq:Pc} reduce in the case $k\le n$ to
  \begin{equation}
    \vcenter{\xymatrix{
        P_k Y \ar[r]\ar[d] &
        Y_k \ar[d]\\
        \colim_{k-1}(Y)\ar[r] &
        \colim_{k}(Y)
      }}
    \qquad\text{and}\qquad
    \vcenter{\xymatrix{
        \partial D^n\tens L_k Y\ar[r]\ar[d] &
        \partial D^n\tens Y_k\ar[d]\\
        L_k Y\ar[r] &
        P_k Y.
      }}
  \end{equation}
  When $k>n$, then $(\bbDelta'_n)\op$ has no objects of degree $k$, so we have $\colim_{k-1}(Y) \cong \colim_k(Y)$.
  Thus,~\eqref{eq:colimY} stabilizes at $n$, so that $\colim(Y) \cong \colim_n(Y)$.

  Therefore, it will suffice to show by induction that if each $Y_n\in\E$, then $\colim_k(Y)\in \E$.
  The case $k=-1$ is the assumption that $\emptyset\in \E$.
  Since \E is closed under pushouts, for the induction step it suffices to show that $P_k Y\in \E$.
  And since $\rup{(\bbDelta'_n)\op}$ contains only identities, $L_k Y = \emptyset$, so that $P_k Y \cong \partial D^k \tens Y_k$.

  To complete the proof we show by induction on $k$ that if $Z\in\E$ then $S^k \tens Z\in\E$, where $S^k = \partial D^{k+1}$ is the topological $k$-sphere.
  If $k=0$, then $S^0 \tens Z = Z\oplus Z$, which is in \E since \E is closed under finite coproducts.
  Thus, suppose that $S^k\tens Z\in\E$.
  Then we have a (homotopy) cocartesian square in \dS:
  \begin{equation}
  \vcenter{\xymatrix@-.5pc{
      S^k\ar[r]\ar[d] &
      \pt\ar[d]\\
      \pt\ar[r] &
      S^{k+1}
      }}
  \end{equation}
  Since $\tens$ is cocontinuous in its first variable, and the one-point space $\pt$ is the unit of the monoidal structure on \dS, we have an induced cocartesian square in \dV:
  \begin{equation}
  \vcenter{\xymatrix@-.5pc{
      S^k\tens Z\ar[r]\ar[d] &
      Z\ar[d]\\
      Z\ar[r] &
      S^{k+1} \tens Z
      }}
  \end{equation}
  Thus, $S^{k+1} \tens Z$ is a pushout of objects in \E, hence also in \E.
\end{proof}

We note in passing that a slight modification of the proof of \autoref{thm:hofin-constr} shows the following.  We will use this result in \cite{PS6}.

\begin{thm}\label{thm:colim-constr}\ 
  \begin{enumerate}
  \item If \D is a derivator and $\E\subseteq \D(\tc)$ is closed under pushouts and coproducts, then it is closed under all colimits.\label{item:cc1}
  \item If $F\colon \D\to\E$ is a morphism of derivators that preserves pushouts and coproducts, then it is cocontinuous.\label{item:cc2}
  \end{enumerate}
\end{thm}
\begin{proof}
  As before, we discuss only~\ref{item:cc1}.
  The proof of~\ref{item:cc2} is analogous, except that we also invoke the fact that a morphism of derivators is cocontinuous (i.e.\ preserves left Kan extensions) as soon as it preserves colimits~\cite[Proposition~2.3]{groth:ptstab}.

  Given $A\in\cCat$,  we now let $B$ be the category of \emph{all} simplices in $A$, with only face maps between them.
  As in the previous proof  $q\colon B\to A$ is homotopy final, and we have a discrete opfibration $p\colon B\to (\bbDelta')\op$, where $\bbDelta'$ is the subcategory of $\bbDelta$ containing all the objects but only the coface maps.
  Thus, if $X\in\D(A)$ then $\colim(X) \cong\colim(p_! q^* X)$, and the objects in $p_! q^* X$ are coproducts of those in $X$.
  Thus, it suffices to show that \E is closed under $(\bbDelta')\op$-colimits.

  Now $(\bbDelta')\op$ is again a Reedy category with only identities in $\rup{(\bbDelta')\op}$, so \autoref{thm:reedy} applies.
  The same argument as in the proof of \autoref{thm:hofin-constr} shows that if $Y\in \D((\bbDelta')\op)$ has each $Y_n\in \E$, then each $\colim_n(Y)\in \E$.
  Thus, it remains only to show \E is closed under $\omega$-colimits, where $\omega$ is the ordinal $(0\le 1\le 2\le \dots)$ regarded as a category.

  The idea of this is a standard one: the sequential colimit of $Z_0 \to Z_1 \to Z_2 \to\cdots$ can equivalently be calculated as a coequalizer of two maps $\coprod_{i\ge 0} Z_i \toto \coprod_{i\ge 0} Z_i$, one whose components are identities, and one whose components are the nonidentity maps in the diagram.
  To express this in derivator language, let $\Upsilon$ be the full sub-poset of $\omega\times \omega\op$ whose objects have the form $(n,n)$ or $(n,n+1)$, and let $v\colon \Upsilon\to\omega$ be the first projection.
  Then $v$ is homotopy final, so it suffices to consider $\Upsilon$-colimits.
  Now let $P$ be the category $(1 \toto 0)$, and define $u\colon \Upsilon\to P$ by $u(a,b) = b-a$, with morphisms of the form $(n,n+1) \to (n,n)$ going to one of the parallel arrows in $P$, and morphisms of the form $(n,n+1) \to (n+1,n+1)$ going to the other.
  Then $u$ is a discrete opfibration, so $u_!$ computes the coproduct of each fiber; thus it suffices to consider $P$-colimits, i.e.\ coequalizers.
  But these are finite, so we can apply \autoref{thm:hofin-constr}.
  (It is also easy to construct coequalizers explicitly out of pushouts in essentially the usual way.)
\end{proof}

Now we can prove a linearity formula for homotopy finite colimits.
Note that if $A$ is strictly homotopy finite, then $\Lambda A$ is just the discrete set $A_0$ of objects of $A$, so that $\sh{A} = A_0 \cdot \lS$.

\begin{thm}\label{thm:eccat}
  Let $A$ be homotopy finite and let \dV be a stable, closed symmetric monoidal derivator.
  If $X\in \dV(A)$ is pointwise right dualizable, then $\colim(X)$ is right dualizable, and for any $f\colon X\to X$ we have
  \begin{equation}\label{eq:eccattr}
    \tr(\colim(f))
    = \sum_{a} \tr(f_a) \cdot \sum_{k\ge 0} (-1)^k \cdot \card{\left\{
      \parbox{5.2cm}{\centering composable strings of nonidentity\\arrows of length $k$ starting at $a$}
    \right\}}
  \end{equation}
\end{thm}
\begin{proof}
  In \autoref{thm:hofin-constr}\ref{item:hofin1}, let $\E\subseteq\dV(\tc)$ be the full subcategory of dualizable objects.
  This obviously contains the zero object, and is closed under pushouts by \autoref{thm:pushouts}.
  Thus, if $X\in \dV(A)$ is pointwise right dualizable, then $\colim(X)$ is right dualizable.

  The inductive arguments in the proof of \autoref{thm:hofin-constr}, together with \autoref{thm:pushouts}, show that for any pointwise dualizable truncated semisimplicial diagram $Y$ and any $g\colon Y\to Y$, we have
  \begin{align}
    \tr(|g|_{k+1})
    &=
    \tr(|g|_{k}) +
    \tr(g_{k+1}) -
    \tr(S^k \tens g_{k+1}).
  \end{align}
  Moreover, for any dualizable $Z$ and $h\colon Z\to Z$ we have
  \[ \tr(S^{k+1}\tens h) = 2\tr(h) - \tr(S^k\tens h).\]
  Since $S^0 \tens Z = Z\oplus Z$, we have $\tr(S^0 \tens h) = 2\tr(h)$, so by induction we have
  \[ \tr(S^k\tens h) =
  \begin{cases}
    2\tr(h) & k\text{ is even}\\
    0 & k\text{ is odd}.
  \end{cases}
  \]
  Thus, for $Y$ and $g$ as above, another induction yields
  \begin{equation}
    \tr(\colim(g)) = \sum_{0\le k\le n+1} (-1)^k \tr(g_k).\label{eq:realiztr}
  \end{equation}
  In \autoref{thm:hofin-constr}   $\colim(X)$ for $X\in\dV(A)$  is constructed as the colimit of a truncated semisimplicial diagram $Y$ where $Y_k$ is the the coproduct of the objects $X_a$ over strings of composable nonidentity morphisms (nondegenerate simplices) in $A$.
  Since the projection from the opposite category of nondegenerate simplices to $A$ selects the first object in a composable string, each simplex $(a_0 \to a_1 \to \cdots \to a_k)$ contributes a summand of $X_{a_0}$, so that we have
  \[ Y_k = \bigoplus_a \bigoplus_{\left\{\parbox{4.3cm}{\scriptsize\centering composable strings of nonidentity\\arrows of length $k$ starting at $a$}\right\}} X_a. \]
  Combining this with~\eqref{eq:realiztr} and the simple additivity formula for coproducts, and rearranging summations, yields~\eqref{eq:eccattr}.
\end{proof}

Note that \autoref{thm:pushouts} is also an instance of this formula, where $A$ is the category $\{b \leftarrow a \to c\}$.
Starting at each object there is one composable string of length 0, contributing $+1$, while the object $a$ also has two strings of length 1, contributing $-2$; thus $\phi_b = \phi_c = 1$ while $\phi_a=-1$.

We end this section by comparing the linearity formula of \autoref{thm:eccat} to the formula for the cardinality of a colimit in~\cite[\S3]{leinster:eccat}.
The latter depends on the notion of a \textbf{weighting} on a finite category $A$ (\cite[Def.~1.10]{leinster:eccat}), which is a function $k^\bullet$ from the objects of $A$ to some ring such that for all $a\in A$,
\begin{equation}
  \sum_{b\in A} \card{(A(a,b))}\cdot k^b = 1.\label{eq:weighting}
\end{equation}
The linearity formula of~\cite{leinster:eccat} is:

\begin{prop}\cite[Proposition~3.1]{leinster:eccat}\label{leinster}
Let $A$ be a finite category and $k^\bullet$ a $\mathbb{Q}$-weighting on $A$.  If $X\colon A\to \bSet$ is finite and a sum of representables
then $\card{\colim X}=\sum_a k^a \card{X(a)}$.
\end{prop}

In fact, from this proposition,~\cite[Corollary~1.5]{leinster:eccat}, and the comments after~\cite[Definition~1.10]{leinster:eccat}, we obtain exactly  the formula in \autoref{thm:eccat} in the special case when $X$ is a coproduct of representables.
This restriction on $X$ was necessary in~\cite{leinster:eccat} essentially since \bSet is not stable, so that its colimits are not ``homotopy correct''.
Passing to the stable case allows the formula to hold for all $X$.

On the other hand, \autoref{leinster}  is stated for any finite category with a weighting, rather than merely the homotopy finite ones.
The following proposition says that in fairly general circumstances, if $A$-shaped colimits admit a linearity formula in our sense, then $A$ has an induced weighting.

\begin{prop}\label{thm:weighting}
  Suppose $A$ is a finite category such that $\lS_{A\op}$ is absolute in some derivator \dV and each monoid $A(b,b)$ acts freely on each homset $A(a,b)$.
  Then the components of the coefficient vector $\tr(\id_{\lS_{A\op}})$ at the conjugacy classes of identities form a weighting on $A$.
\end{prop}

\begin{proof}
  For any $a$, consider the profunctor $Y^a = (\id\times a)^* \lI_A \in \dprof(\dV)(A,\tc)$.
  Since $(Y^a)_b \cong \bigoplus_{A(a,b)} \lS$ and $A$ is finite, $Y^a$ is pointwise dualizable by \autoref{thm:derptdual}.
  Moreover,  by \autoref{thm:smderomega}, we have
  \[ \tr(\id_{Y^a}) \colon \sh{A}\cong \bigoplus_{\pi_0(\Lambda A)} \lS \too \lS \]
  has a component at $\beta\colon b\to b$ equal to the trace of the action of $\beta$ on $\bigoplus_{A(a,b)} \lS$, which is just the number of fixed points of the action of $\beta$ on the homset $A(a,b)$.
  By assumption, this is zero unless $\beta=\id_b$, in which case it is $\card {(A(a,b))}$.

  Since $\lS_{A\op}$ is absolute, $\lS_{A\op}\otimes_{[A]} Y^a$ is dualizable.
  Then by \autoref{thm:dercomposite}, its Euler characteristic can be expressed as the sum on the left-hand side of~\eqref{eq:weighting} 
  where $k^b$ denotes the $b^{\mathrm{th}}$ component of $\chi(\lS_{A\op})$.
  We also have  $\lS_{A\op}\otimes_{[A]} Y^a\cong a^*\lS_{A\op} \cong \lS$ and $\chi(\lS) = 1$ completing the proof.
\end{proof}

The proof of \autoref{thm:weighting} also implies that under its hypotheses, our linearity formula for $A$-shaped colimits reduces 
to~\autoref{leinster} when $X$ is a coproduct of representables.
However, for general $X$, the coefficients at nonidentity endomorphisms can matter.

For instance, consider the case when $A=\bg$ is a finite group $G$ regarded as a one-object groupoid.
We have seen in \autoref{thm:burnside} that $\bg$-colimits are absolute in any $\card G$-divisible \V, and the coefficient at a conjugacy class $C\subseteq G$ is given by $\frac{\card C}{\card G}$.
(We will generalize this result to derivators in \S\ref{sec:deriv-burnside}.)
Thus in this case, $\bg$ satisfies the hypotheses of \autoref{thm:weighting} for the represented derivator $y(\V)$, yielding the weighting consisting of the single number $\frac1{\card G}$, as in~\cite[Example~1.11(b)]{leinster:eccat}.
However, if $X$ is \lS with the trivial $G$-action, then all components of $\chi(X)$ are equal to 1, so that our linearity formula gives
\begin{equation}
  \chi(\colim(X)) = \sum_{C} \frac{\card C}{\card G} = 1.\label{eq:trivial-action}
\end{equation}
\autoref{leinster} does not apply in this case, since this $X$ is not a coproduct of representables.
However,~\eqref{eq:trivial-action} is nevertheless correct even as a formula for the cardinality of the colimit of a trivial action on a one-element set.

On the other hand, if $\lS_{A\op}$ is absolute but the endomorphism monoids do not act freely, then our linearity formula can differ from~\autoref{leinster} even for coproducts of representables.
For instance, suppose $A$ is the free-living idempotent, with one object $\star$ and one nonidentity idempotent $e\in A(\star,\star)$.
Then as we have seen, $\lS_{A\op}$ is a retract of a representable, hence absolute, and its coefficient vector has components $0$ and $1$ at $\id_\star$ and $e$ respectively.
Thus, if $X\in\dV(A)$ is an object equipped with a coherent idempotent, our linearity formula says that $\chi(\colim(X))$ is the trace of the action of this idempotent, as in \S\ref{sec:moncat}.

However, the weighting assigned to $A$ by~\cite{leinster:eccat} is the single number $\frac12$.
And when $X\in\dV(A)$ is a coproduct of $n$ representables, its underlying object is a coproduct of $2n$ copies of \lS, hence has Euler characteristic $2n$.
Thus~\autoref{leinster} says that the Euler characteristic of $\colim(X)$ is $\frac12(2n) = n$.
Our formula also gives the correct answer in this case, since the matrix of the action of the idempotent on a coproduct of $n$ representables is block diagonal with $n$ blocks of the form $\begin{pmatrix}0&0\\1&1\end{pmatrix}$, hence has trace $n$.
But it is a different formula.

\section{The orbit-counting theorem}
\label{sec:deriv-burnside}

In this section we generalize the orbit-counting theorem, \autoref{thm:burnside}, to derivators.
Let $G$ be a finite group and let $\bg$ denote the 1-object category associated to $G$.
We will also denote the set $G$ regarded as a discrete category by $G$.

The main ingredient in our proof is the following theorem, which essentially reduces the derivator case to the ordinary case.
We say that a derivator \dV is \textbf{$n$-divisible} if $\dV(\tc)$ is $n$-divisible; this implies that any morphism in any category $\dV(A)$ can be ``divided by $n$''.

\begin{thm}\label{thm:dgm-ff}
  If \dV is a semi-additive and $\card G$-divisible derivator, then the underlying diagram functor
  \[  \dV(\bg) \to \dV(\tc)^\bg \]
  is fully faithful.
\end{thm}

We can interpret this theorem as the observation that  being a coherent \bg-diagram is a mere \emph{property} of an incoherent one.
To our knowledge, this theorem first appeared explicitly in~\cite{souza:traces}, although it ought to be well-known.

To begin with, note that there is a canonical natural transformation
\begin{equation}\label{eq:gbgcomma}
  \vcenter{\xymatrix{
      G\ar[r]^p\ar[d]_p \drtwocell\omit &
      \tc\ar[d]^u\\
      \tc\ar[r]_u &
      \bg
      }}
\end{equation}
whose component at $g\in G$ is just $g$ regarded as a morphism in \bg.
Since this is a comma square, the fourth defining property of a derivator~\cite[Definition~2.1]{gps:stable} yields isomorphisms $p_!p^* \cong u^*u_!$ and $p_*p^*\cong u^*u_*$.
Moreover, since $G$ is discrete, we have $p_! Z \cong \sum_{g\in G} Z_g$ and $p_* Z \cong \prod_{g\in G} Z_g$.

The underlying diagram functor $\dV(\bg) \to \dV(\tc)^\bg$ associates to any object of $\dV(\bg)$ an object in $\dV(\tc)$ with a $G$-action.
In particular, for any $Y\in\dV(\tc)$, the objects $u^* u_! Y$ and $u^* u_* Y$ of $\dV(\tc)$ have canonical $G$-actions.

\begin{lem}\label{thm:underlying-action}
  In any derivator \dV, the induced $G$-action on $u^* u_! Y \cong \sum_{g\in G} Y$ permutes the summands by left translation.
  Similarly, the $G$-action on $u^* u_* Y \cong \prod_{g\in G} Y$ permutes the factors by right translation.
\end{lem}
\begin{proof}
  For any $g\in G$, we have a natural transformation $g\colon u\to u$ with unique component $g$, and the action of $g$ on $u^* X$ for $X\in\dV(\bg)$ is the image of this transformation under the 2-functor $\dV$.
  Thus, for $Y\in\dV(\tc)$, the induced isomorphism
  \[ p_! p^* Y \toiso u^* u_! Y \xto{g} u^* u_! Y \]
  is the mate-transformation associated to the pasted rectangle on the left below.
  \begin{equation}
    \vcenter{\xymatrix{
        G\ar[r]^p\ar[d]_p \drtwocell\omit &
        \tc\ar[d]^u\\
        \tc\ar[r]^u\ar@{=}[d] \drtwocell\omit{g} &
        \bg \ar@{=}[d]\\
        \tc \ar[r]_u &
        \bg
      }}
    \qquad=\qquad
    \vcenter{\xymatrix{
        G \ar[r] \ar[d]_{g\cdot -} &
        \tc \ar[d] \\
        G\ar[r]^p\ar[d]_p \drtwocell\omit &
        \tc\ar[d]^u\\
        \tc\ar[r]_u &
        \bg
      }}
  \end{equation}
  However, this is equal to the pasted rectangle on the right.
  Thus, after composing again with the mate $p_! p^* \cong u^* u_!$ associated to the lower square on the right, it cancels, leaving only the left translation by $g$.
  The proof for $u_*$ is dual.
\end{proof}

For the rest of this section, we suppose that $G$ is finite and that \dV is a semi-additive derivator.

\begin{prop}\label{thm:freecofreeiso}
  When $G$ is finite and \dV is semi-additive, there is a canonical isomorphism $u_! \toiso u_*$.
\end{prop}
\begin{proof}
  First of all, note that for any $Y\in \dV(\tc)$ we have an isomorphism
  \begin{align}
    \dV(\bg)(u_!Y,u_*Y) &\cong \dV(\tc)(u^* u_! Y, Y)\\
    &\cong \dV(\tc)(p_! p^* Y, Y)\\
    &\cong \dV(\tc)(\textstyle\sum_{g\in G} Y, Y)\\
    &\cong \prod_{g\in G} \dV(\tc)(Y,Y).
  \end{align}
  Let $\phi\colon u_!Y \to u_*Y$ be the morphism that corresponds under this isomorphism to the family of morphisms $(\phi_g\colon Y\to Y)_{g\in G}$ defined by
  \begin{equation}
    \phi_g =
    \begin{cases}
      \id_Y & \quad g=e\\
      0 &\quad g\neq e,
    \end{cases}
  \end{equation}
  where $e\in G$ is the identity element.
  We will show that $\phi$ is an isomorphism.
  It suffices to show that $u^*\phi$ is an isomorphism.
  
  By a similar argument, we have
  \begin{align}
    \dV(\tc)(u^*u_!Y,u^*u_*Y)
    &\cong \dV(\tc)(p_! p^* Y,p_* p^* Y)\\
    &\cong \dV(\tc)\big(\textstyle\sum_{g\in G} Y, \prod_{h\in G} Y\big)\\
    &\cong \prod_{g,h\in G} \dV(\tc)(Y,Y).
  \end{align}
  We have written $\sum$ and $\prod$ for clarity in this isomorphism, but since \dV is semi-additive, finite sums and products are isomorphic and may be written as direct sums $\bigoplus$.
  With this identification, isomorphisms such as
  \begin{equation}
    \dV(\tc)\big({\textstyle\bigoplus_{i} Z_i, \bigoplus_{j} W_j}\big) \cong \prod_{i,j} \dV(\tc)(Z_i,W_j)
  \end{equation}
  allow us to identify composition of morphisms between finite direct sums with matrix multiplication.

  We will show that $u^*\phi$ is a permutation matrix, and hence invertible.
  Note that for any $g\in G$, we have a 2-cell
  \begin{equation}\label{eq:gsq}
  \vcenter{\xymatrix{
      \tc \ar@{=}[r]\ar@{=}[d] \drtwocell\omit{g} &
      \tc\ar[d]^u\\
      \tc\ar[r]_-u &
      \bg
      }}
  \end{equation}
  This gives rise to mates $\gchk\colon Y \to u^*u_! Y$ and $\ghat\colon u^* u_* Y \to Y$.
  Moreover, since~\eqref{eq:gsq} factors through the square~\eqref{eq:gbgcomma} by the morphism $g\colon \tc \to G$, if we make the identifications
  \begin{equation}
    \textstyle
    u^*u_! Y \cong p_! p^* Y \cong \bigoplus_{g\in G} Y
    \qquad\text{and}\qquad
    u^*u_* Y \cong p_* p^* Y \cong \bigoplus_{g\in G} Y
  \end{equation} 
  then \gchk and \ghat correspond respectively to the inclusion of the $g^{\mathrm{th}}$ summand and the projection onto the $g^{\mathrm{th}}$ factor.
  Therefore, the $(g,h)$-component of $u^* \phi$ can be obtained as the composite
  \begin{equation}\label{eq:uphigh}
  \vcenter{\xymatrix{
      Y\ar[r]^-{\gchk} &
      u^*u_! Y\ar[r]^-{u^*\phi} &
      u^*u_* Y\ar[r]^-{\hhat} &
      Y.
      }}
  \end{equation}
  Similarly, the $g$-component of $\phi$ can be obtained as the composite
  \begin{equation}
  \vcenter{\xymatrix{
      Y\ar[r]^-{\gchk} &
      u^*u_! Y\ar[r]^-{\phibar} &
      Y
      }}
  \end{equation}
  where $\phibar$ is the map corresponding to $\phi$ under the adjunction $u^*\dashv u_*$.
  However, since $\phibar$ is by definition the composite $u^*u_! Y \xto{u^* \phi} u^* u_* Y \xto{\epsilon} Y$, where $\epsilon$ is the counit of the adjunction $u^*\dashv u_*$, we can also express the $g$-component of $\phi$ as
  \begin{equation}\label{eq:phig}
    \vcenter{\xymatrix{
        Y\ar[r]^-{\gchk} &
        u^*u_! Y\ar[r]^-{u^*\phi} &
        u^*u_* Y\ar[r]^-{\epsilon} &
        Y.
      }}
  \end{equation}
  Now note that $\epsilon$ is in fact $\ehat$ (recall that $e\in G$ is the identity element), since when $g=e$ the square~\eqref{eq:gsq} commutes.
  Thus, comparing~\eqref{eq:uphigh} and~\eqref{eq:phig}, we obtain an immediate identification $(u^*\phi)_{g,e} = \phi_g$.

  Now observe that for any $g,h\in G$, we have
  \begin{equation}
  \vcenter{\xymatrix{
      \tc\ar@{=}[r]\ar@{=}[d] \drtwocell\omit{h} &
      \tc\ar@{=}[r]\ar[d]^u \drtwocell\omit{g} &
      \tc\ar[d]^u\\
      \tc\ar[r]_-u &
      \bg\ar@{=}[r] &
      \bg
    }}
  \qquad=\qquad
  \vcenter{\xymatrix{
      \tc\ar@{=}[r]\ar@{=}[d] \drtwocell\omit{hg} &
      \tc \ar[d]^u\\
      \tc \ar[r]_-u &
      \bg.
      }}
  \end{equation}
  Thus, by the functoriality of mates, the composite
  \begin{equation}
    \xymatrix{u^* u_* Y \ar[r]^-{g^*} &
      u^* u_* Y \ar[r]^-{\hhat} &
      Y}
  \end{equation}
  is equal to $\widehat{hg}$, where $g^*$ denotes the image of the 2-cell $g\colon u\to u$ under the 2-functor \dV.
  Similarly, the composite
  \begin{equation}
    \vcenter{\xymatrix{
        Y\ar[r]^-{\gchk} &
        u^* u_! Y\ar[r]^-{h^*} &
        u^* u_! Y
      }}
  \end{equation}
  is equal to $\widecheck{hg}$.
  Now we have the commutative diagram
  \begin{equation}
  \vcenter{\xymatrix{
      Y\ar[r]^-{\gchk} \ar@{=}[d] &
      u^*u_! Y\ar[r]^-{u^*\phi} \ar[d]^{h^*} &
      u^*u_* Y\ar[r]^-{\hhat} \ar[d]^{h^*} &
      Y \ar@{=}[d]\\
      Y\ar[r]_-{\widecheck{hg}} &
      u^*u_! Y\ar[r]_-{u^*\phi} &
      u^*u_* Y\ar[r]_-{\epsilon} &
      Y
      }}
  \end{equation}
  which shows that
  \[(u^*\phi)_{g,h} = \phi_{hg} =
  \begin{cases}
    \id_Y &\quad h = g^{-1}\\
    0 &\quad \text{otherwise}.
  \end{cases}
  \]
  This is a permutation matrix, as claimed.
\end{proof}

\begin{lem}\label{thm:dgm-ff-lemma}
  For any $X\in\dV(\bg)$, after applying $u^*$ the composite
  \begin{equation}
    X \xto{\eta} u_* u^* X \otiso u_! u^* X \xto{\epsilon} X\label{eq:diaretr}
  \end{equation}
  is multiplication by $\card  G$.
\end{lem}

\begin{proof}
  The proof of \autoref{thm:freecofreeiso} also showed that for any $Y\in\dV(\tc)$, the map $\epsilon_Y \colon  u^* u_* Y \to Y$ is $\widehat{e}$, which (under the identifications of $u^* u_*$ and $u^* u_!$ with $\card G$-ary direct sums) corresponds to projection onto the $e^{\mathrm{th}}$ factor.
  In particular, this is the case for $\epsilon_{u^* X} \colon  u^* u_* u^* X \to u^* X$.
  Since $u^* \eta_X \colon  u^*X  \to u^* u_* u^* X$ is a section of $\epsilon_{u^* X}$, its component onto the $e^{\mathrm{th}}$ factor must be the identity.
  However, $u^* \eta_X$ is also an equivariant map of $G$-objects in $\dV(\tc)$, and by \autoref{thm:underlying-action} $G$ acts by right translation on $u^* u_* u^* X$.
  Therefore, the component of $u^* \eta_X$ onto the $g^{\mathrm{th}}$ factor must be the action of $g$.

  Similarly, we can show that the component of $u^* \epsilon_X$ out of the $g^{\mathrm{th}}$ summand must be the action of $g$.
  Moreover, we have seen in the proof of \autoref{thm:freecofreeiso} that the middle factor $u^* u_! u^* X \toiso u^* u_* u^* X$ is a permutation matrix which sends the $g^{\mathrm{th}}$ summand to the $(g^{-1})^{\mathrm{th}}$ summand.
  Therefore, $u^*$ of the composite~\eqref{eq:diaretr} is the sum
  \[ \sum_{g,h\in G} h^{-1} \circ \phi_{hg}\circ g \]
  with $\phi_{hg}$ as in \autoref{thm:freecofreeiso}.
  But this is equally
  \[ \sum_{g\in G}( g^{-1} \circ g) = (\card G) \cdot \id.  \qedhere \]
\end{proof}

\begin{proof}[Proof of \autoref{thm:dgm-ff}]
  Since $\card G$ is invertible by assumption, \autoref{thm:dgm-ff-lemma} implies the composite
  \begin{equation}
    u^*X \xto{u^*\eta} u^*u_* u^* X \otiso u^*u_! u^* X \xto{u^*\epsilon} u^*X 
  \end{equation}
  is invertible.  
  Hence so is~\eqref{eq:diaretr}.
 
  In particular, every $X\in\dV(\bg)$ is a natural retract of $u_! u^* X$.
  Therefore, if $X,X'\in\dV(\bg)$, the homset $\dV(\bg)(X,X')$ is a retract of $\dV(\bg)(u_! u^* X,X')$.
  Moreover, this retraction is preserved by the underlying diagram functor $\dV(\bg) \to \dV(\tc)^\bg$.
  Thus, it will suffice to show that
  \begin{equation}
    \dV(\bg)(u_! u^* X, X') \too  \dV(\tc)^\bg(u^* u_! u^* X, u^* X')\label{eq:dgmofretr}
  \end{equation}
  is an isomorphism.
  However, by adjunction we have
  \[\dV(\bg)(u_! u^* X, X') \cong \dV(\tc)(u^* X, u^* X'),\]
  and by \autoref{thm:underlying-action}, $u^* u_! u^* X \cong \bigoplus_{g\in G} u^* X$ is the free $G$-object on $u^* X$ in $\dV(\tc)$.
  Thus, the domain and codomain of~\eqref{eq:dgmofretr} are isomorphic, and it is straightforward to check that this is indeed the desired map.
\end{proof}

\begin{cor}\label{thm:divbicat}
  If \dV is a semi-additive derivator, then the following bicategories with shadows are equivalent:
  \begin{enumerate}
  \item The full sub-bicategory of $\prof(\dV)$, as constructed in \S\ref{sec:der}, whose objects are groupoids $A$ such that \dV is $\card{\nAut_A(a)}$-divisible for each $a\in A$.\label{item:divbi1}
  \item The locally full sub-bicategory of $\prof(\dV(\tc))$, as constructed in \S\ref{sec:traces}, with the same objects as above, and whose 1-cells are the underlying incoherent diagrams of coherent ones.\label{item:divbi2}
  \end{enumerate}
  More precisely, since $\dV(\tc)$ is not cocomplete (though it does have some colimits, like coproducts), the claim is that it does have the necessary colimits in order to define the bicategory~\ref{item:divbi2} by the same formulas that would be used in $\prof(\dV(\tc))$.
\end{cor}
\begin{proof}
  The bicategories in question have the same objects, and by \autoref{thm:dgm-ff} they have equivalent hom-categories.
  Thus, it suffices to show that their composition and identities are the same; or, more precisely, that the (homotopy) colimits and left Kan extensions used in defining~\ref{item:divbi1} are also ordinary colimits and Kan extensions in $\dV(\tc)$.

To this end, if $U$ denotes the underlying diagram functor, then \autoref{thm:dgm-ff} implies that whenever $\dV$ is $\card G$-divisible, we have
\[\dV(\tc)(\pi_!X, Y)\cong \dV(\bg)(X,\pi^*(Y))\cong \dV(\tc)^{\bg}(UX, U\pi^*Y). \] 
Note that $U\pi^* Y$ is the ordinary constant \bg-shaped diagram on $Y\in\dV(\tc)$.
Therefore, this composite isomorphism exhibits $\pi_! X$ as a colimit of $UX$ in $\dV(\tc)$.

In other words, although $\dV(\tc)$ is not cocomplete, it \emph{does} have colimits of those $\bg$-shaped diagrams that underlie some coherent diagram, and these colimits agree with the corresponding homotopy colimits in \dV.
More generally, $\dV(\tc)$ has colimits of coherent $A$-diagrams whenever $A$ is a groupoid such that \dV is $\card{\nAut_A(a)}$-divisible for each $a\in A$, since such colimits are just colimits over the groups $\bB\nAut_A(a)$ followed by a coproduct.
More generally still, it has left Kan extensions along Grothendieck opfibrations whose fibers are groupoids of this sort, and all of these agree with the corresponding homotopy versions in \dV.
However, all the colimits and Kan extensions used in defining the bicategory~\ref{item:divbi2} and its shadow are of this sort (note in particular that $\tw(\bg)$ is equivalent to $\bg$ itself).
This completes the proof.
\end{proof}

Now we can extract the linearity formula and obtain a generalized orbit-counting theorem.

\begin{thm} \label{thm:derivatorburnside}
  Suppose $G$ is a finite group and \dV is a stable and $\card G$-divisible symmetric monoidal derivator.
  Then the constant diagram $(\pi_{\bg\op})^*\lS$ is absolute, and its coefficient vector
  \[\tr(\id_{(\pi_{\bg\op})^*\lS}) \colon \lS \too \sh{\bg} \cong \bigoplus_{C} \lS \]
  has components $\phi_C =\frac{\card C}{\card G}$, where $C$ is the set of conjugacy classes of $G$.
  Therefore, if $X\in \dV(\bg)$ is pointwise right dualizable and $f\colon X\to X$, then
  \begin{equation}
    \tr(\colim(f)) = \sum_{C} \frac{\card C}{\card G}\tr(X_C \circ f),
  \end{equation}
  where $X_C$ denotes the action of any element of $C$ on $X$.
\end{thm}
\begin{proof}
  As we have seen, the underlying diagram of $u_!\lS$ is $\bigoplus_{g\in G} \lS \in\dV(\tc)^{\bg}$ with the right $G$-action that permutes the summands.
  Similarly, the underlying diagram of $(u\op)_!\lS$ is the same coproduct with the corresponding left $G$-action.
  In \autoref{thm:burnside} we observed that these two underlying diagrams are the representables $\bg(\id,a)$ and $\bg(a,\id)$ in $\prof(\dV(\tc))$, and hence dual to each other.
  Therefore, by \autoref{thm:divbicat}, they are also dual in $\prof(\dV)$.

  Similarly, the underlying diagram of $\pi^*\lS \in\dV(\bg)$ is the constant diagram at $\lS$.
  In \autoref{thm:burnside} we showed that this diagram is a retract of $\bg(\id,a)$.
  Thus, by \autoref{thm:divbicat}, $\pi^*\lS$ is also a retract of $u_!\lS$ in $\prof(\dV)$, and therefore absolute.

  Finally, the construction of the coefficient vector as a bicategorical trace is also preserved by the equivalence of \autoref{thm:divbicat}.
  Thus, our calculation of the coefficients for $\prof(\dV(\tc))$ in \autoref{thm:burnside} yields the same result in $\prof(\dV)$.
\end{proof}

There is an easy extension to finite groupoids.
Up to equivalence, a finite groupoid $A$ is a finite disjoint union of categories $\bg$, where $G$ ranges over the automorphism groups $\nAut_A(a)$ of isomorphism classes of objects in $A$.
Thus, the conjugacy classes of $A$ can be identified with pairs $([a],C)$ where $[a]$ is an isomorphism class of objects in $A$ and $C$ is a conjugacy class in $\nAut_A(a)$.

\begin{thm}\label{thm:fingroupoiddual}
  Suppose \dV is as in \autoref{thm:derivatorburnside} and $A$ is a finite groupoid such that for each object $a\in A$, \dV is $\card{\nAut_A(a)}$-divisible.
  If $X\in\dV(A)$ is such that each $X_a$ is right dualizable, then $\colim(X)$ is also right dualizable, and for any $f\colon X\to X$ we have
  \begin{equation}\label{eq:fingroupoidtrace}
    \tr(\colim(f))
    = \sum_{[a]} \sum_{C} \frac{\card C}{\card {\nAut_A(a)}}\tr(X_C \circ f_{a}).
  \end{equation}
\end{thm}
\begin{proof}
  Let $A_0$ denote the set of objects of $A$; then the projection $p\colon A \to A_0$ is an opfibration, so for each $a\in A_0$ the pullback square
  \begin{equation}
    \vcenter{\xymatrix{
        \bB(\nAut_A(a))\ar[r]^-r\ar[d] &
        A\ar[d]\\
        \tc\ar[r]_a &
        A_0
      }}
  \end{equation}
  gives an isomorphism $ (\pi_{B(\nAut_A(a))})_!r^* \to a^* p_!$ as in \cite[Prop.~1.24]{groth:ptstab}.
  Thus, $(p_! X)_a$ is the colimit of the restriction of $X$ to $\bB\nAut_A(a)$, which is dualizable by \autoref{thm:derivatorburnside}.
  But $\colim(X) \cong \colim(p_! X)$, and the colimit over $A_0$ is a finite coproduct, hence also preserves dualizability.
  Formula~\eqref{eq:fingroupoidtrace} follows immediately from \autoref{thm:derivatorburnside} and the simple additivity formula for coproducts.
\end{proof}

\section{EI categories}
\label{sec:ei-categories}
\label{sec:eccatei}

Finally, we combine the results of \S\S\ref{sec:hofin2}--\ref{sec:deriv-burnside} to prove a linearity theorem for colimits over
\textbf{EI-categories}, i.e.\ categories in which every endomorphism is an isomorphism.
Thus, every endomorphism monoid is in fact an automorphism group.
Two extreme examples of EI-categories are groupoids (since all morphisms are isomorphisms) and posets (since all endomorphisms are identities).
Another standard example of an EI-category is the {\bf orbit category} $\sO_G$ of a finite group $G$, whose objects are the transitive $G$-sets $G/H$ and whose morphisms are $G$-maps.

If an EI-category is skeletal (as we may assume without loss of generality), then a morphism in it is invertible if and only if its domain and codomain are the same.
There is thus a partial order induced on the objects, where $a\le a'$ means there is an arrow $a\to a'$, and $a<a'$ iff there is a noninvertible such arrow.
Consider finite strings of composable noninvertible arrows $a_0 \xto{\alpha_1} a_1 \xto{\alpha_2} \cdots \xto{\alpha_n}a_n$; we include strings of length zero, which are just objects of $A$.
We say that two such strings are isomorphic if we have a commutative diagram 
\begin{equation}
  \vcenter{\xymatrix{
      a_0\ar[r]^-{\alpha_1}\ar[d]^{\cong} &
      a_1\ar[r]^-{\alpha_2}\ar[d]^{\cong} &
      \cdots \ar[r]^-{\alpha_n} &
      a_n\ar[d]^{\cong}\\
      a_0\ar[r]_-{\beta_1} &
      a_1\ar[r]_-{\beta_2} &
      \cdots\ar[r]_-{\beta_n} &
      a_n
    }}
\end{equation}
in which the vertical arrows are isomorphisms (hence automorphisms).
In particular, we can consider the set of isomorphism classes of such strings, and the automorphism group $\nAut(\vec{\alpha})$ of any such string.

Finally, any EI-category $A$ has a \textbf{core} $A_{\cong}$, which is a groupoid obtained by discarding all noninvertible arrows.
Since any cyclically composable pair in an EI-category consists of isomorphisms, we have $\Lambda A \cong \Lambda (A_{\cong})$, and thus
\[ \sh{A} \cong \bigoplus_{[a]}\, \sh{\bB\nAut_A(a)} \]
where $[a]$ ranges over isomorphism classes of objects of $A$.
In particular, we can expect to describe a coefficient vector with components indexed by pairs $([a],C)$, where $C$ is a conjugacy class in $\nAut_A(a)$.

\begin{thm}\label{thm:eicatdual}\label{thm:eccat-colim}
  Suppose that $A$ is a finite EI-category, and that \dV is a stable, closed symmetric monoidal derivator which is $\card{\nAut(\vec{\alpha})}$-divisible for each $\vec{\alpha}$.
  If $X\in\dV(A)$ is pointwise dualizable, then $\colim(X)$ is dualizable, and for any $f\colon X\to X$ we have
  \[ \tr(\colim(f)) = \sum_{[a]} \sum_{C} \tr(X_{C} \circ f_{a}) \cdot
  \sum_n (-1)^n \sum_{[\vec{\alpha}]} \sum_{\vec{C}} \frac{\card{\vec{C}}}{\card{\nAut(\vec{\alpha})}}. \]
\end{thm}

\begin{proof}
  Without loss of generality, we may assume $A$ is skeletal.
  Let $B$ be its \textbf{preorder reflection}, i.e.\ the objects of $B$ are those of $A$ and there is a unique arrow $a\to a'$ in $B$ precisely when there is some arrow $a\to a'$ in $A$.
  The assumptions that $A$ is skeletal and an EI-category imply that $B$ is also skeletal, i.e.\ a poset.
  In particular, it is homotopy finite.

  Let $D$ be the opposite of the category of nondegenerate simplices in $B$.
  The objects of $D$ are strings $a_0 < a_1 < \cdots < a_n$, where $a<a'$ means that there is an arrow $a\to a'$ in $A$, but $a\neq a'$.
  The non-identity morphisms in $D$ are given by deleting elements.
  More precisely, a morphism from $a_0 < a_1 < \cdots < a_n$ to $b_0 < b_1 < \cdots < b_m$ consists of an injection $\gamma$ from $[m] = \{0,1,\dots,m\}$ to $[n]$ such that $a_{\gamma(i)} = b_i$ for all $i$.

  Let $F\colon D\to \bCat$ be the functor sending such a string $a_0 < a_1 < \cdots < a_n$ to the category whose objects are composable strings $a_0 \xto{\alpha_1} a_1 \xto{\alpha_2} \cdots \xto{\alpha_n} a_n$ of (necessarily noninvertible) arrows in $A$, and whose morphisms are diagrams
  \begin{equation}\label{eq:autsimplex}
    \vcenter{\xymatrix{
        a_0\ar[r]^-{\alpha_1}\ar[d]^{\cong} &
        a_1\ar[r]^-{\alpha_2}\ar[d]^{\cong} &
        \cdots \ar[r]^-{\alpha_n} &
        a_n\ar[d]^{\cong}\\
        a_0\ar[r]_-{\beta_1} &
        a_1\ar[r]_-{\beta_2} &
        \cdots\ar[r]_-{\beta_n} &
        a_n
      }}
  \end{equation}
  in which the vertical arrows are automorphisms.
  The action of $F$ on face maps in $D$ is given by composing and discarding arrows.

  Let $E$ be the Grothendieck construction of $F$, with induced opfibration $p\colon E\to D$.
  Thus, the objects of $E$ are composable strings $a_0 \xto{\alpha_1} \cdots \xto{\alpha_n} a_n$ of noninvertible morphisms, and a morphism from $a_0 \xto{\alpha_1} \cdots \xto{\alpha_n} a_n$ to $b_0 \xto{\beta_1} \cdots \xto{\beta_m} b_m$ consists of an injection $\gamma:[m]\into [n]$ and isomorphisms $\delta_i : a_{\gamma(i)} \cong b_i$ such that $\beta_i = \delta_i \alpha_{\gamma(i)} \alpha_{\gamma(i)-1} \cdots \alpha_{\gamma(i-1)+1} \delta_{i-1}^{-1}$ for all $i$.
  There is a functor $q\colon E\to A$ sending a string $a_0 \xto{\alpha_1} \cdots \xto{\alpha_n} a_n$ to the object $a_0$, and a morphism $(\gamma,\delta)$ as above to the composite $\alpha_{\gamma(0)}\alpha_{\gamma(0)-1}\cdots \alpha_1$.
  We claim that analogously to the proof of \autoref{thm:hofin-constr}, this functor is homotopy final.
  
  To see this, let $a\in A$.
  The objects of $(a/q)$ are strings $a \xto{\alpha} a_0 \xto{\alpha_1} \cdots \xto{\alpha_n} a_n$ in which $\alpha$ \emph{might} be invertible but the other $\alpha_i$'s are not.
  We claim first that the subcategory of such objects for which $\alpha$ is an identity is coreflective.
  Suppose given some such string; to construct its coreflection we divide into two cases according to whether $\alpha$ is invertible or not.

  If $\alpha$ is not invertible, then the coreflection is $a \xto{\id_a} a \xto{\alpha} a_0 \xto{\alpha_1} \cdots \xto{\alpha_n} a_n$, with the counit discarding the first copy of $a$.
  A morphism $(\gamma,\delta)$ from $b_0 \xto{\beta_1} \cdots \xto{\beta_m} b_m$ to the given string factors through this uniquely by $(\gamma',\delta')$, where $\gamma'(0) = 0$ and $\gamma'(j+1) = \gamma(j)$, while $\delta'_0 = \id_a$ and $\delta'_{j+1} = \delta_j$.

  On the other hand, if $\alpha$ is invertible, so that in particular $a=a_0$, then the coreflection is $a \xto{\id_a} a \xto{\alpha_1 \alpha} a_1 \xto{\alpha_2} \cdots \xto{\alpha_n} a_n$, with counit $(\id_{[n]},\epsilon)$ where $\epsilon_0 = \alpha$ and $\epsilon_{i+1} = \id_{a_{i+1}}$.
  Given a morphism $(\gamma,\delta)$ from $b_0 \xto{\beta_1} \cdots \xto{\beta_m} b_m$ to the given string, we necessarily have $b_0 = a = a_0$ and $\delta_0 = \alpha$.
  Thus $(\gamma,\delta)$ factors uniquely through $(\id_{[n]},\epsilon)$ by $(\gamma',\delta')$, where $\gamma' = \gamma$ while $\delta'_0 = \id_a$ and $\delta'_{j+1} = \delta_{j+1}$.

  We have shown that every object of $(a/q)$ coreflects into the subcategory of strings for which $\alpha$ is an identity.
  But this subcategory has a terminal object, namely the string $a\xto{\id_a} a$.
  Thus $(a/q)$ is connected to $\tc$ by a zigzag of adjoints, so it is homotopy contractible, and hence $q$ is homotopy final.
  Therefore
  \[\colim(X) \cong \colim(q^* X) \cong \colim(p_! q^* X). \]
  However, since $p$ is an opfibration, for each object $\vec{a}=(a_0<\cdots<a_n)$ of $C$, the pullback square
  \begin{equation}
    \vcenter{\xymatrix@C=2pc{
        p^{-1} (\vec{a})\ar[r]^-i\ar[d] &
        E\ar[d]\\
        \tc \ar[r]_-{\vec{a}} &
        D
      }}
  \end{equation}
  and \cite[Prop.~1.24]{groth:ptstab} give an isomorphism $(\pi_{p^{-1}(\vec{a})})_! i^*\to (\vec{a})^*p_!$.
  Thus, $(p_! q^* X)_{\vec{a}}$ is a colimit over $p^{-1}({\vec{a}})$, which is a finite groupoid.
  Hence, by assumption and \autoref{thm:fingroupoiddual}, $(p_! q^* X)_{\vec{a}}$ is dualizable for each ${\vec{a}}$.
  But $\colim(X) \cong \colim(p_! q^* X)$, and $D$ is homotopy finite, so \autoref{thm:eccat} completes the proof of dualizability.

  We now extract the linearity formula from this construction.
  Suppose $f\colon X\to X$.
  The objects of $E$ are strings $\vec{\alpha} = (a_0 \xto{\alpha_1} \cdots \xto{\alpha_n} a_n)$ of composable noninvertible arrows in $A$, and we have $(q^*X)_{\vec \alpha} = X_{a_0}$.
  Now for any object $\vec{a}=(a_0<\cdots<a_n)$ of $D$, the isomorphism classes in $p^{-1} (\vec{a})$ are isomorphism classes $[\vec{\alpha}]$ connecting the objects $a_0,\dots,a_n$ in order.
  Thus we have
  \[(p_! q^* X)_{\vec{a}} = \bigoplus_{[\vec{\alpha}]} (X_{a_0} / \nAut(\vec{\alpha})) \]
  where $\nAut(\vec{\alpha})$ denotes the group of automorphisms of $\vec{\alpha}$ as in~\eqref{eq:autsimplex}, which act on $X_{a_0}$ via their first components $a_0 \toiso a_0$.
  Thus, by \autoref{thm:fingroupoiddual} and the simple additivity formula for coproducts, the trace of the induced endomorphism of $(p_! q^* X)_{\vec{a}}$ is
  \[ \sum_{[\vec{\alpha}]} \sum_{\vec{C}} \frac{\card {\vec{C}}}{\card{\nAut(\vec{\alpha})}} \tr(X_{C_0} \circ f_{a_0}) \]
  where $\vec{C}$ ranges over conjugacy classes in $\nAut(\vec{\alpha})$, and $X_{C_0}$ denotes the action of such a $\vec{C}$ on $X_{a_0}$ via its first component.
  Now \autoref{thm:eccat} yields the following formula for the trace of the induced endomorphism of $\colim(X) = \colim(p_! q^* X)$:
  \begin{multline}
    \tr(\colim(f)) = \sum_{\vec{a}}
    \left(\sum_{[\vec{\alpha}]} \sum_{\vec{C}} \frac{\card{\vec{C}}}{\card{\nAut(\vec{\alpha})}}\tr(X_{C_0} \circ f_{a_0})\right)\cdot \\
    \left(\sum_{k\ge 0} (-1)^k \cdot \card {\left\{ \parbox{4cm}{\centering composable strings of nonidentity face maps of length $k$ starting at $\vec{a} \in D$}\right\}}\right)
  \end{multline}
  However, by \autoref{thm:realiz-coeff} which we prove below, the last factor is simply equal to $(-1)^n$, where $n$ is the length of $\vec{a}$.
  Therefore, rearranging the summations, we obtain
  \[ \tr(\colim(f)) = \sum_{a} \sum_{C} \tr(X_{C} \circ f_{a}) \cdot
  \sum_n (-1)^n \sum_{[\vec{\alpha}]} \sum_{\vec{C}} \frac{\card{\vec{C}}}{\card{\nAut(\vec{\alpha})}} \]
  where $a$ ranges over objects of $A$, $C$ ranges over conjugacy classes of automorphisms of $a$, $[\vec{\alpha}]$ ranges over isomorphism classes of strings of $n$ composable noninvertible arrows starting at $a$, and $\vec{C}$ ranges over conjugacy classes of automorphisms of $\vec{\alpha}$ which restrict to $C$ on $a$.
\end{proof}

\begin{lem}\label{thm:realiz-coeff}
  For any $n$, we have
  \begin{equation}
    \sum_{k\ge 0} (-1)^k \cdot \card {\left\{
      \parbox{5.4cm}{\centering composable strings of $k$ nonidentity\\ face maps starting at $[n] \in \bbDelta\op$}
    \right\}}
    = (-1)^n.
  \end{equation}
\end{lem}

\begin{proof}
  This lemma has many proofs~\cite{mo:chains}.
  Here we sketch a topological one; in \S\ref{sec:uniqueness} on page~\pageref{proof:realiz-coeff} we will give another.
  Let $g(n,k)$ be the number of composable strings of $k$ nonidentity face maps starting at $[n]$; thus the claim is that
  \[ \sum_{k\ge 0} (-1)^k \, g(n,k) = (-1)^n. \]

  Let $\Delta$ be the standard $n$-simplex and $\Delta'$ its barycentric subdivision.
  Then the $k$-simplices of $\Delta'$ are composable strings of $k+1$ face maps starting at $[n] \in \bbDelta\op$ of which the first one might be an identity.
  Those for which the first map is not an identity are precisely those that do not contain the barycenter, i.e.\ that lie in the boundary $\partial\Delta'$.
  Therefore, for $k>0$, $g(n,k)$ is the number of $(k-1)$-simplices in $\partial\Delta'$.
  Since the Euler characteristic of a simplicial complex is the alternating sum of its simplices by dimension, we have
  \[\chi(\partial \Delta')=\sum_{k\geq 1}(-1)^{k-1}g(n,k)=-\left(\sum_{k\geq 1}(-1)^{k}g(n,k)\right)\]
  However, $\partial \Delta'$ is topologically an $(n-1)$-sphere, so this is also equal to $\chi(S^{n-1}) = 1+(-1)^{n-1}$.
  After adding in $g(n,0)=1$, we have the desired statement.
\end{proof}

  \autoref{thm:eicatdual} implies that if \dV is stable and \emph{rational} (i.e. $n$-divisible for all nonzero integers $n$), then colimits of all finite EI-categories are absolute.\footnote{To be precise, it only says that such colimits preserve dualizability, which is a \emph{consequence} of absoluteness.
    However, when we generalize to bicategories, preserving dualizability becomes equivalent to absoluteness; see \S\ref{sec:uniqueness}.}
  However, these are not the only absolute colimits in such a \dV.
  For instance, if $p\colon B\to A$ is a homotopy final functor, then the colimit of $X\in\dV(A)$ can be computed as the colimit of $p^*(X) \in \dV(B)$; thus if $B$-colimits are absolute in \dV, so are $A$-colimits.

  As a concrete example, if $A = \mathbf{B}F_n$ is the one-object groupoid corresponding to the free group on $n$ generators, then there is a homotopy final functor $p\colon B\to A$ where $B$ has two objects $b_1$ and $b_2$ with $n+1$ nonidentity arrows from $b_1$ to $b_2$.
  The functor $p$ sends one of these arrows to the identity and the others to the free generators.
  Since $B$ is finite, $B$-colimits are absolute in any stable \dV, and hence so are $A$-colimits.
  The formula~\eqref{eq:eccattr} gives $\tr(\colim(g)) = \tr(g_{b_2}) - n \tr(g_{b_1}) $ for any endomorphism $g$ of a pointwise dualizable $Y\in \dV(B)$, whence $\tr(\colim(f)) = \tr(\colim(p^*(f))) = (1-n)\tr(f_{\pt})$ for any endomorphism $f$ of a pointwise dualizable $X\in\dV(A)$.

  Thus, finiteness of the group $G$ is not necessary for $\bg$-colimits to be absolute or to have a linearity formula.
  In fact, when \dV is the homotopy category of chain complexes, $\bg$-colimits are absolute whenever $G$ is of \emph{type FP} as defined in~\cite[VIII.6]{brown:cog}.
  See also~\cite[Example 8.7]{PS3}.
  
  We might call a category \textbf{finally homotopy finite} if it admits a homotopy final functor from a homotopy finite category; these are a homotopical version of the duals of the \emph{L-finite categories} of~\cite{pare:sclim}.
  Thus, colimits over all such categories are absolute in stable derivators.
  More generally, we can consider categories admitting a homotopy final functor from a finite EI-category, whose colimits will be absolute in any rational stable derivator.

  However, even this does not exhaust the absolute colimits in such derivators.
  We have seen that splitting of idempotents is absolute in \emph{any} \dV, and the free-living idempotent is not finally homotopy finite (see~\cite[Example 4.4.5.1]{lurie:higher-topoi}).
  We do not know a characterization of all colimits that are absolute in any stable derivator or in any rational stable derivator.

\begin{rmk}
  A related question is whether there is any \dV in which \emph{all} colimits are absolute.
  We have seen in \autoref{thm:eccat} that finite colimits are absolute in any stable \dV, and in \autoref{eg:sup} that infinite coproducts are absolute in suplattices.
  However, these two properties are impossible to combine nontrivially, because of the ``Eilenberg swindle''.
  Specifically, if countably infinite coproducts are absolute in \dV, then we can ``add up countably many parallel morphisms'' $\{f_i \colon  X\to Y\}_{i\in \mathbb{N}}$ to get a single morphism $\sum_{i\in \mathbb{N}} f_i \colon  X\to Y$, just as we can add up finitely many parallel morphisms in a semi-additive category.
  But if \dV is also additive, so that we can subtract morphisms, then for any $f\colon X\to Y$ we have
  \[f = \left(f + \sum_{i\in \mathbb{N}} f\right) - \sum_{i\in \mathbb{N}} f = \sum_{i\in \mathbb{N}} f - \sum_{i\in \mathbb{N}} f = 0.\]
  Since stable derivators are additive, countable coproducts cannot be absolute in any nontrivial stable monoidal derivator.
\end{rmk}

There are other ways to organize the many terms  in  \autoref{thm:eccat-colim}, one of which is the formula of~\cite{souza:traces}.
We end this section by explaining how his formula is equivalent to ours.

Let $A$ be a finite EI-category.
Recall that according to \autoref{thm:eccat-colim}, the coefficient associated to a conjugacy class $C$ in $\nAut_A(a)$ is
\[ \sum_n (-1)^n \sum_{[\vec{\alpha}]} \sum_{\vec{C}} \frac{\card{\vec{C}}}{\card{\nAut(\vec{\alpha})}}. \]
Undoing the rearrangement at the end of the proof, we can rewrite this as a sum over finite sequences $\vec{a} = (a_0 < a_1 < \dots < a_n)$ of objects of $A$:
\[ \sum_{\vec{a}} (-1)^{|\vec{a}|} \sum_{[\vec{\alpha}]} \sum_{\vec{C}} \frac{\card{\vec{C}}}{\card{\nAut(\vec{\alpha})}}. \]
where $|\vec{a}|$ denotes the length of $\vec{a}$, and $\vec{\alpha}$ runs only over sequences of arrows whose underlying sequence of objects is $\vec{a}$.

The formula of~\cite{souza:traces} also begins with $\sum_{\vec{a}} (-1)^{|\vec{a}|}$, so from now on we fix a particular $\vec{a}$ of length $n$ and consider only the rest of the sum:
\begin{equation}
  \sum_{[\vec{\alpha}]} \sum_{\vec{C}} \frac{\card{\vec{C}}}{\card{\nAut(\vec{\alpha})}}\label{eq:ds1}
\end{equation}
Let $G_i = \nAut_A(a_i)$, let $G = \prod_{0\le i\le n} G_i$, and let $Z$ be the set of sequences of arrows $\vec{\alpha}$ with underlying sequence of objects $\vec{a}$.
Then the group $G$ acts on the set $Z$ as follows: the group element $(g_0,g_1,\dots,g_n)$ sends the sequence
\[a_0 \xto{\alpha_1} a_1 \xto{\alpha_2} \cdots \xto{\alpha_n}a_n\]
to the sequence
\[a_0 \xto{g_1\alpha_1 g_0^{-1}} a_1 \xto{g_2 \alpha_2 g_1^{-1}} \cdots \xto{g_n \alpha_ng_{n-1}^{-1}}a_n.\]
The isomorphism classes of $\vec{\alpha}$ are precisely the \emph{orbits} of this group action, and the automorphism group of an $\vec{\alpha}$ class is precisely its \emph{stabilizer group} $G_{\vec{\alpha}}$.
Thus, we can rewrite~\eqref{eq:ds1} as
\begin{equation}\label{eq:ds2}
  \sum_{[\vec{\alpha}] \in Z/G} \sum_{\vec{C}} \frac{\card{\vec{C}}}{\card{G_{\vec{\alpha}}}}
\end{equation}
Recall that $\vec{C}$ runs over conjugacy classes in $G_{\vec{\alpha}}$ which restrict to $C$ on $a_0$.
Thus, if we let $S = C \times \prod_{1\le i\le n} G_i \subseteq G$, then~\eqref{eq:ds2} is equal to
\begin{equation}\label{eq:ds3}
  \sum_{[\vec{\alpha}] \in Z/G} \frac{\card{(G_{\vec{\alpha}}\cap S)}}{\card{G_{\vec{\alpha}}}}.
\end{equation}
Now we can invoke the following lemma.

\begin{lem}
  Let a finite group $G$ act on a finite set $Z$, and let $S\subseteq G$ be a union of conjugacy classes.
  Then
  \[ \sum_{[z] \in Z/G} \frac{\card{(G_z \cap S)}}{\card{(G_z)}} = \sum_{g\in S} \frac{\card{(Z^g)}}{\card{G}}. \]
\end{lem}
The assumption on $S$ ensures that the left-hand summand depends only on $[z]\in Z/G$.
Note that if $S=G$, this reduces to the orbit-counting theorem.
\begin{proof}
  Both sides are additive under disjoint union of $G$-sets, so it suffices to consider the case of an orbit $Z=G/H$ for some subgroup $H\le G$.
  In this case the left-hand side becomes $\frac{\card{(H\cap S)}}{\card H}$.
  Now if $g\in H$ then $\card{((G/H)^g)} = \card{(G/H)}$, while otherwise $\card{((G/H)^g)} = 0$; thus the right-hand side also becomes
  \[ \sum_{g\in H\cap S} \frac{\card{(G/H)}}{\card G} = \card{(H\cap S)} \frac{\;\frac{\card{G}}{\card{H}}\;}{\card{G}} = \frac{\card{(H\cap S)}}{\card H}. \qedhere \]
\end{proof}

Returning to the situation at hand, the lemma tells us that~\eqref{eq:ds3} is equal to
\begin{align}
  \sum_{\vec{g}\in S} \frac{\card{(Z^{\vec{g}})}}{\card{G}}
  &= \sum_{g_0\in C} \sum_{g_1\in G_1} \dots \sum_{g_n \in G_n} \frac{\card{(Z^{\vec{g}})}}{(\card{G_0})(\card{G_1}) \cdots (\card{G_n})}\\
  &= \sum_{C_1} \dots \sum_{C_n} \left(\card{(Z^{\vec{g}})}\right) \frac{\card{C}}{\card{G_0}} \frac{\card{C_1}}{\card{G_1}} \cdots \frac{\card{C_n}}{\card{G_n}}\label{eq:ds4}
\end{align}
where $C_i$ ranges over conjugacy classes in $G_i$, and $\vec{g}=(g_0,g_1,\dots,g_n)$ for some $g_i\in C_i$ (it doesn't matter which for $\card{(Z^{\vec{g}})}$).
If we choose a subset $T_i \subseteq G_i$ for each $i$ containing exactly one element from each conjugacy class, with $h_0$ the element of $T_0$ in our chosen conjugacy class $C$, then~\eqref{eq:ds4} is equal to
\begin{align}
  \sum_{h_1 \in T_1} \dots \sum_{h_n \in T_n}
  \sum_{\alpha_1}
  \dots
  \sum_{\alpha_n}
  \frac{\card{[h_0]}}{\card{G_0}} \frac{\card{[h_1]}}{\card{G_1}} \cdots \frac{\card{[h_n]}}{\card{G_n}}\label{eq:ds4}
\end{align}
in which $\alpha_i$ ranges over those morphisms $a_{i-1} \to a_i$ such that $h_i \alpha_i (h_{i-1})^{-1} = \alpha_i$, or equivalently $h_i \alpha_i = \alpha_i h_{i-1}$.

This is essentially the formula of~\cite[Lemma 23 and Theorem 26]{souza:traces}.
There it is phrased as follows.
Let $\End(A)$ be the category whose objects are endomorphisms $h\colon a\to a$ in $A$, and whose morphisms from $h\colon a\to a$ to $k\colon b\to b$ are morphisms $\alpha\colon a\to b$ such that $\alpha h = k\alpha$.
Note that the isomorphism classes of objects in $\End(A)$ are precisely the conjugacy classes in automorphism groups of objects in $A$.
Let $J$ be a skeleton of $\End(A)$; thus it amounts to choosing a representative $h$ of each conjugacy class in each $\nAut_A(a)$.
Therefore,~\eqref{eq:ds4} can be written as a sum over composable strings of arrows in $J$
\begin{equation}\label{eq:ds5}
  \sum_{h_0 \xto{\alpha_1} h_1 \xto{\alpha_2} \cdots \xto{\alpha_n} h_n} \frac{\card{[h_0]}}{\card{G_0}} \frac{\card{[h_1]}}{\card{G_1}} \cdots \frac{\card{[h_n]}}{\card{G_n}}
\end{equation}
where $h_i$ is required to be an endomorphism of $a_i$.
However, since we are also adding these up over all $\vec{a}$, the end result can be interpreted as a sum over \emph{all} composable sequences of noninvertible arrows in $J$.
This is how it is written in~\cite{souza:traces}.

\part{Linearity in bicategories}
\label{part:bicat}

In \S\ref{sec:moncat}, we used bicategorical traces in the bicategory $\prof(\V)$ to prove linearity formulas for symmetric monoidal traces in a symmetric monoidal category \V.
We generalized to derivators in \S\ref{sec:der}.
In this part, we generalize further to the case when we \emph{start} with a bicategory (or a \emph{derivator bicategory}) instead of a symmetric monoidal category, thereby obtaining linearity formulas for \emph{bicategorical} traces.  

The ease of this generalization is one of the primary motivations for our general approach to linearity.
In particular, in \S\ref{sec:derbicat} we will use it to generalize the additivity formula of~\cite{add} to bicategorical traces, without having to generalize the complicated axioms of~\cite{add} to the case of bicategorical traces.
(This was done for the first four axioms (TC1)--(TC4) in~\cite[\S16.7]{maysig:pht}, but it is much more difficult to generalize the final axiom (TC5) to bicategories.
It should be possible to generalize the version of May's proof for monoidal derivators presented in~\cite{gps:additivity} to derivator bicategories, but the present approach avoids this question entirely.)

Just as the additivity formula of~\cite{add} applies to the Lefschetz number, which is a trace in the symmetric monoidal derivator of spectra, the bicategorical version in \S\ref{sec:derbicat} applies to an analogous invariant arising from the derivator bicategory of parametrized spectra called the \emph{Reidemeister trace}.
As remarked in \S\ref{sec:introduction}, we postpone the details of the application to Reidemeister trace to the companion paper~\cite{PS6}; in \S\ref{sec:derbicat} we will only sketch the argument.

As in \autoref{part:more}, we postpone some details and proofs until \autoref{part:formal}.

\section{Linearity in ordinary bicategories}
\label{sec:bicat}

Let \W be a closed bicategory equipped with a shadow valued in a category \T.
We assume that \W is \emph{locally complete and cocomplete}, i.e.\ its hom-categories $\W(R,S)$ are complete and cocomplete --- since the composition functor $\odot$ has both adjoints, it automatically preserves colimits in each variable.
We assume furthermore that \T is cocomplete and the shadow functors $\sh{-}\colon \W(R,R) \to \T$ are cocontinuous.

We now define a new closed bicategory $\prof(\W)$, with shadow also valued in $\T$, as follows.
\begin{itemize} 
\item An object is a pair $(A,R)$ where $A\in\cCat$ and $R$ is an object of $\W$.
\item The hom category from $(A,R)$ to $(B,S)$ is the functor category $\W({R,S})^{A\times B\op}$.
\item The composite of $H\colon (A,R)\hto(B,S)$ and $K\colon (B,S)\hto(C,T)$ is defined by
  \[ (H\odot K)(a,c) = \int^{b\in B} H(a,b) \odot K(b,c). \]
  Note that this is a coend in the cocomplete hom-category $\W(R,T)$.
\item The unit 1-morphism $\lI_{(A,R)} \colon  (A,R)\hto (A,R)$ defined by the copower
  \[ \lI_{(A,R)}(a,a') = A(a,a') \cdot \lI_R.\]
\item The shadow of $H\colon (A,R) \hto (A,R)$ is
  \[ \sh{H} = \int^{a\in A} \sh{H(a,a)}. \]
  Note that this is a coend in $\T$.
\item The internal hom $\rhd$ is defined for $H\colon (A,R)\hto (B,S)$ and $K\colon (C,T) \hto (B,S)$ by
  \[ (H\rhd K)(c,a) = \int_{b\in B} H(a,b) \rhd K(c,b), \]
  and similarly for $\lhd$.
\end{itemize}
The intent is to generalize as closely as possible the construction of $\prof(\V)$ from \S\ref{sec:moncat}.
Indeed, if we regard a monoidal category $\V$ as a bicategory with one object, the two constructions agree.
A version of \autoref{rmk:enriched-cats} also applies here. 

Now suppose given $X\colon A\to \W(R,S)$ and $\Phi\colon A\op \to \W(R,R)$, which we can regard respectively as morphisms $(A,R) \hto (\tc,S)$ and $(\tc,R) \hto (A,R)$ in $\prof(\W)$.
Thus, we have a composite $\Phi \odot X \colon  (\tc,R) \hto (\tc,S)$, which is essentially just a morphism $R\hto S$ in $\W$; we call this the \textbf{$\Phi$-weighted colimit of $X$} and denote it $\colim^\Phi(X)$.
If $\Phi$ is furthermore constant at the unit object $\lI_R$, then the $\Phi$-weighted colimit of $X$ is simply the ordinary colimit of $X\colon A \to \W(R,S)$ in the category $\W(R,S)$.
Now we can apply \autoref{thm:compose-duals} and \autoref{thm:compose-traces}.

\begin{defn}\
  \begin{itemize}
  \item A functor $X\colon A\to \W(R,S)$ is \textbf{pointwise right dualizable} if it is right dualizable as a morphism $(A,R) \hto (\tc,S)$ in $\prof(\W)$.
  \item A weight $\Phi\colon A\op \to \W(R,R)$ is \textbf{absolute} if it is right dualizable as a morphism $(\tc,R) \hto (A,R)$ in $\prof(\W)$.
  \end{itemize}
\end{defn}

\begin{lem}\label{thm:bicatpwdual}
  A functor $X\colon A\to \W(R,S)$ is pointwise right dualizable if and only if each 1-cell $X(a) \in\W(R,S)$ is right dualizable in $\W$.
\end{lem}
\begin{proof}
  Just like \autoref{thm:smcpwdual}.
\end{proof}

\begin{thm}
  If $X\colon A\to \W(R,S)$ is pointwise right dualizable and $\Phi\colon A\op\to \W(R,R)$ is absolute, then $\colim^\Phi(X)$ is right dualizable.
  In this case, for any $P\colon \tc\to \W(S,S)$ and endomorphism $f\colon X\to X\odot P$, we have
  \begin{equation}\label{eq:bicatlin}
    \tr(\colim^\Phi(f)) = \tr(f) \circ \tr(\id_\Phi).
  \end{equation}
\end{thm}
\begin{proof}
  By \autoref{thm:compose-duals} and \autoref{thm:compose-traces}.
\end{proof}

Note that we are now including a twisting in the target of $f$; this is because the application to Reidemeister trace requires it.
In this section, however, $P$ will always be the unit $\lI_S$.
As observed in \autoref{rmk:twisting}, we could also include a source twisting if desired.

As before, in order to make use of this, we analyze the two factors further.
The shadow of $(A,R)$ in $\prof(\W)$ is 
\begin{equation}
  \sh{(A,R)} = \left(\int^{a\in A} A(a,a) \right) \cdot \sh{R},
\end{equation}
i.e.\ the copower of the shadow of (the identity 1-cell of) $R$ by the set of conjugacy classes in the ordinary category $A$.
Thus, for $X\colon A \to \W(R,S)$, $P\colon \tc\to \W(S,S)$ and $f\colon X\to X\odot P$, the trace $\tr(f) \colon  \sh{(A,R)}\to \sh{P}$ is determined by one morphism $\sh{R} \to \sh{P}$ for each conjugacy class of $A$, which we can identify using a component lemma. 

\begin{lem}[The  component lemma for bicategories]\label{thm:bicatomega}
  For any right dualizable $X\colon (A,R) \hto (\tc,S)$ with $f\colon X\to X\odot P$, and any morphism $\alpha\in A(a,a)$, the component $\tr(f)_{[\alpha]}$ is the trace in \W of the composite
  \begin{equation}
    \xymatrix{X_a \ar[r]^-{X_{\alpha}} & X_a \ar[r]^-{f_a} & X_a\odot P}
  \end{equation}
\end{lem}

\begin{proof}
This lemma can be proven explicitly like \autoref{thm:smcomega2a}, but we will deduce it from a more abstract result in \S\ref{sec:basechangebicat} on page \pageref{thm:bicatomegap}.
\end{proof}

Continuing with the analogy, for any absolute $\Phi\colon A\op\to \W(R,R)$ we refer to $\tr(\id_\Phi)\colon  \sh{R} \to \sh{(A,R)}$ as its \textbf{coefficient vector}.
When the target category \T of the shadow is semi-additive and $A$ is finite,  $\sh{(A,R)}$ is a direct sum of copies of $\sh{R}$ indexed by the conjugacy classes of $A$.
Thus, $\tr(\id_\Phi)$ is a column vector with entries in the semiring $\T\big(\sh{R},\sh{R}\big)$; as before, we denote these entries by $\phi_{[\alpha]}$ and call them the \textbf{coefficients} of $\Phi$.
Thus, we have a linearity formula for bicategorical traces:

\begin{cor}\label{thm:bicatlin2}
  If the target category $\T$ of the shadow is semi-additive, $A$ is finite, and $\Phi\colon A\op\to \W(R,R)$ is absolute, then we have
  \begin{equation}\label{eq:bicatlin2}
    \tr(\colim^\Phi f) = \sum_{[\alpha]} \phi_{[\alpha]} \cdot \tr(f_a \circ X_\alpha).
  \end{equation}
  for any pointwise dualizable $X\colon A\to \W(R,S)$, any 1-cell $P\in \W(S,S)$, and any 2-cell $f\colon X\to X\odot P$.
\end{cor}

This formula is syntactically identical to~\eqref{eq:smclin2}; the only difference is that now $X$ is a functor $A\to \W(R,S)$, $\Phi$ is a functor $A\op \to \W(R,R)$, $f$ is an endomorphism of $X$, and the equation is between morphisms $\sh{R}\to\sh{P}$ in \T.

We now have essentially all the same examples of absolute weights and coefficient vectors that we had in \S\ref{sec:moncat}.

\begin{eg}
  Let $A$ be the empty category, and $\Phi\colon A\to \W(R,R)$ the unique \W-profunctor for some object $R$ of \W.
  If $U\colon A\to \W(T,R)$ is the unique \W-profunctor for some other object $T$, then $\mu_{\Phi,U}$ is the unique map from an initial to a terminal object of $\W(T,R)$, 
  which is an isomorphism for all of these categories  just when \W is {\bf locally pointed}, i.e.\ each category $\W(R,S)$ is pointed.
  In this case, the $\Phi$-weighted colimit of the unique $X\colon A\hto \W(R,S)$ is the zero object of $\W(R,S)$, which is therefore pointwise dualizable.
  The shadow of $A$ is the initial object of \T, so when \T is pointed as well (as is usually the case), the trace of the identity morphism of $0\in\W(R,S)$ is the zero morphism $\sh{R} \to \sh{S}$ in \T.
\end{eg}

\begin{eg}\label{eg:bicatadditive}
  Let $R$ be an object of \W, and let $A$ be the discrete category with two objects.
  Then $X\colon A \to \W(R,S)$ consists of a pair of 1-cells $X_a, X_b \in \W(R,S)$, and is pointwise right dualizable just when $X_a$ and $X_b$ are right dualizable in \W.
  If $\Phi\colon A\op \hto \W(R,R)$ is constant at $\lI_R$, then the weighted colimit $\colim^\Phi(X) = \Phi\odot X$ is the local coproduct $X_a \sqcup X_b$.
  Just as in \autoref{eg:dirsum}, we conclude that if \W is locally pointed, then $\mu_{\Phi,U}$ is the canonical map from a binary coproduct to a binary product, which is an isomorphism if \W is \emph{locally semi-additive}.
  In this case, our $\Phi$ is absolute, so that local binary coproducts preserve right dualizability.

  Now the shadow of $(A,R)$ is the coproduct $\sh{R} \oplus \sh{R}$ in \T, and for a pointwise right dualizable $X\colon A \to \W(R,S)$ and an endomorphism $f\colon X\to X$, the trace $\tr(f)\colon \sh{R} \oplus \sh{R} \to \sh{S}$ has components $\tr(f_a)$ and $\tr(f_b)$.
  If \T is also semi-additive, then the coefficient vector of $\Phi$ is determined by two components $\phi_a,\phi_b\colon \sh{R} \to\sh{R}\oplus \sh{R}$, which as before we can determine to both be $1$ by a judicious choice of $X$.
  Thus, we have the linearity formula
  \[ \tr (f_a \oplus f_b) = \tr(f_a) + \tr(f_b) \]
  exactly as in \autoref{eg:dirsum}, but now for bicategorical traces.
  For instance, when \W is the bicategory of rings and bimodules, this yields the additivity of the Hattori-Stallings trace under direct sums.
\end{eg}

All the other examples from \S\ref{sec:moncat-examples} generalize to bicategories in an entirely analogous way.
We leave the details to the reader.

\section{Linearity in derivator bicategories}
\label{sec:derbicat}

We now combine the theory of linearity for monoidal derivators (\S\ref{sec:der}) with that for ordinary bicategories (\S\ref{sec:bicat}) to obtain a theory of linearity for \emph{derivator bicategories}.
This is necessary for the application to Reidemeister trace~\cite{PS6}, which is a bicategorical trace but is linear in the \emph{stable} sense of \S\ref{sec:der}.
It is also necessary for the uniqueness theorem in \S\ref{sec:uniqueness}.

The definition of derivator bicategory is obtained from the definition of a bicategory by simply replacing all hom-categories with derivators.

\begin{defn}\label{def:derivbicat}
  A \textbf{derivator bicategory} $\dW$ consists of the following data.
  \begin{itemize}
  \item A collection of objects $R$, $S$, $T$, $\ldots$.
  \item For each pair of objects $R$ and $S$ a derivator $\dW({R,S})$.
    We think of the category $\dW(R,S)(A)$ as the homotopy category of $A$-shaped diagrams in $\dW(R,S)$.
  \item For each triple of objects $R$, $S$, and $T$, a morphism of derivators
    \[\odot \colon \dW({R,S})\times \dW({S,T})\to \dW({R,T}).\]
    That is, we have a pseudonatural transformation between 2-functors $\cCat\op\to\cCAT$, which has components
    \[ \dW({R,S})(A)\times \dW({S,T})(A)\to \dW({R,T})(A).\]
  \item We require these morphisms $\odot$ to be cocontinuous in each variable separately \cite[Definition~3.19]{gps:additivity}.
  \item For each object $R$, a morphism of derivators $\lI\colon y(\tc)\to \dW(R,R)$ (hence an object $\lI_{R,A}\in \dW(R,R)(A)$, varying pseudonaturally in $A\in\cCat$).
  \item Natural unit and associativity isomorphisms, i.e.\ invertible modifications
    \begin{equation}
      \vcenter{\xymatrix{
          \dW(R,S) \times \dW(S,T) \times \dW(T,U)\ar[r]^-{\id\times\odot}\ar[d]_{\odot\times\id} \drtwocell\omit{\cong} &
          \dW(R,S) \times \dW(S,U)\ar[d]^\odot\\
          \dW(R,T) \times \dW(T,U)\ar[r]_-{\odot} &
          \dW(R,U)
        }}
    \end{equation}
    \begin{equation}
      \xymatrix{\dW(R,S) \ar[r]^-{(\id,\lI)} \drlowertwocell{\cong} &
        \dW(R,S) \times \dW(S,S)\ar[d]^\odot \\
        & \dW(R,S)
      }
      \qquad
      \xymatrix{\dW(R,S) \ar[r]^-{(\lI,\id)} \drlowertwocell{\cong} &
        \dW(R,R) \times \dW(R,S)\ar[d]^\odot \\
        & \dW(R,S).
      }
    \end{equation}
  \item The usual pentagon and unit axioms for a bicategory hold.
  \end{itemize}
  A derivator bicategory is \textbf{closed} if the morphisms $\odot$ participate in a two-variable adjunction of derivators.
\end{defn}

As in \cite{gps:additivity}, we define the {\bf external composition} to be the composite
\[\dW({R,S})(A)\times \dW({S,T})(B)\to \dW({R,S})(A\times B)\times \dW({S,T})(A \times B) \to \dW({R,T})(A\times B)\]
where the first maps are restrictions induced by the projections.
Joint cocontinuity is defined in terms of this composition rather than the original \emph{internal} composition.

Unsurprisingly, we also need to extend the notion of shadow to the derivator case.

\begin{defn}
  A \textbf{shadow} on a derivator bicategory \dW consists of a derivator \dT and cocontinuous morphisms of derivators
  \[\sh{-} \colon \dW(R,R) \xto{} \dT\]
  for each object $R$, together with invertible modifications
  \begin{equation}
    \vcenter{\xymatrix{
        \dW(R,S)\times \dW(S,R)\ar[rr]^{\cong} \ar[d]_{\odot} \drrtwocell\omit{\cong} &&
        \dW(S,R)\times \dW(R,S)\ar[d]^{\odot}\\
        \dW(R,R)\ar[r]_-{\sh{-}} &
        \dT\ar@{<-}[r]_-{\sh{-}} &
        \dW(S,S)
      }}
  \end{equation}
  satisfying the usual compatibility axioms for a shadow (\cite[Defn.~4.1]{PS2}).
\end{defn}

Note that just as we did in \S\ref{sec:bicat}, we require the shadow functors to be cocontinuous (in the appropriate sense).

A derivator bicategory \dW has an {\bf underlying ordinary bicategory} with the same objects, and whose hom-category from $R$ to $S$ is $\dW(R,S)(\tc)$.
If \dW has a shadow, then so does its underlying ordinary bicategory.

One obvious way to construct derivator bicategories is by taking the homotopy bicategory of a \emph{model bicategory}.
Recall that from a model category $C$ with weak equivalences $W$,
we define a derivator $\ho(C)$ by \[\ho(C)(A)\coloneqq (C^A)[(W^A)^{-1}].\]
Motivated by this, we say a  \textbf{model bicategory} 
is a closed bicategory $\sB$ with model structures on each of the
hom-categories $\sB(R,S)$ that satisfy the pushout-product and unit axioms.

\begin{itemize} 
\item (Pushout-Product Axiom)
If $X\rightarrow Y$ and $K\rightarrow L$ are cofibrations,
then the map \[(X\odot L) +_{(X\odot K)}(Y\odot K)\rightarrow (Y\odot L)\] is a
cofibration, which is a weak equivalence if either of the maps  $X\rightarrow Y$
or $K\rightarrow L$ are.

\item (Unit Axiom)
If $Q\lI_R\rightarrow \lI_R$ is a cofibrant
replacement for a bicategorical unit $\lI_R$, then
\begin{align*}
  Q\lI_R\odot Y \rightarrow \lI_R\odot Y\cong Y \qquad\text{and}\qquad
  X\odot Q\lI_R \rightarrow X\odot \lI_R\cong X
\end{align*}
are weak equivalences for any cofibrant $X$ and $Y$.
\end{itemize}
If $\sB$ additionally has a shadow, we call it a \textbf{Quillen shadow} if it takes values in a model category $\sT$ and each functor $\sh{-}\colon \sB(R,R) \to \sT$ is left Quillen.

\begin{eg}
  Any monoidal model category can be regarded as a model bicategory with one object.
  If it is symmetric, then its identity functor is a Quillen shadow.
\end{eg}

\begin{eg}
  There is a model bicategory whose objects are noncommutative rings, and where $\sB(R,S)$ is the category of unbounded chain complexes of $R$-$S$-bimodules with a projective model structure.
  The pushout product and unit axioms can be proven by adapting the arguments of~\cite{hovey:modelcats} from the monoidal case.
  It also has a Quillen shadow with values in chain complexes of abelian groups, which coequalizes the left and right actions.
\end{eg}

Now we can state the following theorem.

\begin{thm}\label{thm:modelbicat}
  For any model bicategory $\sB$, there is a closed derivator bicategory $\ho(\sB)$ whose objects are
  the objects of $\sB$, and whose hom-derivator $\ho(\sB)(R,S)$ is the derivator determined by the model category $\sB(R,S)$.
  Moreover, if $\sB$ has a Quillen shadow, then $\ho(\sB)$ has a shadow.
\end{thm}

One can prove this by extending results of \cite{gps:additivity} from the monoidal case. 
We omit the details.

\begin{rmk}
\autoref{thm:modelbicat} does have its limitations, however.
In particular, we cannot use it to construct the derivator bicategory \Ex of parametrized spectra, which is the example relevant for~\cite{PS6}.
In that paper we instead enhance the construction of~\cite{shulman:frbi,PS3} to construct a derivator bicategory from an \emph{indexed monoidal derivator}, the latter of which can be obtained from an \emph{indexed monoidal model category}.
\end{rmk}

We now proceed with a straightforward generalization of \S\ref{sec:der}.
In fact, now that we have the notion of derivator bicategory, we can prove a stronger result: the bicategory $\prof(\dW)$ in fact underlies a new derivator bicategory $\dprof(\dW)$.

\begin{thm}\label{thm:derivbicat-prof}
  Given a derivator bicategory \dW, we can construct a derivator bicategory $\dprof(\dW)$, with underlying ordinary bicategory denoted $\prof(\dW)$.  The latter is described as follows:
  \begin{itemize} 
  \item An object is a pair $(A,R)$ where $A\in\cCat$ and $R$ is an object of $\dW$.
  \item The hom category from $(A,R)$ to $(B,S)$ is $\dW({R,S})(A\times B\op)$.
  \item The composition functors are 
    \begin{align*}
      \dW({R,S})(A\times B\op)\times \dW({S,T})(B\times C\op)&
      \xto{\odot} \dW({R,T})(A\times B\op\times B\times C\op)\\
      &\xto{\int^B} \dW({R,T})(A\times C\op).
    \end{align*}
  \item The unit object of $(A,R)$ is $\lI_{(A,R)} = (t,s)_! \lI_{R,\tw(A)} \in \dW(R,R)(A\times A\op)$.
  \end{itemize}
  For the derivator bicategory $\dprof(\dW)$, the hom-derivators are defined by
  \[\dprof(\dW)((A,R),(B,S))(C) = \dW(R,S)(A\times B\op\times C),\]
  i.e.\ $\dprof(\dW)((A,R),(B,S))$ is the shifted derivator $\shift{\dW(R,S)}{A\times B\op}$.
  The composition and units are defined analogously.
  If \dW is closed, then so is $\dprof(\dW)$ (and hence also $\prof(\dW)$).

  Finally, if \dW has a shadow valued in a derivator \dT, then so does $\dprof(\dW)$, defined for $H\in \prof(\dW)((A,R),(A,R)) = \dW({R,R})(A\times A\op)$ by
  \begin{equation}\label{eq:derivbicatshadow}
    \sh{H} = \int^A \symm^* \sh{H}_{A\times A\op}.
  \end{equation}
  Here $\sh{-}_{A\times A\op}$ denotes the shadow functor $\dW(R,R)(A\times A\op) \to \dT(A\times A\op)$, and $\symm$ is the symmetry as before.
  It follows that $\prof(\dW)$ also has a shadow valued in $\dT(\tc)$.
\end{thm}

\begin{proof}
  For the ordinary bicategory $\prof(\dW)$, the proof is essentially identical to the proof of~\cite[Theorem~5.9]{gps:additivity}.

  For the derivator bicategory, we note that just as a derivator \D induces a \emph{shifted} derivator $\shift\D A$ for any $A\in\cCat$, with $\shift\D A(B) = \D(A\times B)$, a closed derivator bicategory $\dW$ induces a shifted version $\shift\dW A$ with $\shift \dW A(R,S)(B) = \dW(R,S)(A\times B)$.
  This follows from~\cite[Example~8.14]{gps:additivity}.
  Thus, applying \autoref{thm:derivbicat-prof} 
  to $\shift\dW A$, and then letting $A$ vary, we obtain all the data and coherence axioms of $\dprof(\dW)$.

  For closedness of $\dprof(\dW)$, we must show that the composition morphisms of $\dprof(\dW)$ are two-variable left adjoints.
  However, the two-variable morphism
  \[ \dprof(\dW)((A,R),(B,S)) \times \dprof(\dW)((B,S),(C,T)) \xto{\odot} \dprof(\dW)((A,R),(C,T)) \]
  can be regarded as a shifted version of the composition morphism of $\dW$:
  \[ \shift{\dW(R,S)}{A\times B\op} \times \shift{\dW(S,T)}{B\times C\op} \to \shift{\dW(R,T)}{A\times C\op} \]
  in which $B$ is canceled but $A$ and $C$ are treated externally.
  This is a two-variable left adjoint by~\cite[Examples 8.15 and 8.16]{gps:additivity}.

  Suppose \dW has a shadow valued in \dT.
  By cocontinuity of this shadow,~\eqref{eq:derivbicatshadow} is equivalent to
  \begin{equation}
    \sh{H} = \Bigsh{ \int^A \symm^* H}_{\tc}.
  \end{equation}
  with $\int^A$ now denoting the coend in the derivator $\dW(R,R)$ rather than in $\dT$.
  For $H\in\prof(W)((A,R),(B,S))$ and $K\in \prof(\dW)((B,S),(A,R))$, we take the shadow isomorphism $\sh{H\odot K} \cong \sh{K\odot H}$ to be the composite
  \begin{align*}
    \sh{H\odot K}
    &\cong \int^A \Bigsh{\int^B (H\odot K)}_{A\times A\op}\\
    &\cong \int^A \int^B \sh{H\odot K}_{A\times A\op\times B\times B\op}\\
    &\cong \int^B \int^A \sh{K\odot H}_{A\times A\op\times B\times B\op}\\
    &\cong \int^B \Bigsh{ \int^A (K\odot H)}_{B\times B\op}\\
    &\cong \sh{K\odot H}
  \end{align*}
  (We have omitted the symmetry isomorphisms for brevity).
  Combining the argument that proves the associativity of composition in $\prof(\dV)$,~\cite[Lemma~5.12]{gps:additivity} with the shadow axiom for \dW proves the shadow axiom for $\prof(\dW)$.
  We extend this construction to $\dprof(\dW)$ by shifting, i.e.\ we define
  \[ \dprof(\dW)((A,R),(A,R)) \to \dT \]
  to be the composite
  \[ \shift{\dW(R,R)}{A\times A\op} \xto{\sh{-}} \shift{\dT}{A\times A\op} \xto{\int^A} \dT. \]
  Since shifting preserves cocontinuity, both of these morphisms are cocontinuous.
\end{proof}

In particular, this strengthens the result of~\cite[Theorem~5.9]{gps:additivity} which we cited in \S\ref{sec:der}: if \dV is a closed symmetric monoidal derivator, then not only do we have a bicategory $\prof(\dV)$, but we have a derivator bicategory $\dprof(\dV)$.

We suppose from now on that \dW is a closed derivator bicategory.
As before, we define the \textbf{$\Phi$-weighted colimit} of $X\in\dW(R,S)(A)$ by $\Phi\in\dW(R,R)(A\op)$ to be the composite $\colim^\Phi(X) = \Phi\odot X$ in $\prof(\dW)$, where $\Phi$ and $X$ are regarded as 1-cells $(\tc,R)\hto (A,R)$ and $(A,R)\hto (\tc,S)$ in $\prof(\dW)$, respectively.
We then have the analogue of \autoref{thm:colim-is-colim}.

\begin{prop}\label{thm:bicat-colim-is-colim}
  For any $X\in\dW(R,S)(A)$, if $\Phi = (\pi_{A\op})^*\lI_R$ is constant at the unit 1-cell of $R$, then we have $\colim^{\Phi}(X) \cong \colim(X)$, where $\colim(X)$ denotes the usual colimit $(\pi_A)_!(X)$ in the derivator $\dW(R,S)$.
\end{prop}
\begin{proof}
  Just like \autoref{thm:colim-is-colim}.
\end{proof}

\begin{defn}\ 
  \begin{itemize}
  \item A coherent diagram $X\in\dW(R,S)(A)$ is \textbf{pointwise dualizable} if it is right dualizable when regarded as a 1-cell $(A,R)\hto (\tc,S)$ in $\prof(\dW)$.
  \item A coherent diagram $\Phi\in\dW(R,R)(A\op)$ is \textbf{absolute} if it is right dualizable when regarded as a 1-cell $(\tc,R)\hto (A,R)$ in $\prof(\dW)$.
  \end{itemize}
\end{defn}

Thus, by \autoref{thm:compose-duals}, we have:

\begin{thm}
  If $X\in\dW(R,S)(A)$ is pointwise dualizable and $\Phi\in\dW(R,R)(A\op)$ is absolute, then $\colim^\Phi(X)$ is dualizable.
\end{thm}

We also have a version of \autoref{thm:smcpwdual}, whose proof is essentially identical.

\begin{lem}
  $X\in\dW(R,R)(A)$ is pointwise dualizable if and only if each object $X_a\in\dW(R,R)(\tc)$ is right dualizable in the underlying bicategory of \dW.
\end{lem}

Thus, by \autoref{thm:compose-traces}, we have a linearity formula.

\begin{thm}\label{thm:bicatderivlinearity}
  If $X\in\dW(R,S)(A)$ is pointwise dualizable and $\Phi\in\dW(R,R)(A\op)$ is absolute, then for any $P\in \dW(S,S)(\tc)$ and $f\colon X\to X\odot P$ we have
  \[ \tr(\colim^\Phi (f)) = \tr(f) \circ \tr(\id_\Phi). \]
\end{thm}

In other words, the following diagram commutes in $\dT(\tc)$:
\begin{equation}
  \vcenter{\xymatrix@C=3pc{
      \sh{R}\ar[r]^-{\tr(\id_\Phi)}\ar[dr]_{\tr(\colim^\Phi (f))} &
      \sh{(A,R)}\ar[d]^{\tr(f)}\\
      &
      \sh{P}
      }}
\end{equation}
Here we have identified the shadow $\sh{(\tc,R)}$ in $\prof(\dW)$ with the shadow $\sh{R}$ in \dW, which as usual is the shadow of the unit 1-cell $\lI_R$.
Our analysis of the shadows of units leading to eq.~\eqref{eq:LA} can be repeated essentially verbatim to conclude that here we have
\begin{align}
  \sh{(A,R)} &= (\pi_{\Lambda A})_!(\pi_{\Lambda A})^* \sh{R}\notag\\
  &= |NA| \tens \sh{R}.\label{eq:derLA}
\end{align}
(computed in the derivator \dT).
Similarly, every conjugacy class $[\alpha]$ in $A$ yields a uniquely determined morphism $\sh{R} \xto{[\alpha]} \sh{(A,R)}$ in $\dT(\tc)$.

For this theorem to be useful we need to be able to compute $\tr(f)$ and $\tr(\id_\Phi)$.  As before, we have a component lemma 
that enables us to compute $\tr(f)$.

\begin{lem}[The  component lemma for derivator bicategories]\label{thm:derbicatomega}
  If $X\in\dW(R,S)(A)$ is pointwise dualizable and $f\colon X\to X\odot P$, then for any conjugacy class $[a\xto{\alpha} a]$ in $A$, the composite
  \begin{equation}
    \xymatrix{ \sh{R} \ar[r]^-{[\alpha]} & \sh{(A,R)} \ar[r]^-{\tr(f)} & \sh{P} }
  \end{equation}
  is equal to the trace in $\dW(R,S)(\tc)$ of the composite
  \begin{equation}
    \xymatrix{ X_a \ar[r]^-{X_\alpha} & X_a \ar[r]^-{f_a} & X_a\odot P_a .}
  \end{equation}
\end{lem}

\begin{proof}We will prove a generalization of this result in \S\ref{sec:basechangebicat} on page \pageref{thm:compderivbicat}.
\end{proof}

Finally, in the semi-additive case we can deduce a more familiar-looking formula.

\begin{defn}
  If $\Phi\in\dW(R,R)(A\op)$ is absolute, then $\tr(\id_\Phi)\colon \sh{R} \to \sh{(A,R)}$ is called its \textbf{coefficient vector}.
  If the target derivator \dT of the shadow is semi-additive and we can express the coefficient vector of $\Phi$ as a linear combination
  \begin{equation}
    \tr(\id_\Phi) = \sum_{[\alpha]} \phi_{[\alpha]} \cdot [\alpha],\label{eq:philin}
  \end{equation}
  for $\phi_{[\alpha]}\colon \sh{R} \to \sh{R}$ in $\dT(\tc)$, then we refer to the $\phi_{[\alpha]}$ as the \textbf{coefficients} of $\Phi$ and say that $\Phi$ has a \textbf{coefficient decomposition}.
\end{defn}

\begin{cor}\label{thm:derbilin2}
  If $\dT$ is semi-additive and $\Phi\in \dW(R,R)(A\op)$ is absolute and has a coefficient decomposition, then we have
  \begin{equation}\label{eq:bicatlin2}
    \tr(\colim^\Phi f) = \sum_{[\alpha]} \phi_{[\alpha]} \cdot \tr(f_a \circ X_\alpha).
  \end{equation}
  for any pointwise dualizable $X\in \dW(R,S)(A)$, $P\in \dW(S,S)(\tc)$ and $f\colon X\to X\odot P$.
\end{cor}

All the examples from previous sections have generalizations to derivator bicategories.
To give some idea of these generalizations, and because we have a use for it in \cite{PS6}, we will give an outline of the 
generalization of \autoref{eg:cofiber}.

We say that a derivator bicategory \dW is \textbf{locally} semi-additive, stable, or $n$-divisible if each derivator $\dW(R,S)$ has the corresponding property.

\begin{thm}\label{eg:bicatcofib}
If \dW is a locally-stable closed derivator bicategory with a shadow valued in a stable derivator \dT, $X\in \dW(R,S)(\bbtwo)$ is pointwise dualizable and $f\colon X\to X$, then
  \[ \tr(\colim(f)) = \tr(f_b) - \tr(f_a). \]
\end{thm}

 Note that this is an equality of morphisms $\sh{R}\to \sh{S}$ in $\dT(\tc)$.
  We also have the twisted version, which applies to traces of any $f\colon X\to X\odot P$ with $P\in\dW(S,S)(\tc)$.

\begin{proof}[Outline of proof]
For any object $R$, let $\Phi\in \dW(R,R)(\bbtwo\op)$ be the essentially unique diagram of the form $(0\ot \lI_R)$.
  Just as in \autoref{eg:cofiber}, we can conclude that $\Phi$-weighted colimits are {\bf local cofibers}, i.e.\ cofibers in the hom-derivators of \dW, and that $\Phi$ is absolute whenever \dW is stable.

  As before, since $\Lambda \bbtwo $ is discrete on two objects, we have $\sh{(\bbtwo,R)} = \sh{R} \oplus \sh{R}$.
  Thus, $\tr(\id_\Phi)$ is determined by two morphisms $\phi_a,\phi_b \colon  \sh{R} \to \sh{R}$, and we have
  \[ \tr(\colim^\Phi (f)) = \phi_a \cdot \tr(f_a) + \phi_b \cdot \tr(f_b)\]
  for any $f\colon X\to X$.
  We calculate the coefficients exactly as before: the traces of identity maps of the cofibers of $(\lI_R\to \lI_R)$ and $(0\to\lI_R)$ are $0$ and $1$, respectively, while the trace of the identity of $\lI_R$ is $1$, so we have $0 = \phi_a+\phi_b$ and $1=0+\phi_b$, whence $\phi_b = 1$ and $\phi_a = -1$.
\end{proof}

We now sketch the promised application to the Reidemeister trace; see~\cite{PS6} for the details.

By a {\bf parametrized space} over a topological space $B$, we mean a space $X$, called the {\bf total space}, together with maps $B\xto{s}X\xto{p}B$ so that $p\circ s$ is the identity.
A bicategory whose 0-cells are topological spaces, whose 1-cells from $A$ to $B$ are parametrized spaces over $A\times B$ (or more precisely, parametrized spectra), and whose 2-cells are stable homotopy classes of maps of total spaces that commute with the maps $s$ and $p$ is constructed in~\cite{maysig:pht}. 
The bicategorical product, denoted $\odot$, is defined in terms of the pullback.
In~\cite{PS6} we will extend this to a derivator bicategory.

Let $i\colon X\into Y$ be an inclusion of spaces and suppose $f\colon Y\to Y$ is a continuous map so that $f(X)\subset X$.
Then we have induced maps of parametrized spaces
\[\widehat{f|_X}\colon S^0_X\to S^0_X\odot S_{f|_X}\text{ and }\widehat{f}\colon S^0_Y\to S^0_Y\odot S_{f}.\]
Here $S^0_Y$ is the parametrized space $Y\amalg Y$ over $Y$, and $S_f$ is the {\bf twisted path space} of $f$ (the space of triples $(y_1,y_2,\lambda)$ where $y_1,y_2\in Y$ and $\lambda$ is a path from $y_1$ to $f(y_2)$), regarded as a space over $Y\times Y$ (technically, with a disjoint section adjoined).
The definitions of $S^0_X$ and $S_{f|_X}$ are similar.

By composing with the inclusion $i$, we can define a parametrized space $i_!(S^0_X)$ over $Y$, with an induced map $i_!(\widehat{f|_X})\colon i_!(S^0_X)\to i_!(S^0_X)\odot S_f$.
We then have a map of $\bbtwo$-diagrams of parametrized spaces over $Y$:
\[\widehat{f_i}\colon \big(i_!(S^0_X)\to S^0_Y\big)\to \big(i_!(S^0_X)\to S_Y^0\big)\odot S_f\]
whose components are $i_!(\widehat{f|_X})$ and $\widehat{f}$.

It is shown in ~\cite{maysig:pht} that if $X$ and $Y$ are closed smooth manifolds or compact ENRs, then $S^0_X$ and $S^0_Y$ are dualizable in the derivator bicategory of parametrized spectra.
In this case, \autoref{eg:bicatcofib} implies 
\[\tr(\widehat{f})-\tr(i_!(\widehat{f|_X}))=\tr(\colim \widehat{f_i})\]
as maps $S^0 \to \sh{S_f}$.
Note that this is a twisted trace, with twisting object $P=S_f$.
The shadow of a parametrized space over $B\times B$ is its pullback along the diagonal (technically, with the section also quotiented out), so $\sh{S_f}$ is the {\bf twisted loop space} of $f$:
\[\Lambda^{f}Y\coloneqq \{\lambda\in Y^I\mid f(\lambda(0))=\lambda(1)\}.\]
Thus, $\tr(\widehat{f})$ is an element of the zeroth stable homotopy of $\Lambda^fY$; it can be identified with the Reidemeister trace of $f$.
Similarly, $\tr(i_!(\widehat{f_X}))$ can be identified with the image of the Reidemeister trace of $f|_X$ under the map
$\Lambda^{f|_X}X\to \Lambda^{f}Y$ induced by $i$.
The remaining piece, $\tr(\colim \widehat{f_i})$, is the \emph{relative Reidemeister trace} of $f$ from \cite{kate:relative}; it is a refinement of the Reidemeister trace of the induced map $f/X \colon  Y/X \to Y/X$ on the quotient space.
Thus, in the end we obtain the formula mentioned in the introduction:
\[R(f)-i(R(f|_X))=R_{Y|X}(f).\]

\begin{rmk}\label{eg:bicatpushout}
  The other examples from \S\S\ref{sec:eccat}--\ref{sec:ei-categories} also generalize directly to the bicategorical case.
  For instance, if \dW is locally stable and \dT is stable, and $X\in \dW(R,S)(\mathord\ulcorner)$ is a pointwise right dualizable span, then its pushout $\colim(X)$ is also right dualizable, and if $f\colon X\to X$ is an endomorphism, then
  \[ \tr(\colim(f)) = \tr(f_{c}) + \tr(f_{b}) - \tr(f_{a}) \]
  as morphisms $\sh{R} \to \sh{S}$.

  Similarly, if $A$ is a finite EI-category, and that \dW is locally stable and locally $\card{\nAut(\vec{\alpha})}$-divisible for each $\vec{\alpha}$, with a shadow valued in a stable and $\card{\nAut(\vec{\alpha})}$-divisible derivator \dT, then for any pointwise dualizable $X\in\dW(R,S)(A)$, its colimit $\colim(X)\in\dW(R,S)(\tc)$ is dualizable, and for any $f\colon X\to X$ we have
  \[ \tr(\colim(f)) = \sum_{[a]} \sum_{C} \tr(X_{C} \circ f_{a}) \cdot
  \sum_n (-1)^n \sum_{[\vec{\alpha}]} \sum_{\vec{C}} \frac{\card{\vec{C}}}{\card{\nAut(\vec{\alpha})}}. \]
  The notation is as in \autoref{thm:eicatdual}.

  We omit the proofs of these generalizations, since they are remarkably similar to those of \autoref{thm:pushouts} and \autoref{thm:eicatdual}.
\end{rmk}

\section{The uniqueness of linearity formulas}
\label{sec:uniqueness}

As an additional application of the extension of linearity to bicategorical trace, we now prove a uniqueness theorem for linearity formulas.
So far, we have shown that \emph{if} a weight $\Phi$ is absolute, \emph{then} $\Phi$-weighted colimits preserve dualizability, and moreover the coefficient vector of $\Phi$ yields \emph{a} linearity formula.
Now we will reverse these implications: we show that \emph{if} $\Phi$-weighted colimits preserve dualizability, \emph{then} $\Phi$ is absolute; and that moreover in this case, the coefficient vector of $\Phi$ is the \emph{only} linearity formula.

The main idea involved in the proof of the latter statement should not be surprising since we have already seen it in many examples.
Namely, once we know that a linearity formula exists, we can deduce what its coefficients \emph{must} be, by considering some simple examples.
There are even a canonical set of examples to consider, namely the representable diagrams.
This naive method doesn't quite work in all cases (for instance, we weren't able to use it to calculate the coefficients for the orbit-counting theorem), but a refinement of it does: we must consider not the individual representable diagrams, but their totality, as a 1-cell in the bicategory of profunctors.

For this we need linearity formulas that apply to bicategorical traces.
It follows that the generality of linearity for bicategorical trace is necessary even to \emph{state} the uniqueness theorem.

First we introduce a bit of terminology, to make precise the location in which our ``coefficients'' live when comparing linearity formulas in different categories.
Let $Z$ be a commutative semiring; a \textbf{$Z$-module} is a commutative monoid equipped with an associative, unital, bilinear action of $Z$.
We say that a derivator \D is \textbf{$Z$-linear} if each category $\D(A)$ is $Z$-linear (i.e.\ enriched over $Z$-modules) and each functor $u^*\colon \D(B) \to \D(A)$ is likewise $Z$-linear.
This implies that the functors $u_!$ and $u_*$ are also $Z$-linear.

We say that a derivator bicategory \dW is \textbf{locally $Z$-linear} if each derivator $\dW(R,S)$ is $Z$-linear and the composition morphisms are $Z$-bilinear.
In particular, for any object $R$ of \dW, we have a semiring homomorphism $Z \to \dW(R,R)(\tc)(\lI_R,\lI_R)$, defined by multiplying by the identity.
If \dW has a shadow valued in \dT, we say that the shadow is \textbf{$Z$-linear} if $\dT$ is $Z$-linear and the shadow morphisms $\dW(R,R) \to \dT$ are $Z$-linear.

For example, any locally semi-additive \dW is $\mathbb{N}$-linear.
If it is locally additive (such as if it is locally stable), then it is $\mathbb{Z}$-linear.
And it is $n$-divisible, as defined in \S\ref{sec:deriv-burnside}, exactly when it is $\mathbb{N}[\frac{1}{n}]$-linear.
(In these cases linearity is a mere property rather than a structure, since the unique map $\mathbb{N}\to Z$ from the initial semiring $\mathbb{N}$ is an epimorphism.)

Now we can state the uniqueness theorem.
For simplicity, we consider only the case of constant weights (``conical colimits''); for the general case we would need a way to say in what sense two weights in different bicategories are ``the same''.

\begin{thm}\label{thm:uniqueness}
  Let $Z$ be a commutative semiring and $\fB$ be a class of closed, locally semi-additive, $Z$-linear derivator bicategories with $Z$-linear shadows.
  Assume that \fB is closed under $\dprof$, i.e.\ if $\dW\in\fB$ then $\dprof(\dW)\in\fB$.
  Let $A$ be a finite category, and assume the following.
  \begin{enumerate}
  \item If $\dW\in\fB$ and $X\in \dW(R,S)(A)$ is pointwise right dualizable, then $\colim(X)$ is right dualizable.\label{item:uniqh1}
  \item There are elements $\phi_{[\alpha]}\in Z$ indexed by the conjugacy classes of $A$ such that for any $\dW$ and $X$ as in~\ref{item:uniqh1} and any $f\colon X\to X$, we have\label{item:uniqh2}
    \[ \tr(f) = \sum_{[\alpha]} \phi_{[\alpha]} \tr(f_a \circ X_\alpha). \]
  \end{enumerate}
  Then for any $\dW\in\fB$ and any object $R$ of $\dW$, the constant diagram $(\pi_{A\op})^*{\lI_R} \in \dW(R,R)(A\op)$ is absolute and has a coefficient decomposition, and its coefficients are (the images in \dT of) the $\phi_{[\alpha]}$.
\end{thm}
\begin{proof}
  See page~\pageref{proof:uniqueness} in \S\ref{sec:bco-der}.
\end{proof}

To explain the role of $Z$ and $\fB$, we show how this theorem applies to a few of our previous examples.

\begin{eg}\label{eg:uniqueness-dirsum}
  Consider the case $A=\{a,b\}$ of direct sums, as in \autoref{eg:dirsum}.
  Let $Z=\mathbb{N}$ and let $\fB$ be the class of \emph{all} closed, locally semi-additive derivator bicategories with shadows.
  This is clearly closed under $\dprof$.
  As remarked in \autoref{eg:dirsum} (in the monoidal case), it is not hard to prove by explicit calculation that if $X$ and $Y$ are right dualizable in such a bicategory, then so is $X\oplus Y$, and that for any $f\colon X\to X$ and $g\colon Y\to Y$ we have $\tr(f\oplus g) = \tr(f)+\tr(g)$.
  Therefore, \autoref{thm:uniqueness} implies that the weights $(\pi_{A\op})^*{\lI_R}$ for coproducts are right dualizable in any locally semi-additive derivator bicategory, and their coefficients are $1$ and $1$.
\end{eg}

\begin{eg}\label{eg:uniqueness-pushout}
  Let $A=\mathord\ulcorner$ be the diagram for pushouts, as in \autoref{thm:pushouts} and \autoref{eg:bicatpushout}.
  Let $Z=\mathbb{Z}$ and let $\fB$ be the class of \emph{locally stable} closed derivator bicategories with shadows valued in a stable derivator; this is also closed under $\dprof$.
  Combining the Mayer-Vietoris sequence of~\cite{gps:stable} with~\cite[Lemma 11.6]{gps:additivity}, it follows that the pushout of right dualizable 1-cells in such a bicategory is again right dualizable.
  Now~\cite{gps:additivity} showed, by generalizing the method of~\cite{add}, that the additivity formula~\eqref{eq:cofiber-add} holds in any stable closed symmetric monoidal derivator.
  It might be possible to generalize this method to apply to derivator bicategories as well.
  If so, then in \autoref{thm:pushouts}, from this and the additivity of direct sums we would obtain the additivity formula~\eqref{eq:pushout-add} for pushouts.
  Therefore, \autoref{thm:uniqueness} would imply that the weights for pushouts are right dualizable in any locally stable derivator bicategory, and their coefficients are $1$, $1$, and $-1$ as in \autoref{thm:pushouts}.
\end{eg}

Note that in both of these cases, the uniqueness theorem is calculating the coefficient vector in essentially the same way that we did before: namely, once we know that a linearity formula exists, we can look at simple examples to identify its coefficients.
Examples~\ref{eg:uniqueness-dirsum} and~\ref{eg:uniqueness-pushout} do this in a more redundant way, deriving the general linearity formula first and then restricting it to the simple examples to conclude that its coefficients coincide with those of the coefficient vector.

There are other examples, however, where it does seem to be easier to proceed in this way.
Specifically, if $A$-colimits can be constructed out of smaller colimits in some concrete way, then it is sometimes easier to derive the linearity formula for $A$-colimits directly from those for the smaller colimits, rather than to explicitly analyze the coefficient vector of $A$.
\autoref{thm:uniqueness} tells us that in this case, the coefficients in the resulting linearity formula are actually automatically the components of the coefficient vector itself.

\begin{eg}
  In \autoref{thm:eccat} we proved a linearity formula for homotopy finite colimits that holds in any stable, closed symmetric monoidal derivator, and in \autoref{eg:bicatpushout} we observed that essentially the same proofs apply to any locally-stable closed derivator bicategory.
  The proof involved constructing such colimits out of pushouts, rather than directly analyzing the weight for absoluteness and calculating its coefficient vector.
  However, if we let $Z=\mathbb{Z}$ and $\fB$ be the class of locally-stable closed derivator bicategories, then \autoref{thm:uniqueness} implies that in fact, for any homotopy finite $A$, the weight $(\pi_{A\op})^*{\lI_R}$ is absolute in any $\dW\in\fB$, and its coefficients are precisely those appearing in \autoref{thm:eccat}.

  Similarly, let $A$ be a finite EI-category, let $Z=\mathbb{Z}[S^{-1}]$ where $S$ is the set of cardinalities of automorphism groups of objects of $A$, and let $\fB$ be the class of $S$-divisible locally-stable closed derivator bicategories.
  Then from \autoref{thm:eicatdual} and \autoref{thm:uniqueness}, we conclude that $(\pi_{A\op})^*{\lI_R}$ is absolute in any $\dW\in\fB$, and its coefficients are those appearing in \autoref{thm:eicatdual}.
\end{eg}

To be honest, in concrete applications the uniqueness theorem is not very important: usually what we care about is \emph{having} a linearity formula, not about whether it comes from a coefficient vector.
However, it does reassure us that our theory is ``complete'', in the sense that it includes all sufficiently general linearity formulas.

We end this section with two amusing applications.
First, we give a quick proof that stable implies semi-additive in the monoidal or bicategorical case.
(Compare to the proof of~\cite[Proposition 4.7]{groth:ptstab}, which applies in any stable derivator.)

\begin{prop}\label{thm:semiadditive}
  Any locally-stable closed derivator bicategory is locally semi-additive (and hence locally additive, by \autoref{rmk:additive}).
\end{prop}
\begin{proof}
  Suppose \dW is locally stable.
  Then for any 1-cells $X,Y\in\dW(R,S)(\tc)$, we have two cocartesian squares
  \begin{equation}
    \vcenter{\xymatrix{
        \Omega X\ar[r]\ar[d] &
        0\ar[d]\\
        0\ar[r] &
        X
      }}
    \qquad\text{and}\qquad
    \vcenter{\xymatrix{
        0\ar[r]\ar[d] &
        Y\ar[d]\\
        0\ar[r] &
        Y
      }}
  \end{equation}
  (the first being cocartesian in addition to cartesian because of stability).
  Taking their coproduct, we have a cocartesian square
  \begin{equation}
    \vcenter{\xymatrix{
        \Omega X\ar[r]\ar[d] &
        Y\ar[d]\\
        0\ar[r] &
        X+Y,
      }}
  \end{equation}
  so $X+Y$ is the cofiber of a coherent morphism $\Omega X \to Y$.
  Now if $X$ and $Y$ are right dualizable, so is $\Omega X \cong \Omega \lI_R \odot X$ by stability; thus by \autoref{eg:cofiber} $X+Y$ is also right dualizable.

  We have shown that in any locally-stable closed derivator bicategory, the coproduct of two right dualizable 1-cells is right dualizable.
  Therefore, the first part of the uniqueness theorem implies that coproducts are absolute in any such $\dW$, which is to say that it is locally semi-additive.
\end{proof}

Second, we give another proof of \autoref{thm:realiz-coeff}.

\label{proof:realiz-coeff}
\begin{replem}{thm:realiz-coeff}
  For any $n$, we have
  \begin{equation}
    \sum_{k\ge 0} (-1)^k \cdot \card {\left\{
      \parbox{5.4cm}{\centering composable strings of $k$ nonidentity\\ face maps starting at $[n] \in \bbDelta\op$}
    \right\}}
    = (-1)^n.
  \end{equation}
\end{replem}
\begin{proof}
  Applying formula~\eqref{eq:eccattr} to the strictly homotopy finite category $(\bbDelta_m')\op$ for some $m\ge n$, we conclude that for any truncated semisimplicial diagram $Y$ and $g\colon Y\to Y$ we have
  \begin{equation}
    \tr(\colim(g))
    = \sum_{n\le m} \tr(g_n) \cdot \sum_{k\ge 0} (-1)^k \cdot \card {\left\{
      \parbox{5.2cm}{\centering composable strings of nonidentity\\arrows of length $k$ starting at $[n]$}
    \right\}}
  \end{equation}
  However, by~\eqref{eq:realiztr} we also have
  \begin{equation}
    \tr(\colim(g)) = \sum_{n \le m} (-1)^k \tr(g_n).
  \end{equation}
  The claim now follows from the Uniqueness \autoref{thm:uniqueness}.
\end{proof}

Karol Szumi\l{}o has pointed out that this proof is basically the same as that given in \S\ref{sec:ei-categories}, since the construction of $(\bbDelta_m')\op$-colimits by the method of \autoref{thm:hofin-constr} amounts to barycentrically subdividing it.

\part{Base change objects and component lemmas}
\label{part:formal}

In this fourth and final part of the paper, we complete the proofs of the component lemmas.
This consists mainly of calculations with derivators using the tools of~\cite{groth:ptstab,gps:stable,gps:additivity}.
However, to simplify some of these calculations, we first introduce a slightly more abstract framework for bicategorical trace.

\section{The abstract theory of base change objects}
\label{sec:framed}
\label{sec:base_change}

In this section we introduce an abstract framework that will enable us to give a high-level proof of \autoref{thm:smderomega}, \autoref{thm:bicatomega}, \autoref{thm:derbicatomega}, and other similar statements.
It is based on the notion of \emph{framed bicategory}~\cite{shulman:frbi} (also known as a \emph{proarrow equipment}~\cite{wood:proarrows-i}), which is an enhancement of a bicategory that includes \emph{base change objects} such as the representable profunctors $B(\id,f)$ and $B(f,\id)$.
This will allow us to establish a very general form of the component lemmas, 
in \autoref{thm:general-omega}.
In the next two sections we will apply this to bicategories and derivators.

We start by recalling the definitions.

\begin{defn}
  A \textbf{(pseudo) double category} \lW consists of:
  \begin{itemize}
  \item A category $\lW_0$, whose objects we call objects of \lW and whose morphisms we call vertical arrows.
  \item A category $\lW_1$, whose objects we call horizontal arrows and whose morphisms we call $2$-cells.
  \item Functors $S,T\colon \lW_1\to \lW_0$ called horizontal source and horizontal target.
  \item A {composition} functor $\odot\colon  \lW_1 \times_{\lW_0} \lW_1 \to \lW_0$.
  \item A {unit} functor $\lI\colon \lW_0 \to \lW_1$.
  \item Associativity and unit isomorphisms for composition, satisfying appropriate axioms.
  \end{itemize}
\end{defn}

We draw a 2-cell $\phi\colon M\to N$ as
\begin{equation}
  \vcenter{\xymatrix{
      S(M)\ar[r]|{|}^M\ar[d]_{S(\phi)} \drtwocell\omit{\phi} &
      T(M)\ar[d]^{T(\phi)}\\
      S(N)\ar[r]|{|}_{N} &
      T(N)
      }}
\end{equation}
Composition in $\lW_1$ is vertical pasting of such squares, while the composition functor $\odot$ yields a horizontal pasting.

Every double category has a \textbf{horizontal bicategory} obtained by neglecting the vertical arrows.
If \lW is a double category, we denote its horizontal bicategory by \W.
Conversely, any bicategory can be regarded as a double category whose only vertical arrows are identities; thus any monoidal category can similarly be regarded as a double category with one object and one vertical arrow.

Just as any bicategory is equivalent to a strict 2-category, any pseudo double category is equivalent to a strict one (see~\cite{gp:double-limits}).
Thus, we will write as if our double categories were strict, even though the examples we care about are not.

\begin{defn}[{\cite[Theorem~4.1]{shulman:frbi}}]
  The following three conditions on a double category \lW are equivalent; when they hold we call it a \textbf{framed bicategory}.
  \begin{itemize}
  \item $(S,T)\colon \lW_1 \to \lW_0 \times \lW_0$ is a categorical fibration.
  \item $(S,T)\colon \lW_1 \to \lW_0 \times \lW_0$ is a categorical opfibration.
  \item For every vertical arrow $f\colon A\to B$, there exist horizontal arrows $B(\id,f)\colon A\hto B$ and $B(f,\id)\colon B\hto A$ and 2-cells
    \begin{equation}
      \vcenter{\xymatrix{
          A\ar[r]|{|}^{\lI_A}\ar@{=}[d] \drtwocell\omit{\alpha} &
          A\ar[d]^f
          &
          A\ar[r]|{|}^{B(\id,f)}\ar[d]_f \drtwocell\omit{\beta} &
          B\ar@{=}[d]
          &
          A\ar[r]|{|}^{\lI_A} \ar[d]_f \drtwocell\omit{\eta} &
          A\ar@{=}[d]
          &
          B\ar[r]|{|}^{B(f,\id)} \ar@{=}[d] \drtwocell\omit{\epsilon} &
          A\ar[d]^f
          \\
          A\ar[r]|{|}_{B(\id,f)} &
          B
          &
          B\ar[r]|{|}_{\lI_B} &
          B
          &
          B\ar[r]|{|}_{B(f,\id)} &
          A
          &
          B\ar[r]|{|}_{\lI_B} &
          B
        }}
    \end{equation}
    such that the following four composites are identities.
    \begin{equation}
      \vcenter{\xymatrix{
          \ar@{=}[r]\ar@{=}[d] \drtwocell\omit{\alpha} &
          \ar[r]\ar[d] \drtwocell\omit{\beta} &
          \ar@{=}[d]\\
          \ar[r] &
          \ar@{=}[r] &
        }}
      \qquad
      \vcenter{\xymatrix{
          \ar@{=}[r]\ar@{=}[d] \drtwocell\omit{\alpha} &
          \ar[d]\\
          \ar[r]\ar[d] \drtwocell\omit{\beta}&
          \ar@{=}[d]\\
          \ar@{=}[r] &
        }}
      \qquad
      \vcenter{\xymatrix{
          \ar[r]\ar@{=}[d] \drtwocell\omit{\epsilon} &
          \ar@{=}[r]\ar[d] \drtwocell\omit{\eta} &
          \ar@{=}[d]\\
          \ar@{=}[r] &
          \ar[r] &
        }}
      \qquad
      \vcenter{\xymatrix{
          \ar@{=}[r]\ar[d] \drtwocell\omit{\eta} &
          \ar@{=}[d]\\
          \ar[r]\ar@{=}[d] \drtwocell\omit{\epsilon}&
          \ar[d]\\
          \ar@{=}[r] &
        }}
    \end{equation}
  \end{itemize}
\end{defn}

Any bicategory (or monoidal category), regarded as a double category, is a framed bicategory.
The other primary examples we have in mind are extensions of the bicategories $\prof(\V)$ and $\prof(\W)$ from previous sections, whose vertical arrows come from functors between small categories.
These can be obtained by applying the following two general constructions.

\begin{thm}[{\cite[Prop.~11.10]{shulman:frbi}}]
  If \lW is a framed bicategory with local coequalizers (i.e.\ coequalizers in each category $\W(A,B)$ that are preserved by $\odot$ in each variable), then there is a framed bicategory $\lMod(\lW)$ described as follows.
  \begin{itemize}
  \item Its objects are monads in \W: endo-1-morphisms $M\colon A\hto A$ together with a multiplication $M\odot M\to M$ and unit $\lI_A \to M$ satisfying associativity and unit axioms.
  \item Its horizontal arrows from $M\colon A\hto A$ to $N\colon B\hto B$ are modules in \W: horizontal arrows $H\colon A\hto B$ together with actions $M\odot H\to H$ and $H\odot N\to H$ that are associative and unital and commute with each other.
  \item Its vertical arrows from $M\colon A\hto A$ to $N\colon B\hto B$ are pairs $(f,\phi)$ of a vertical arrow $f\colon A\to B$ and a 2-cell
    \begin{equation}
      \vcenter{\xymatrix{
          R\ar[r]|{|}^M\ar[d]_f \drtwocell\omit{\phi} &
          R\ar[d]^f\\
          S\ar[r]|{|}_N &
          S
        }}
    \end{equation}
    in \lW, which commute with the monad structures of $M$ and $N$.
  \item Its 2-cells are 2-cells in \lW which commute with the module actions of their domain and codomain.
  \item The composite of modules $H\colon M\hto N$ and $K\colon N\hto P$ is the local coequalizer of the two actions
    \[ H\odot N\odot K \toto K\odot K. \]
  \item The horizontal unit $\lI_M$ of a monad $M\colon A\hto A$ is the horizontal arrow $M$ itself, regarded as a module.
  \end{itemize}
\end{thm}
\begin{proof}
 See~\cite[Prop.~11.10]{shulman:frbi}.
\end{proof}

\begin{thm}
  If \lW is a framed bicategory with local coproducts, then there is a framed bicategory $\lMat(\lW)$ described as follows.
  \begin{itemize}
  \item Its objects are set-indexed families $(A_i)_{i\in I}$ of objects of \lW.
  \item Its horizontal arrows from $(A_i)_{i\in I}$ to $(B_j)_{j\in J}$ are matrices in \W: families of horizontal arrows $(H_{ij}\colon A_i\hto B_j)_{i\in I, j\in J}$.
  \item Its vertical arrows from $(A_i)_{i\in I}$ to $(B_j)_{j\in J}$ $M\colon A\hto A$ are pairs $(f,\phi)$ of a function $f\colon I\to J$ and a family $(\phi_i\colon A_i \to B_{f(i)})_{i\in I}$ of vertical arrows in \lW.
  \item Its 2-cells are families of 2-cells in \lW.
  \item The composite of matrices $(H_{ij}\colon A_i\hto B_j)_{i\in I, j\in J}$ and $(K_{jl}\colon B_j\hto C_l)_{j\in J, l\in L}$ is the family of local coproducts
    \[ \Big( \coprod_{j\in J} H_{ij} \odot K_{jl} \Big)_{i\in I, l\in L}. \]
  \item The horizontal unit $\lI_A$ of a family $(A_i)_{i\in I}$ is the identity matrix defined by
    \[ (\lI_A)_{ii'} =
    \begin{cases}
      \lI_{A_i} &\quad i=i'\\
      \emptyset & \quad i\neq i'.
    \end{cases}
    \]
  \end{itemize}
\end{thm}
\begin{proof}
  A straightforward modification of the proof of~\cite[Prop.~11.10]{shulman:frbi}.
\end{proof}

Any local colimits possessed by \lW are inherited by $\lMod(\lW)$ and $\lMat(\lW)$.
In particular, if \lW is locally cocomplete, we can construct $\lMod(\lMat(\lW))$.

If we regard a cocomplete closed symmetric monoidal category \V as a framed bicategory, then the objects, vertical arrows, and horizontal arrows of $\lMod(\lMat(\V))$ are exactly \V-enriched categories, functors, and profunctors respectively.
Recall that a \V-profunctor between \V-categories $A\hto B$ is equivalently a \V-enriched functor $A\otimes B\op \to \V$.
In \S\ref{sec:moncat} we defined \V-profunctors between unenriched categories, but a \V-profunctor $A\hto B$ in this sense can be identified with a \V-profunctor in the enriched sense from $\V[A]$ to $\V[B]$, where $\V[A]$ is the \V-category ``freely generated'' by the ordinary category $A$, with $\V[A](a,a') = A(a,a') \cdot \lS$.
Thus, the bicategory $\prof(\V)$ from \S\ref{sec:moncat} is equivalent to the full sub-bicategory of the horizontal bicategory of $\lMod(\lMat(\lW))$ determined by the objects $\V[A]$.
In this case, the representable horizontal arrows $B(\id,f)$ and $B(f,\id)$ defined in the abstract framed-bicategory context agree with the representable profunctors from \S\ref{sec:moncat}.

With these examples in mind, for a general locally cocomplete framed bicategory \lW we denote $\lMod(\lMat(\lW))$ by $\lProf(\lW)$ and refer to its objects, vertical arrows, and horizontal arrows as \textbf{\lW-enriched categories}, \textbf{functors}, and \textbf{profunctors} respectively.
The more classical case is when \lW is a locally cocomplete bicategory \W; see~\cite{walters:sheaves-cauchy-1,street:cauchy-enr,ckw:axiom-mod}.
In this case, the bicategory $\prof(\W)$ constructed in \S\ref{sec:bicat} sits inside $\lProf(\W)$ similarly to the monoidal case, where the pair $(A,R)$ corresponds to the \W-category whose object family is $(R)_{a\in \mathrm{ob}(A)}$ and whose matrix monad is $(A(a,a')\cdot \lI_R)_{a,a'\in \mathrm{ob}(A)}$.

We now recall some basic properties of framed bicategories.

\begin{lem}\label{lem:basechangeder}
  For any vertical arrow $f\colon A\to B$ in a framed bicategory, $B(f,\id)$ is right dual to $B(\id,f)$ in the horizontal bicategory.
\end{lem}
\begin{proof}
  The evaluation and coevaluation are the composites
  \begin{equation}
    \vcenter{\xymatrix{
        \ar[r]\ar@{=}[d] \drtwocell\omit{\epsilon} &
        \ar[r]\ar[d] \drtwocell\omit{\beta} &
        \ar@{=}[d]\\
        \ar@{=}[r] &
        \ar@{=}[r] &
      }}
    \qquad\text{and}\qquad
    \vcenter{\xymatrix{
        \ar@{=}[r]\ar@{=}[d] \drtwocell\omit{\alpha} &
        \ar@{=}[r]\ar[d] \drtwocell\omit{\eta} &
        \ar@{=}[d]\\
        \ar[r] &
        \ar[r] &.
      }}
  \end{equation}
  The required identities follow from those for $\alpha$, $\beta$, $\eta$, and $\epsilon$.
\end{proof}

\begin{lem}\label{thm:bco-ff}
  In a framed bicategory, there is a natural bijection between 2-cells
  \begin{equation}
    \vcenter{\xymatrix{
        R\ar[r]|{|}^{H}\ar[d]_g \drtwocell\omit{} &
        R'\ar[d]^f\\
        S\ar[r]|{|}_{K} &
        S'
      }}
  \end{equation}
  and 2-cells $H\odot B(\id,f)\to B(\id,g) \odot K$ in $\W(R,S)$.
\end{lem}
\begin{proof}
  The bijection is defined via pasting 
  \begin{equation}
    \raisebox{.7cm}{$
    \vcenter{\xymatrix{
        R\ar[r]|{|}^{H}\ar[d]_g \drtwocell\omit{} &
        R'\ar[d]^f\\
        S\ar[r]|{|}_{K} &
        S'
      }}\quad\mapsto \quad
    \vcenter{\xymatrix{R\ar[r]|{|}^{\lI_R}\ar@{=}[d] \drtwocell\omit{\alpha} &
        R\ar[r]|{|}^{H}\ar[d]_g \drtwocell\omit{} &
        R'\ar[d]^f\ar[r]|{|}^{B(\id,f)}\drtwocell\omit{\beta}&S'\ar@{=}[d]\\
        R\ar[r]|{|}_{B(\id,g)}&S\ar[r]|{|}_{K} &
        S'\ar[r]|{|}_{\lI_{S'}}&S'
      }}
    $}\qedhere
  \end{equation}
\end{proof}

Next we extend the definition of a shadow from bicategories to framed bicategories.
Given a double category \lW, let $\loops\lW$ denote the \textbf{category of endo-1-cells} in \lW, defined as follows:
\begin{itemize}
\item Its objects are pairs $(R,H)$ where $R$ is an object of \lW and $H\colon R\hto R$ is a horizontal arrow.
\item Its morphisms from $(R,H)$ to $(S,K)$ are pairs $(f,\phi)$ where $f\colon R\to S$ is a vertical arrow and $\phi$ is a 2-cell
  \begin{equation}
    \vcenter{\xymatrix{
        R\ar[r]|{|}^H\ar[d]_f \drtwocell\omit{\phi} &
        R\ar[d]^f\\
        S\ar[r]|{|}_K &
        S.
      }}
  \end{equation}
\end{itemize}

Note that for any object $R$, the horizontal hom-category $\W(R,R)$ is a non-full subcategory of $\loops\lW$.

\begin{defn}
  A \textbf{shadow} on a double category \lW is a functor
  \[ \sh{-} \colon  \loops\lW \to \T \]
  to some other category $\T$, such that
  \begin{itemize}
  \item the composite functors $\W(R,R) \to \loops\lW \to \T$ equip the bicategory $\W$ with a shadow, and
  \item the isomorphisms $\sh{H\odot K} \cong \sh{K\odot H}$ of this shadow are natural with respect to all morphisms in $\loops\lW$.
    By this we mean that given 2-cells
    \begin{equation}
      \vcenter{\xymatrix{
          R\ar[r]|{|}^H\ar[d]_f \drtwocell\omit{\phi} &
          S\ar[d]^g\\
          R'\ar[r]|{|}_{H'} &
          S'
        }}
      \qquad\text{and}\qquad
      \vcenter{\xymatrix{
          S\ar[r]|{|}^K\ar[d]_g \drtwocell\omit{\psi} &
          R\ar[d]^f\\
          S'\ar[r]|{|}_{K'} &
          R'
        }}
    \end{equation}
    in \lW, so that
    \begin{align*}
      (f,\phi\odot\psi) &\colon (R,H\odot K) \to (R',H'\odot K') \qquad\text{and}\\
      (g,\psi\odot\phi) &\colon (S,K\odot H) \to (S',K'\odot H')
    \end{align*}
    are morphisms in $\loops\lW$, the following square commutes in \T:
    \begin{equation}
      \vcenter{\xymatrix{
          \sh{H\odot K}\ar[r]^{\cong}\ar[d]_{\sh{(f,\phi\odot\psi)}} &
          \sh{K\odot H}\ar[d]^{\sh{(g,\psi\odot\phi)}}\\
          \sh{H'\odot K'}\ar[r]_{\cong} &
          \sh{K'\odot H'}.
        }}
    \end{equation}
  \end{itemize}
\end{defn}

Our central examples inherit this structure.

\begin{thm}
  If \lW is a locally cocomplete framed bicategory with a cocontinuous shadow valued in a cocomplete category \T, then $\lMod(\lW)$ and $\lMat(\lW)$ (hence also $\lProf(\lW)$) also have a shadow valued in \T.
\end{thm}

This generalizes the construction in \cite[\S9.3]{kate:traces}.  

\begin{proof}

  For $\lMod(\lW)$, we define the shadow of a module $H\colon M\hto M$ to be the coequalizer of the two maps
  \[ \sh{M\odot H} \toto \sh{H} \]
  in \T.
  The first of these two maps is the left action of $M$ on $H$, while the second is the right action composed with the shadow isomorphism.
  The shadow of a module 2-cell is induced using the universal property of coequalizers and the fact that the shadow of \lW acts on 2-cells in \lW.
  The axioms are straightforward to verify.

  Similarly, for $\lMat(\lW)$, we define the shadow of a matrix $(H_{ii'}\colon A_i\hto A_{i'})_{i,i'\in I}$ to be the coproduct
  \[ \coprod_{i\in I} \sh{H_{ii}}. \]
  The rest of the structure is analogous.
\end{proof}

The following is the central abstract result  we will use to deduce all the variants of the component lemma. 

\begin{thm}\label{thm:general-omega}
  Suppose given a 2-cell in a framed bicategory with a shadow:
  \begin{equation}
    \vcenter{\xymatrix{
        R\ar[r]|{|}^H\ar[d]_f \drtwocell\omit{\phi}&
        R\ar[d]^f\\
        S\ar[r]|{|}_K &
        S
      }}
  \end{equation}
  and let $B(\id,\phi)\colon  H\odot B(\id,f) \to B(\id,f)\odot K$ denote the corresponding 2-cell obtained from \autoref{thm:bco-ff}.
  Since $B(\id,f)$ is right dualizable, $B(\id,\phi)$ has a trace in the horizontal bicategory \W.
  On the other hand, we can directly apply the shadow $\sh{-}\colon \loops\lW \to \T$ to $(f,\phi)$; then
  \[ \tr(B(\id,\phi)) = \sh{(f,\phi)} \]
  as morphisms $\sh{H} \to \sh{K}$ in $\T$.
\end{thm}
\begin{proof}
  Invoking the definitions of $B(\id,\phi)$ and of the evaluation and coevaluation in terms of the framed bicategory structure, we find that $\tr(B(\id,\phi))$ is the following composite.
  We have drawn pasting composites of 2-cells in \lW with $\sh{-}$ markers around them to indicate application of $\sh{-}\colon \loops\lW\to\T$; thus the picture below denotes a morphism in $\T$.
  \begin{equation}
  \vcenter{\xymatrix@C=3pc{
      \makebox[-1mm]{$\Big\langle$}\makebox[0mm]{$\Big\langle$} &
      &
       &
      \ar[r]^H \ar@{=}[d] &
      \ar@{=}[r]\ar@{=}[d] \drtwocell\omit{\alpha} &
      \ar@{=}[r]\ar[d]^f \drtwocell\omit{\eta} &
      \ar@{=}[d]
      & \makebox[-1mm]{$\Big\rangle$}\makebox[0mm]{$\Big\rangle$}
      \\
      \makebox[-1mm]{$\Big\langle$}\makebox[0mm]{$\Big\langle$} &
       &
      \ar@{=}[r]\ar@{=}[d]\drtwocell\omit{\alpha} &
      \ar[r]^H\ar[d]_f \drtwocell\omit{\phi} &
      \ar[r]^{B(\id,f)}\ar[d]^f \drtwocell\omit{\beta} &
      \ar[r]^{B(f,\id)}\ar@{=}[d] &
      \ar@{=}[d]
      & \makebox[-1mm]{$\Big\rangle$}\makebox[0mm]{$\Big\rangle$}
      \\
      \makebox[-1mm]{$\Big\langle$}\makebox[0mm]{$\Big\langle$} &
       &
      \ar[r]_{B(\id,f)} &
      \ar[r]^K \ar@{}[dr]|{\cong} &
      \ar@{=}[r] &
      \ar[r]^{B(f,\id)} &
      & \makebox[-1mm]{$\Big\rangle$}\makebox[0mm]{$\Big\rangle$}
      \\
      \makebox[-1mm]{$\Big\langle$}\makebox[0mm]{$\Big\langle$} &
      \ar[r]^{B(f,\id)}\ar@{=}[d] \drtwocell\omit{\epsilon} &
      \ar[r]^{B(\id,f)}\ar[d]^f \drtwocell\omit{\beta} &
      \ar[r]^K \ar@{=}[d] & \ar@{=}[d]
      &&& \makebox[-1mm]{$\Big\rangle$}\makebox[0mm]{$\Big\rangle$}
      \\
      \makebox[-1mm]{$\Big\langle$}\makebox[0mm]{$\Big\langle$} &
      \ar@{=}[r] &
      \ar@{=}[r] &
      \ar[r]_K &
      &&& \makebox[-1mm]{$\Big\rangle$}\makebox[0mm]{$\Big\rangle$}
      }}
  \end{equation}
  One of the defining equations of $B(\id,f)$ enables us to simplify this to
  \begin{equation}
  \vcenter{\xymatrix@C=3pc{
      \makebox[-1mm]{$\Big\langle$}\makebox[0mm]{$\Big\langle$} &
       &
      \ar@{=}[r]\ar@{=}[d]\drtwocell\omit{\alpha} &
      \ar[r]^H\ar[d]_f \drtwocell\omit{\phi} &
      \ar@{=}[r]\ar@{=}[d]^f \drtwocell\omit{\eta} &
      \ar@{=}[d]
      & \makebox[-1mm]{$\Big\rangle$}\makebox[0mm]{$\Big\rangle$}
      \\
      \makebox[-1mm]{$\Big\langle$}\makebox[0mm]{$\Big\langle$} &
       &
      \ar[r]_{B(\id,f)}  &
      \ar[r]^K \ar@{}[d]|{\cong} &
      \ar[r]^{B(f,\id)}
      && \makebox[-1mm]{$\Big\rangle$}\makebox[0mm]{$\Big\rangle$}
      \\
      \makebox[-1mm]{$\Big\langle$}\makebox[0mm]{$\Big\langle$} &
      \ar[r]^{B(f,\id)}\ar@{=}[d] \drtwocell\omit{\epsilon} &
      \ar[r]^{B(\id,f)}\ar[d]^f \drtwocell\omit{\beta} &
      \ar[r]^K \ar@{=}[d] & \ar@{=}[d]
      && \makebox[-1mm]{$\Big\rangle$}\makebox[0mm]{$\Big\rangle$}
      \\
      \makebox[-1mm]{$\Big\langle$}\makebox[0mm]{$\Big\langle$} &
      \ar@{=}[r] &
      \ar@{=}[r] &
      \ar[r]_K &
      && \makebox[-1mm]{$\Big\rangle$}\makebox[0mm]{$\Big\rangle$}
      }}
  \end{equation}
  Now the naturality of the shadow isomorphism on $\loops\lW$ implies that this is equal to
  \begin{equation}
  \vcenter{\xymatrix@C=3pc{
      \makebox[-1mm]{$\Big\langle$}\makebox[0mm]{$\Big\langle$} &
      \ar@{=}[r]\ar@{=}[d]\drtwocell\omit{\alpha} &
      \ar[r]^H\ar[d]_f \drtwocell\omit{\phi} &
      \ar@{=}[r]\ar@{=}[d]^f \drtwocell\omit{\eta} &
      \ar@{=}[d]
      & \makebox[-1mm]{$\Big\rangle$}\makebox[0mm]{$\Big\rangle$}
      \\
      \makebox[-1mm]{$\Big\langle$}\makebox[0mm]{$\Big\langle$} &
      \ar[r]_{B(\id,f)} \ar[d]^f \drtwocell\omit{\beta}  &
      \ar[r]^K \ar@{=}[d] &
      \ar[r]^{B(f,\id)} \ar@{=}[d] \drtwocell\omit{\epsilon}
      & \ar@{=}[d]
      & \makebox[-1mm]{$\Big\rangle$}\makebox[0mm]{$\Big\rangle$}
      \\
      \makebox[-1mm]{$\Big\langle$}\makebox[0mm]{$\Big\langle$} &
      \ar@{=}[r] &
      \ar[r]_K &\ar@{=}[r]
      && \makebox[-1mm]{$\Big\rangle$}\makebox[0mm]{$\Big\rangle$}
      }}
  \end{equation}
  followed by the shadow isomorphism $\sh{\lI_S\odot K} \cong \sh{K\odot \lI_S}$.
  However, the latter is the identity by one of the axioms of a shadow.
  Finally, two of the defining laws of $B(\id,f)$ and $B(f,\id)$ reduce this composite to simply $\sh{(f,\phi)}$.
\end{proof}

  \autoref{thm:general-omega} implies in particular that a shadow on a framed bicategory is uniquely determined by the underlying shadow on its horizontal bicategory.
  However, at present our interest is in applying it in the other direction: we will construct a shadow on a framed bicategory and use \autoref{thm:general-omega} to yield a more explicit characterization of traces for endomorphisms of representable proarrows.

\begin{rmk}
  There is also a sort of converse to \autoref{thm:general-omega}.
  If \W is any bicategory, assumed for simplicity to be a strict 2-category, then there is a framed bicategory \lW whose horizontal bicategory is \W, whose vertical arrows are right dualizable 1-cells in \W, and whose 2-cells
\begin{equation}
  \vcenter{\xymatrix{
      A\ar[r]|{|}^Q\ar[d]_{N} \drtwocell\omit{\phi} &
      B\ar[d]^{M}\\
      C\ar[r]|{|}_{P} &
      D
      }}
\end{equation}
  are 2-cells $\phi\colon Q\odot M \to N\odot P$ in \W.
  The converse of \autoref{thm:general-omega} says that any shadow on \W can be (uniquely) extended to a shadow on this \lW, where the shadow of $\phi$ above is the trace of the corresponding 2-cell in \W.
  The axioms required of this shadow are basically the properties of bicategorical traces from~\cite[\S7]{PS2}.
  In particular, its functoriality on 2-cells in $\loops\lW$ is precisely \autoref{thm:compose-traces}, the composition theorem for traces that plays such a major role in this paper.
\end{rmk}

\section{Base change objects for bicategories}\label{sec:basechangebicat}
In this section, we will  use \autoref{thm:general-omega} to identify the components of bicategorical traces in the non-homotopical examples (\autoref{thm:bicatomega}).
We start by describing base change objects in terms of the bicategorical structure used in \S\ref{sec:bicat}.

From the categorical fibration $(S,T)\colon \lW_1 \to \lW_0\times \lW_0$ of a framed bicategory \lW,
we obtain a pseudofunctor $(\lW_0\times \lW_0)\op \to \cCAT$ sending $(A,B)$ to the category $\W(A,B)$.
We write $f^*Mg^*$ for the action of a pair of vertical arrows $(f,g)\colon (A,B)\to (A',B')$ on $M\colon A\hto B$ under this pseudofunctor; it is called the \textbf{restriction} of $M$ along $f$ and $g$.

We can further extend this  
to a pseudofunctor whose domain is a 2-category.
The \textbf{vertical 2-category} $\cV\lW$ of a double category has underlying category $\lW_0$, and a 2-cell from $f\colon A\to B$ to $g\colon A\to B$ is a 2-cell
\begin{equation}\label{eq:v2cell}
  \vcenter{\xymatrix{
      A\ar[r]|{|}^{\lI_A}\ar[d]_g \drtwocell\omit{} &
      A\ar[d]^f\\
      B\ar[r]|{|}_{\lI_B} &
      B
      }}
\end{equation}
in \lW.
Pasting of 2-cells in \lW yields the required compositional structure making $\cV\lW$ a 2-category.

\begin{lem}\label{thm:2cat-restriction}
  The pseudofunctor $\W\colon (\lW_0\times \lW_0)\op \to \cCAT$ can be extended to a pseudofunctor $\cV\lW\op \times \cV\lW\coop \to \cCAT$.
\end{lem}

As always, $\cV\lW\op$ denotes reversal of 1-morphisms but not 2-morphisms, while $\cV\lW\coop$ denotes reversal of both.

\begin{proof}
  Suppose given 2-cells
  \begin{equation}
    \vcenter{\xymatrix{
        A\ar[r]|{|}^{\lI_A}\ar[d]_g \drtwocell\omit{\phi} &
        A\ar[d]^f\\
        B\ar[r]|{|}_{\lI_B} &
        B
      }}
    \qquad\text{and}\qquad
    \vcenter{\xymatrix{
        C\ar[r]|{|}^{\lI_C}\ar[d]_k \drtwocell\omit{\psi} &
        C\ar[d]^h\\
        D\ar[r]|{|}_{\lI_D} &
        D.
      }}
  \end{equation}
  Their image under the desired pseudofunctor must be a natural transformation whose component at $M\colon B\hto D$ is a map $f^*Mk^* \to g^*Mh^*$.
  The fibrational structure of $(S,T)$ yields cartesian 2-cells
  \begin{equation}
    \vcenter{\xymatrix@C=4pc{
        A\ar[r]|{|}^{f^*Mk^*}\ar[d]_f \drtwocell\omit{\quad\mathrm{cart}_1} &
        C\ar[d]^k\\
        B\ar[r]|{|}_M &
        D
      }}
    \qquad\text{and}\qquad
    \vcenter{\xymatrix@C=4pc{
        A\ar[r]|{|}^{g^*Mh^*}\ar[d]_g \drtwocell\omit{\quad\mathrm{cart}_2} &
        C\ar[d]^h\\
        B\ar[r]|{|}_M &
        D.
      }}
  \end{equation}
  Pasting ${\mathrm{cart}_1}$ of these with $\phi$ and $\psi$ we have
  \begin{equation}
  \vcenter{\xymatrix{
      A\ar[r]|{|}\ar[d]_g \drtwocell\omit{\phi} &
      A\ar[r]|{|}\ar[d]^f\drtwocell\omit{\mathrm{cart}_1} &
      C\ar[r]|{|}\ar[d]_k \drtwocell\omit{\psi}&
      C\ar[d]^h\\
      B\ar[r]|{|} &
      B\ar[r]|{|} &
      D\ar[r]|{|} &
      D
      }}
  \end{equation}
  whose unique factorization through ${\mathrm{cart}_2}$ yields the desired map $f^*Mk^* \to g^*Mh^*$.
  The verification of functoriality is straightforward.
\end{proof}

Given a 2-cell $\mu\colon f\to g$ in $\cV\lW$,
we have an induced map $B(\id,\mu)\colon B(\id,f) \to B(\id,g)$ in $\W(A,B)$, defined as the following composite:
\begin{equation}
  \vcenter{\xymatrix{
      \ar[r]|{|}^{\lI_A}\ar@{=}[d] \drtwocell\omit{\beta} &
      \ar[r]|{|}^{\lI_A}\ar[d]^g \drtwocell\omit{\mu} &
      \ar[r]|{|}^{B(\id,f)}\ar[d]^f \drtwocell\omit{\alpha}&
      \ar@{=}[d]\\
      \ar[r]|{|}_{B(\id,g)} &
      \ar[r]|{|}_{\lI_B} &
      \ar[r]|{|}_{\lI_B} &
      }}
\end{equation}
Similarly, we have an induced map $B(\mu,\id)\colon B(g,\id) \to B(f,\id)$.
These constructions yield pseudofunctors $\cV\lW \to \W$ and $\cV\lW\coop \to \W$; see~\cite[Appendix~C]{shulman:frbi}.

\begin{lem}\label{thm:bco-restriction}
  For $f$, $g$ and $M$ as above and $\phi$ and $\psi$ as in the proof of \autoref{thm:2cat-restriction}
  \[f^*Mg^* \cong A'(\id,f) \odot M \odot B'(g,\id)\]
  \label{thm:bco-2cat-restriction}
  and under this isomorphism  the map $f^*Mk^* \to g^*Mh^*$ is identified 
  with the map
  \[ B(\id,f) \odot M \odot D(k,\id) \xto{B(\id,\phi) \odot \id_M \odot D(\psi,\id)} B(\id,g) \odot M \odot D(h,\id). \]
\end{lem}

\begin{proof}
  The isomorphism follows from the proof of equivalence of the three definitions of framed bicategory in~\cite[Theorem~4.1]{shulman:frbi}.
  The identification of maps is  a straightforward calculation using the proof of~\cite[Theorem~4.1]{shulman:frbi}.
\end{proof}

We can now prove the Component Lemma for bicategories.

\begin{replem}{thm:bicatomega}[The component lemma for bicategories]\label{thm:bicatomegap}
  Let \W be a locally cocomplete bicategory with a cocontinuous shadow.
  For any right dualizable $X\colon (A,R) \hto (\tc,S)$ in $\prof(\W)$ with $f\colon X\to X\odot P$, and any morphism $\alpha\in A(a,a)$, the component $\tr(f)_{[\alpha]}$ is the trace in \W of the composite
  \begin{equation}
    \xymatrix{X(a) \ar[r]^-{X_{\alpha}} & X(a) \ar[r]^-{f_a} & X(a)\odot P.}\label{eq:trfalbi}
  \end{equation}
\end{replem}
\begin{proof}
  In outline, the proof is identical to that of \autoref{thm:smcomega2a}.
  With $a\colon (\tc,R)\to (A,R)$ the functor picking out $a\in A$, the representable profunctor $(A,R)(\id,a)$
  is right dualizable.
  Moreover, we have a 2-cell in $\lProf(\W)$:
  \begin{equation}
    \vcenter{\xymatrix{
        (\tc,R)\ar[r]^{\lI}\ar[d]_a \drtwocell\omit{\alpha} &
        (\tc,R)\ar[d]^a\\
        (A,R)\ar[r]_{\lI} &
        (A,R)
      }}
  \end{equation}
  whose unique component is the coprojection $\lI_R \to A(a,a) \cdot \lI_R$ indexed by $\alpha$.
  Thus, there is an induced endomorphism $(A,R)(\id,\alpha) \colon (A,R)(\id,a) \to (A,R)(\id,a)$, and by \autoref{thm:compose-traces}, the trace of the composite
  \begin{equation}\label{eq:trfalbi2}
    (A,R)(\id,a) \odot X \xto{(A,R)(\id,\alpha)\odot \id} (A,R)(\id,a) \odot X\odot P \xto{\id \odot f} (A,R)(\id,a) \odot X\odot P
  \end{equation}
  is equal to the composite
  \[ \sh{R} \xto{\tr((A,R)(\id,\alpha))} \sh{(A,R)} \xto{\tr(f)} \sh{P}. \]
  Now note that $(A,R)(\id,a) \odot X \cong a^*X \cong X(a)$.
  Moreover, \autoref{thm:bco-2cat-restriction} informs us that under this isomorphism,~\eqref{eq:trfalbi2} is identified with~\eqref{eq:trfalbi}.
  The proof is now completed by \autoref{thm:general-omega}, which identifies $\tr((A,R)(\id,\alpha))$ with $\sh{\alpha}$, where by construction the latter picks out the component of $\sh{A}$ indexed by $[\alpha]$.
\end{proof}

\begin{rmk}
  Since the framed bicategories of profunctors constructed in this section have general enriched categories as their objects, we can obtain versions of the linearity formulas from \S\ref{sec:moncat} and \S\ref{sec:bicat} that apply to the more general case of weighted colimits of \emph{enriched} diagrams.
  However, we do not have any interesting examples, so we leave the details to the reader.
\end{rmk}

\section{Base change objects for derivators}
\label{sec:bco-der}

We now apply the theory of \S\ref{sec:base_change}
to prove the component lemmas 
for derivators (\autoref{thm:smderomega} and \autoref{thm:derbicatomega}).
The structure is similar to \S\ref{sec:basechangebicat}, and we build on the results there. 

\begin{thm}\label{thm:der-frbi}
  Let \dW be a closed derivator bicategory with a shadow.
  Then the bicategory $\prof(\dW)$ and its shadow constructed in \S\ref{sec:derbicat} extend to a framed bicategory $\lProf(\dW)$ with a shadow, whose vertical arrows are of the form $(f,R)\colon (A,R) \to (B,R)$ for a functor $f\colon A\to B$.
\end{thm}

We could state this theorem more generally replacing \dW by a \emph{framed derivator bicategory} \lW, thereby allowing the vertical arrows in $\lProf(\lW)$ to incorporate a vertical arrow in $\lW$ as well as a functor in \cCat.
Moreover, most motivating examples of derivator bicategories are in fact framed, including the bicategory of rings and complexes of bimodules (its vertical arrows are ring homomorphisms) and the bicategory of parametrized spectra that we will use in~\cite{PS6} (its vertical arrows are maps of base spaces).
However, here we have no need for these more general vertical arrows, so we can avoid giving a definition of framed derivator bicategories.

\begin{proof}
  We define a 2-cell
  \begin{equation}
  \vcenter{\xymatrix{
      (A,R)\ar[r]|{|}^M \ar[d]_{(f,R)} \drtwocell\omit{\phi} &
      (B,S)\ar[d]^{(g,S)}\\
      (A',R)\ar[r]|{|}_N &
      (B',S)
      }}
  \end{equation}
  in $\lProf(\dW)$ to be a morphism $M \to (f\times g\op)^* N$ in $\dW(R,S)(A\times B\op)$.
  Vertical composites and identities are obvious.
  The horizontal identity 2-cell of $(f,R)\colon (A,R) \to (B,R)$ is the mate-transformation $\lI_{(A,R)} \to (f\times f\op)^* \lI_{(B,R)}$ induced by the following commutative square:
  \begin{equation}
    \vcenter{\xymatrix{
        \tw(A)\ar[r]^{\tw(f)}\ar[d]_{(t,s)} &
        \tw(B)\ar[r]\ar[d]^{(t,s)} &
        \tc\\
        A\times A\op\ar[r]_{f\times f\op} &
        B\times B\op      
      }}
  \end{equation}
  The horizontal composite of 
  \begin{equation}
  \vcenter{\xymatrix{
      (A,R)\ar[r]|{|}^M \ar[d]_{(f,R)} \drtwocell\omit{\phi} &
      (B,S)\ar[d]^{(g,S)} \ar[r]|{|}^P \drtwocell\omit{\psi} &
      (C,T) \ar[d]^{(h,T)} \\
      (A',R)\ar[r]|{|}_N &
      (B',S)\ar[r]|{|}_Q &
      (C',T)
      }}
  \end{equation}
  is the composite
  \begin{align*}
    (\pi_{\tw(B)\op})_! (t\op,s\op)^* (M\odot P)
    &\xto{\phi\odot \psi} (\pi_{\tw(B)\op})_! (t\op,s\op)^* ((f\times g\op)^*N \odot (g\times h\op)^*Q)\\
    &\cong (\pi_{\tw(B)\op})_! (t\op,s\op)^* (f\times g\op \times g\times h\op)^* (N\odot Q)\\
    &\cong (\pi_{\tw(B)\op})_! (f\times \tw(g)\op \times h\op)^* (t\op,s\op)^* (N\odot Q)\\
    &\too (f\times h\op)^* (\pi_{\tw(B')\op})_! (t\op,s\op)^* (N\odot Q).
  \end{align*}
  Here $\odot$ denotes the \emph{external} version of the two-variable derivator morphism
  \[\odot \colon \dW(R,S) \times W(S,T) \to \dW(R,T),\]
  and the final map in the composite is the mate-transformation induced by the following commutative square:
  \begin{equation}
  \vcenter{\xymatrix@C=6pc{
      A\times \tw(B)\op \times C\op\ar[r]^-{f\times \tw(g)\op\times h\op} \ar[d]_{\pi_{\tw(B)\op}} &
      A'\times \tw(B')\op \times (C')\op\ar[d]^{\pi_{\tw(B')\op}}\\
      A\times C\op\ar[r]_-{f\times h\op} &
      A'\times (C')\op
      }}
  \end{equation}
  Associativity and unitality of the horizontal composition of 2-cells are automatic because the construction of the associativity and unit isomorphisms for $\prof(\dW)$ in~\cite{gps:additivity} use homotopy exact squares that are natural with respect to the small categories appearing therein, and mates are functorial under pasting of squares.
  Thus, we have a double category $\lProf(\dW)$.
  Moreover, the definition of the 2-cells implies immediately that it has restrictions, hence is a framed bicategory.

  Similarly, the shadow of a 2-cell $\phi\colon M\to (f\times f\op)^*N$, for $M \colon  (A,R) \hto (A,R)$ and $N\colon (B,S) \hto (B,S)$ with $f\colon A\to B$, is the map
  \begin{align*}
    (\pi_{\tw(A)\op})_! (t\op,s\op)^* M
    &\xto{\phi} (\pi_{\tw(A)\op})_! (t\op,s\op)^* (f\times f\op)^*N\\
    &\cong (\pi_{\tw(A)\op})_! (\tw(f)\op)^* (t\op,s\op)^* N\\
    &\too(\pi_{\tw(B)\op})_!(t\op,s\op)^* N
  \end{align*}
  where the final morphism is induced by the commutative square
  \begin{equation}
  \vcenter{\xymatrix@C=3pc{
      \tw(A)\op \ar[r]^{\tw(f)\op}\ar[d] &
      \tw(B)\op\ar[d]\\
      \tc\ar@{=}[r] &
      \tc.
      }}
  \end{equation}
  Finally, the naturality of the shadow isomorphism with respect to 2-cells also follows from the naturality with respect to small categories of the homotopy exact squares appearing in its definition.
\end{proof}

The proof of \autoref{thm:der-frbi} implies that for $M\in \dW(R,S)(B,D)$ and $f\colon A\to B$ and $g\colon C\to D$, the restriction $(f,R)^* M (g,S)^*$ in the framed-bicategory sense may be identified with the restriction $(f\times g\op)^*M \in \dW(R,S)(A,C)$ in the derivator sense.
Thus, by \autoref{thm:bco-restriction}, we have
\begin{equation}
  (f\times g\op)^*M \cong (B,R)(\id,f) \odot M \odot (D,S)(g,\id)\label{eq:derbicat-bcorestr}
\end{equation}

Now, any natural transformation $\mu\colon f\to g \colon  A \to B$ induces a 2-cell
\begin{equation}
  \vcenter{\xymatrix@C=3pc{
      (A,R)\ar[r]|{|}^\lI\ar[d]_{(g,R)} \drtwocell\omit{\mathrlap{(\mu,R)}} &
      (A,R)\ar[d]^{(f,R)}\\
      (B,R)\ar[r]|{|}_\lI &
      (B,R)
    }}
\end{equation}
as the following composite:
\begin{align*}
  (t,s)_! \pi_{\tw(A)}^* \lI_R
  &\cong (t,s)_! \tilde{\mu}^* \pi_{\tw(B)}^* \lI_R\\
  &\too (g\times f\op)^* (t,s)_! \pi_{\tw(B)}^* \lI_R.
\end{align*}
Here $\tilde{\mu}\colon \tw(A) \to \tw(B)$ is the functor sending a morphism $\alpha\colon a\to a'$ of $A$ (regarded as an object of $\tw(A)$) to the morphism $\mu_{a'} \circ f(\alpha) = g(\alpha) \circ \mu_{a}$ of $B$, and the final morphism above is induced by the following commutative square:
\begin{equation}
  \vcenter{\xymatrix{
      \tw(A)\ar[r]^{\tilde{\mu}}\ar[d]_{(t,s)} &
      \tw(B)\ar[d]^{(t,s)}\\
      A\times A\op\ar[r]_{g\times f\op} &
      B\times B\op.
      }}
\end{equation}
In fact, for any $R$, this construction defines a 2-functor $(-,R)\colon \cCat \to \cV(\lProf(\dW))$.
However, we will not need its functoriality, so we leave the proof to the reader.
What we do need is that the induced action of $\mu$ on horizontal arrows of $\lProf(\dW)$ agrees with that arising from the derivator structure of \dW.
For simplicity, we state and prove only the one-sided version.

\begin{lem}\label{thm:der-bco-2cat-restr}
  For $\mu\colon f\to g \colon  A \to B$ and $M\in \dW(R,S)(B,C)$, the derivator restriction map
  \[(\mu\times \id)^*M \colon  (f\times \id)^*M\to (g\times \id)^* M\]
  is equal to the framed-bicategory restriction map
  \[(\mu,R)^* M \colon  (f,R)^*M\to (g,R)^*M\]
  defined as in \autoref{thm:2cat-restriction}.
  Therefore, by \autoref{thm:bco-2cat-restriction}, the isomorphism~\eqref{eq:derbicat-bcorestr} identifies $\mu^*M$ with $(B,R)(\id,\mu) \odot \id_M$.
\end{lem}
\begin{proof}
  By definition, $(\mu,R)^* M$ is obtained by factoring the horizontal pasting
  \begin{equation}
    \vcenter{\xymatrix@C=3pc{
        (A,R)\ar[r]|{|}^{\lI}\ar[d]_{(g,R)} \drtwocell\omit{\mathrlap{(\mu,R)}} &
        (A,R)\ar[r]|{|}^{f^*M}\ar[d]|{(f,R)} \drtwocell\omit{\mathrlap{\mathrm{cart}}} &
        (C,S)\ar@{=}[d]\\
        (B,R)\ar[r]|{|}_{\lI} &
        (B,R)\ar[r]|{|}_M &
        (C,S)
      }}
  \end{equation}
  through the cartesian 2-cell defining $(g,R)^*M$.
  However, in the case of $\lProf(\dW)$, these cartesian 2-cells are simply identities $(f\times \id)^*M = (f\times \id)^*M$ and $(g\times \id)^*M=(g\times \id)^*M$.
  Thus, by definition of horizontal composition of 2-cells in $\lProf(\dW)$ and of $(\mu,R)$, we must show that the composite
  \begin{align*}
    (f\times \id)^*M
    &\cong (\pi_{\tw(A)\op})_! (t\op,s\op)^* (\lI_{(A,R)}\odot (f\times \id)^*M)\\
    &= (\pi_{\tw(A)\op})_! (t\op,s\op)^* ((t,s)_! \pi_{\tw(A)}^* \lI_R \odot (f\times \id)^*M)\\
    &\cong (\pi_{\tw(A)\op})_! (t\op,s\op)^* ((t,s)_! \tilde{\mu}^* \pi_{\tw(B)}^* \lI_R \odot (f\times \id)^*M)\\
    &\too (\pi_{\tw(A)\op})_! (t\op,s\op)^* ((g\times f\op)^* (t,s)_! \pi_{\tw(B)}^* \lI_R \odot (f\times \id)^*M)\\
    &= (\pi_{\tw(A)\op})_! (t\op,s\op)^* ((g\times f\op)^* \lI_{(B,R)} \odot (f\times \id)^*M)\\
    &\cong (\pi_{\tw(A)\op})_! (t\op,s\op)^* (g\times f\op \times f\times \id)^* (\lI_{(B,R)}\odot M)\\
    &\cong (\pi_{\tw(A)\op})_! (g\times \tw(f)\op \times \id)^* (t\op,s\op)^* (\lI_{(B,R)}\odot M)\\
    &\too (g\times \id)^* (\pi_{\tw(B)\op})_! (t\op,s\op)^* (\lI_{(B,R)}\odot M)\\
    &\cong (g\times \id)^* M.
  \end{align*}
  is equal to $(\mu\times \id)^*M$.
  Omitting the first and last isomorphisms (which are just unit isomorphisms in $\prof(\dW)$), we may see this as a sequence of natural transformations applied to $\lI_R \odot M$.
  From the definitions of the two noninvertible factors, it is the composite mate-transformation relating the two outer paths from the top-right to the bottom-left in the following diagram of functors (restricting to the left and left Kan extending downwards).
  (For brevity, we omit the factor $C\op$ from the notation from now on, as it plays no role in the calculation.)
  \begin{equation}\label{eq:der-bco-2cat-restr-1}
    \vcenter{\xymatrix@C=3pc{
        &
        \tw(A)\times A \ar[r]^{\tilde{\mu}\times f}\ar[d] &
        \tw(B)\times B\ar[r]\ar[d] &
        B \\
        A\times \tw(A)\op \ar[r]\ar@{=}[d] &
        A\times A\op\times A \ar[r]^{g\times f\op\times f} &
        B\times B\op\times B \ar@{=}[d]\\
        A\times \tw(A)\op \ar[r]^{g\times\tw(f)}\ar[d] &
        B\times\tw(B)\op \ar[r]\ar[d] &
        B\times B\op\times B \\
        A \ar[r]^{g} &
        B 
      }}
  \end{equation}
  Now the proof of the unit isomorphisms of $\prof(\dV)$ in~\cite{gps:additivity} has two steps, involving the following two homotopy exact squares: one pullback
  \begin{equation}
    \vcenter{\xymatrix@C=3pc{
        \tw(B)\times_{B\op} \tw(B)\op\ar[r]\ar[d]&
        \tw(B)\times B \ar[d]\\
        B\times \tw(B)\op\ar[r]&
        B\times B\op\times B
      }}\label{eq:profunit1}
  \end{equation}
  and one transformation that composes up pairs of morphisms:
  \begin{equation}
    \vcenter{\xymatrix{
        \tw(B)\times_{B\op} \tw(B)\op \ar[r]\ar[d] \drtwocell\omit &
        B\ar@{=}[d]\\
        B\ar@{=}[r] &
        B.
      }}\label{eq:profunit2}
  \end{equation}
  (The $(-)\op$'s are switched around from the proof in~\cite{gps:additivity} since here we are considering the \emph{left} unit isomorphism.)
  Applying~\eqref{eq:profunit1} to both $A$ and $B$ in our case, and using the universal property of pullbacks and the functoriality of mates,~\eqref{eq:der-bco-2cat-restr-1} becomes
  \begin{equation}
    \vcenter{\xymatrix@C=1.5pc{
        \tw(A)\times_{A\op} \tw(A)\op \ar[rr]^{\tilde{\mu}\times_f \tw(f)} \ar[d]&&
        \tw(B)\times_{B\op} \tw(B)\op \ar[r]\ar[d] &
        \tw(B)\times B \ar[r]\ar[d] &
        B \\
        A\times \tw(A)\op \ar[rr]^{g\times\tw(f)}\ar[d] &&
        B\times\tw(B)\op \ar[r]\ar[d] &
        B\times B\op \times B \\
        A \ar[rr]^{g} &&
        B 
      }}
  \end{equation}
  Now applying~\eqref{eq:profunit2} for $B$, we obtain
  \begin{equation}
  \vcenter{\xymatrix{
      \tw(A)\times_{A\op} \tw(A)\op \ar[r]^-{f t\op}\ar[d]_s \drtwocell\omit{\overline{\mu}} &
      B \ar@{=}[d]\\
      A \ar[r]_{g} &
      B .
      }}
  \end{equation}
  However, this square factors into
  \begin{equation}
  \vcenter{\xymatrix{
      \tw(A)\times_{A\op} \tw(A)\op \ar[r]^-{t\op}\ar[d]_s \drtwocell\omit{} &
      A  \ar[r]^{f} \ar@{=}[d] \drtwocell\omit{\mu} &
      B \ar@{=}[d]\\
      A \ar@{=}[r] &
      A  \ar[r]_{g} &
      B .
      }}
  \end{equation}
  in which the left-hand square is~\eqref{eq:profunit2} for $A$.
  Thus, after passing fully across both unit isomorphisms, we obtain simply $\mu^*$.
\end{proof}

We will also need to know about the shadows of such natural transformations.
As we observed in eq.~\eqref{eq:derLA}, the argument for~\eqref{eq:LA} applies essentially verbatim to conclude
\[ \sh{(A,R)} \cong |N(\Lambda A)| \tens \sh{R}. \]

\begin{lem}\label{thm:derbi-2cell-sh}
  For a natural transformation $\mu\colon f\to f\colon A\to B$ and any $R$, the shadow $\sh{(\mu,R)}\colon \sh{(A,R)} \to \sh{(B,R)}$ in $\lProf(\dW)$ is the map induced by the functor $\Lambda \mu \colon  \Lambda A \to \Lambda A$ from \S\ref{sec:der}.
\end{lem}
\begin{proof}
  Putting together the definitions of $(\mu,R)$ and the shadow of $\lProf(\dW)$, we see that $\sh{(\mu,R)}$ is the composite mate-transformation comparing the two extreme paths from the top right to the bottom left below (restricting to the left and left Kan extending downwards):
  \begin{equation}\label{eq:derbi-2cell-sh-1}
    \vcenter{\xymatrix@C=3pc{
      & \tw(A) \ar[r]^{\tilde\mu}\ar[d] &
      \tw(B) \ar[r] \ar[d] &
      \tc\\
      \tw(A)\op \ar[r] \ar@{=}[d] &
      A\times A\op \ar[r]^{f\times f\op} &
      B\times B\op \ar@{=}[d]\\
      \tw(A)\op \ar[r]^{\tw(f)\op} \ar[d] &
      \tw(B)\op \ar[r] \ar[d] &
      B\times B\op\\
      \tc\ar@{=}[r] &
      \tc
    }}
  \end{equation}
  In \S\ref{sec:der} we introduced $\Lambda A$ as a pullback
  \begin{equation}
    \vcenter{\xymatrix@C=3pc{
        \Lambda A \ar[r] \ar[d] \pullbackcorner & \tw(A) \ar[d]^{(t,s)} \\
        \tw(A)\op \ar[r]_{(s\op,t\op)} & A\times A\op.
      }}
  \end{equation}
  Thus, applying this to both $A$ and $B$ in~\eqref{eq:derbi-2cell-sh-1}, and using the universal property of the pullback, we obtain
  \begin{equation}
    \vcenter{\xymatrix@C=3pc{
        \Lambda A \ar[r]^{\Lambda \mu} \ar[d] &
        \Lambda B \ar[r] \ar[d] &
        \tw(B) \ar[r] \ar[d] &
        \tc \\
        \tw(A)\op \ar[r]^{\tw(f)\op} \ar[d] &
        \tw(B)\op \ar[r] \ar[d] &
        B\times B\op\\
        \tc \ar@{=}[r] &
        \tc
      }}
  \end{equation}
  It is easy to see that the induced map is indeed $\Lambda \mu$, as shown.
\end{proof}

We can now prove a generalization of  \autoref{thm:derbicatomega} (which includes \autoref{thm:smderomega} as a special case).
Recall that any $\alpha\colon a\to a$ in $A$ induces a map $[\alpha] \colon  \sh{R} \to \sh{(A,R)}$ by way of the functor $\Lambda \alpha \colon \tc \to \Lambda A$.

\begin{replem}{thm:derbicatomega}[The component lemma for derivator bicategories]\label{thm:compderivbicat} 
  If $X\in\dW(R,S)(A)$ is pointwise dualizable and $f\colon X\to X\odot P$, then for any conjugacy class $[a\xto{\alpha} a]$ in $A$, the composite
  \begin{equation}
    \xymatrix{ \sh{R} \ar[r]^-{[\alpha]} & \sh{(A,R)} \ar[r]^-{\tr(f)} & \sh{P} }
  \end{equation}
  is equal to the trace in $\dW(R,S)(\tc)$ of the composite
  \begin{equation}\label{eq:trfalderbi}
    \xymatrix{ X_a \ar[r]^-{X_\alpha} & X_a \ar[r]^-{f_a} & X_a\odot P_a .}
  \end{equation}
\end{replem}
\begin{proof}
  As before, let $a\colon \tc\to A$ be the functor picking out $a\in A$.
  Then the representable profunctor $(A,R)(\id,a)\colon (\tc,R) \hto (A,R)$ is right dualizable, and we have the 2-cell
  \begin{equation}
    \vcenter{\xymatrix{
        (\tc,R)\ar[r]^{\lI}\ar[d]_{(a,R)} \drtwocell\omit{\mathrlap{(\alpha,R)}} &
        (\tc,R)\ar[d]^{(a,R)}\\
        (A,R)\ar[r]_{\lI} &
        (A,R)
      }}
  \end{equation}
  defined as above, with induced endomorphism $(A,R)(\id,\alpha) \colon (A,R)(\id,a) \to (A,R)(\id,a)$.
  Thus, by \autoref{thm:compose-traces}, the trace of the composite
  \begin{equation}\label{eq:trfalderbi2}
    (A,R)(\id,a) \odot X \xto{(A,R)(\id,\alpha)\odot \id} (A,R)(\id,a) \odot X \xto{\id \odot f} (A,R)(\id,a) \odot X\odot P
  \end{equation}
  is equal to the composite
  \[ \sh{R} \xto{\tr((A,R)(\id,\alpha))} \sh{(A,R)} \xto{\tr(f)} \sh{P}. \]
  But by the proof in \autoref{thm:der-frbi} that $\lProf(\dW)$ is framed, we have
  \[(A,R)(\id,a) \odot (X\odot P) \cong a^*(X\odot P) \cong (X\odot P)_a,\] 
  while \autoref{thm:der-bco-2cat-restr} and \autoref{thm:bco-2cat-restriction} inform us that under this isomorphism,~\eqref{eq:trfalderbi2} is identified with~\eqref{eq:trfalderbi}.
  Finally, \autoref{thm:general-omega} identifies $\tr((A,R)(\id,\alpha))$ with $\sh{(\alpha,R)}$, while \autoref{thm:derbi-2cell-sh} identifies this with the map $[\alpha]$.
\end{proof}

Finally, we can  establish the uniqueness of our linearity formula using \autoref{thm:derbi-2cell-sh}.

\begin{proof}[Proof of \autoref{thm:uniqueness}]\label{proof:uniqueness}
  For any $\dW\in \fB$ and any object $R$ of $\dW$, let $(A,R)(\id,-)$ denote the unit object $\lI_{(A,R)} \in \dprof(\dW)((A,R),(A,R)) = \dW(R,R)(A\op\times A)$ regarded as an object of $\dprof(\dW)((\tc,R),(A,R))(A)$.
  Then by \autoref{thm:bicat-colim-is-colim}, we have
  \[ \colim((A,R)(\id,-)) \cong
  (\pi_{A\op})^* \lI_R \otimes_{[A]} \lI_{(A,R)} \cong
  (\pi_{A\op})^* \lI_R. \]
  In other words, the colimit of $(A,R)(\id,-)$ is the constant weight $(\pi_{A\op})^* \lI_R$, regarded as a morphism from $(\tc,R)$ to $(A,R)$ in $\dprof(\dW)$.
  Thus, by~\ref{item:uniqh1} applied in $\dprof(\dW)$, to show that $(\pi_{A\op})^* \lI_R$ is absolute, it will suffice to show that $(A,R)(\id,-)$ is pointwise right dualizable.
  But its value at $a\in A$ is just the representable profunctor $(A,R)(\id,a)$, which as we have seen is always right dualizable.

  Similarly, applying~\ref{item:uniqh2} to the identity map of $(A,R)(\id,-)$, we find that the coefficient vector of the constant weight $(\pi_{A\op})^* \lI_R$ is the sum
  \[ \sum_{[\alpha]} \phi_{[\alpha]} \tr((A,R)(\id,\alpha)) \]
  where for $\alpha\in A(a,a)$, the map $(A,R)(\id,\alpha)$ is the induced endomorphism of $(A,R)(\id,a)$.
  However, by \autoref{thm:general-omega} and \autoref{thm:derbi-2cell-sh}, $\tr((A,R)(\id,\alpha))$ is just the map $\sh{R} \to \sh{(A,R)}$ induced by $[\alpha]$, so to say that the coefficient vector is the above sum is exactly to say that the $\phi_{[\alpha]}$ are its components.
\end{proof}

\bibliographystyle{alpha}
\bibliography{additivity}

\begin{thebibliography}{EKMM97}

\bibitem[AK88]{ak:closure}
M.~H. Albert and G.~M. Kelly.
\newblock The closure of a class of colimits.
\newblock {\em J. Pure Appl. Algebra}, 51(1-2):1--17, 1988.

\bibitem[Bro82]{brown:cog}
Kenneth~S. Brown.
\newblock {\em Cohomology of groups}, volume~87 of {\em Graduate Texts in
  Mathematics}.
\newblock Springer-Verlag, New York, 1982.

\bibitem[Cis03]{cisinski:idcm}
Denis-Charles Cisinski.
\newblock Images directes cohomologiques dans les cat\'egories de mod\`eles.
\newblock {\em Ann. Math. Blaise Pascal}, 10(2):195--244, 2003.

\bibitem[Cis06]{cisinski:presheaves}
Denis-Charles Cisinski.
\newblock Les pr\'efaisceaux comme mod\`eles type d'homotopie.
\newblock {\em Ast\'erisque}, (308), 2006.

\bibitem[CKW87]{ckw:axiom-mod}
Aurelio Carboni, Stefano Kasangian, and Robert Walters.
\newblock An axiomatics for bicategories of modules.
\newblock {\em J. Pure Appl. Algebra}, 45(2):127--141, 1987.

\bibitem[DP80]{dp:duality}
Albrecht Dold and Dieter Puppe.
\newblock Duality, trace, and transfer.
\newblock In {\em Proceedings of the International Conference on Geometric
  Topology (Warsaw, 1978)}, pages 81--102, Warsaw, 1980. PWN.

\bibitem[dS14]{souza:traces}
Martin Gallauer~Alves de~Souza.
\newblock Traces in monoidal derivators, and homotopy colimits.
\newblock {\em Advances in Mathematics}, 261:26--84, 2014.
\newblock arXiv:1303.0153.

\bibitem[EKMM97]{ekmm}
A.~D. Elmendorf, I.~Kriz, M.~A. Mandell, and J.~P. May.
\newblock {\em Rings, Modules, and Algebras in Stable Homotopy Theory},
  volume~47 of {\em Mathematical Surveys and Monographs}.
\newblock American Mathematical Society, 1997.
\newblock With an appendix by M. Cole.

\bibitem[Fra]{franke:triang}
Jens Franke.
\newblock Uniqueness theorems for certain triangulated categories with an
  {A}dams spectral sequence.
\newblock Available at \url{http://www.math.uiuc.edu/K-theory/0139/}.

\bibitem[GGT14]{mo:chains}
Ira Gessel, Darij Grinberg, and Martin Tancer.
\newblock Counting chains of inclusions.
\newblock
  \url{http://mathoverflow.net/questions/168641/counting-chains-of-inclusions},
  2014.

\bibitem[GP99]{gp:double-limits}
Marco Grandis and Robert Pare.
\newblock Limits in double categories.
\newblock {\em Cahiers Topologie G\'eom. Diff\'erentielle Cat\'eg.},
  XL(3):162--220, 1999.

\bibitem[GPS13]{gps:additivity}
Moritz Groth, Kate Ponto, and Michael Shulman.
\newblock The additivity of traces in monoidal derivators.
\newblock To appear in the \emph{Journal of K-theory}. arXiv:1212.3277, 2013.

\bibitem[GPS14]{gps:stable}
Moritz Groth, Kate Ponto, and Michael Shulman.
\newblock {M}ayer-{V}ietoris sequences in stable derivators.
\newblock {\em Homotopy, Homology and Applications}, 16(1):256--294, 2014.
\newblock arXiv:1306.2072.

\bibitem[Gro90]{grothendieck:derivateurs}
Alexandre Grothendieck.
\newblock Les d\'erivateurs.
\newblock
  \url{http://people.math.jussieu.fr/~maltsin/groth/Derivateursengl.html},
  1990.

\bibitem[Gro13]{groth:ptstab}
Moritz Groth.
\newblock Derivators, pointed derivators and stable derivators.
\newblock {\em Algebr. Geom. Topol.}, 13(1):313--374, 2013.

\bibitem[Hel88]{heller:htpythys}
Alex Heller.
\newblock Homotopy theories.
\newblock {\em Memoirs of the American Mathematical Society}, 71(383):vi+78,
  1988.

\bibitem[Hov99]{hovey:modelcats}
Mark Hovey.
\newblock {\em Model Categories}, volume~63 of {\em Mathematical Surveys and
  Monographs}.
\newblock American Mathematical Society, 1999.

\bibitem[Kel82]{kelly:enriched}
G.~M. Kelly.
\newblock {\em Basic concepts of enriched category theory}, volume~64 of {\em
  London Mathematical Society Lecture Note Series}.
\newblock Cambridge University Press, 1982.
\newblock Also available online in \textit{Reprints in Theory and Applications
  of Categories}, No. 10 (2005) pp. 1-136.

\bibitem[Law74]{lawvere:metric-spaces}
F.~William Lawvere.
\newblock Metric spaces, generalized logic, and closed categories.
\newblock {\em Rend. Sem. Mat. Fis. Milano}, 43:135--166, 1974.
\newblock Reprinted as Repr. Theory Appl. Categ. 1:1--37, 2002.

\bibitem[Lei08]{leinster:eccat}
Tom Leinster.
\newblock The {E}uler characteristic of a category.
\newblock {\em Documenta Mathematica}, 13:21--49, 2008.

\bibitem[LMSM86]{lms:equivariant}
L.~G. Lewis, Jr., J.~P. May, M.~Steinberger, and J.~E. McClure.
\newblock {\em Equivariant stable homotopy theory}, volume 1213 of {\em Lecture
  Notes in Mathematics}.
\newblock Springer-Verlag, Berlin, 1986.
\newblock With contributions by J. E. McClure.

\bibitem[Lur09]{lurie:higher-topoi}
Jacob Lurie.
\newblock {\em Higher topos theory}.
\newblock Number 170 in Annals of Mathematics Studies. Princeton University
  Press, 2009.

\bibitem[Mal]{m:introder}
Georges Maltsiniotis.
\newblock Introduction \`a la th\'eorie des d\'erivateurs.
\newblock \url{http://www.math.jussieu.fr/~maltsin/ps/m.ps}.

\bibitem[May01]{add}
J.~P. May.
\newblock The additivity of traces in triangulated categories.
\newblock {\em Adv. Math.}, 163(1):34--73, 2001.

\bibitem[MS06]{maysig:pht}
J.~P. May and J.~Sigurdsson.
\newblock {\em Parametrized homotopy theory}, volume 132 of {\em Mathematical
  Surveys and Monographs}.
\newblock American Mathematical Society, Providence, RI, 2006.

\bibitem[Par90]{pare:sclim}
Robert Par\'e.
\newblock Simply connected limits.
\newblock {\em Canad. J. Math.}, 42:731--746, 1990.

\bibitem[Pon10]{kate:traces}
Kate Ponto.
\newblock Fixed point theory and trace for bicategories.
\newblock {\em Ast\'erisque}, (333):xii+102, 2010.

\bibitem[Pon11]{kate:relative}
Kate Ponto.
\newblock Relative fixed point theory.
\newblock {\em Algebr. Geom. Topol.}, 11(2):839--886, 2011.

\bibitem[Pon12]{kate:higher}
Kate Ponto.
\newblock Coincidence invariants and higher {R}eidemeister traces.
\newblock arXiv:1209.3710, 2012.

\bibitem[PS12]{PS3}
Kate Ponto and Michael Shulman.
\newblock Duality and traces for indexed monoidal categories.
\newblock {\em Theory and Applications of Categories}, 26(23):582--659
  (electronic), 2012.

\bibitem[PS13]{PS2}
Kate Ponto and Michael Shulman.
\newblock Shadows and traces for bicategories.
\newblock {\em Journal of Homotopy and Related Structures}, 8(2):151--200,
  2013.
\newblock arXiv:0910.1306.

\bibitem[PS14a]{PS6}
Kate Ponto and Michael Shulman.
\newblock The linearity of fixed point invariants.
\newblock In preparation, 2014.

\bibitem[PS14b]{PS4}
Kate Ponto and Michael Shulman.
\newblock The multiplicativity of fixed point invariants.
\newblock {\em Algebr. Geom. Topol.}, 14(3):1275--1306, 2014.
\newblock arXiv:1203.0950.

\bibitem[RV14]{rv:reedy}
Emily Riehl and Dominic Verity.
\newblock The theory and practice of reedy categories.
\newblock arXiv:1304.6871, 2014.

\bibitem[Shu06]{shulman:htpylim}
Michael Shulman.
\newblock Homotopy limits and colimits and enriched homotopy theory.
\newblock arXiv:math.CT/0610194, 2006.

\bibitem[Shu08]{shulman:frbi}
Michael Shulman.
\newblock Framed bicategories and monoidal fibrations.
\newblock {\em Theory Appl. Categ.}, 20(18):650--738 (electronic), 2008.
\newblock arXiv:0706.1286.

\bibitem[Str81]{street:cauchy-enr}
Ross Street.
\newblock Cauchy characterization of enriched categories.
\newblock {\em Rend. Sem. Mat. Fis. Milano}, 51:217--233 (1983), 1981.
\newblock Reprinted as \textit{Repr. Theory Appl. Categ.} 4:1--16, 2004.

\bibitem[Str83]{street:absolute}
Ross Street.
\newblock Absolute colimits in enriched categories.
\newblock {\em Cahiers Topologie G\'eom. Diff\'erentielle}, 24(4):377--379,
  1983.

\bibitem[Wal81]{walters:sheaves-cauchy-1}
R.F.C. Walters.
\newblock Sheaves and {C}auchy-complete categories.
\newblock {\em Cahiers Topologie G\'eom. Diff\'erentielle}, 22(3):283--286,
  1981.
\newblock Third Colloquium on Categories, Part IV (Amiens, 1980).

\bibitem[Woo82]{wood:proarrows-i}
R.~J. Wood.
\newblock Abstract proarrows. {I}.
\newblock {\em Cahiers Topologie G\'eom. Diff\'erentielle}, 23(3):279--290,
  1982.

\end{thebibliography}

\end{document}